\documentclass[11pt,a4paper,reqno]{amsart} 
  
\usepackage{amsmath,amsfonts,amsthm,amsopn,amssymb,cite,mathrsfs}
\usepackage{cancel,marginnote}

\usepackage{color,enumitem,graphicx}
\usepackage[colorlinks=true,urlcolor=blue,
citecolor=red,linkcolor=blue,linktocpage,pdfpagelabels]{hyperref}
\usepackage[english]{babel}

\usepackage[left=2.8cm,right=2.8cm,top=2.8cm,bottom=2.8cm]{geometry}

\makeindex

\numberwithin{equation}{section}
\newtheorem{theorem}{Theorem}[section]
\newtheorem{lemma}{Lemma}[section]

\newtheorem{remark}{Remark}[section]

\newtheorem{definition}{Definition}[section]

\usepackage[hyperpageref]{backref}

\title[A class of degenerate elliptic problems involving the Grushin operator]
{A class of degenerate elliptic problems involving the Grushin operator: critical mass and ground states}


\begin{document}
\maketitle

\centerline{\scshape Aleks Jevnikar}
\medskip
{\footnotesize
\centerline{Department of Mathematics, Computer Science and Physics, University of Udine}
\centerline{Via del le Scienze
206, Udine, 33100, Italy}
\centerline{aleks.jevnikar@uniud.it}}

\centerline{\scshape Senli Liu}
\medskip
{\footnotesize
\centerline{School of Mathematics and Statistics, Hunan University of Science and Technology}
\centerline{Xiangtan, Hunan 411201, China}
\centerline{liusl@hnust.edu.cn}}

\centerline{\scshape Yu Su}
\medskip
{\footnotesize
\centerline{School of Mathematics and Big Data, Anhui University of Science and Technology}
\centerline{Huainan, Anhui 232001, China}
\centerline{yusumath@aust.edu.cn}
}

\begin{abstract}
We study the following degenerate elliptic equation with combined nonlinearities and prescribed mass
\begin{equation*}
\begin{aligned}
\begin{cases}
-\Delta_{\gamma}u+\omega u+|u|^{p-2}u=|u|^{q-2}u
\ \
\mbox{in }\mathbb{R}^{N},
\\
\int_{\mathbb{R}^{N}}
|u|^{2}
\mathrm{d}z=c,
\end{cases}
\end{aligned}
\end{equation*}
where $\Delta_{\gamma}$ is the Grushin operator and $\omega\in\mathbb{R}$ appears as an unknown Lagrange multiplier. We characterize the critical mass threshold and the energy ground state solutions. Firstly, we prove the existence of optimizers for the Gagliardo-Nirenberg inequality involving the Grushin operator together with estimates for the related best constant. 
Secondly, we characterize the critical mass by the best constant of Gagliardo-Nirenberg inequality
and prove the existence of $x$-radial energy ground state solutions.
Finally, we describe the relation between the $x$-radial action and energy ground state solutions, and the optimizers of Gagliardo-Nirenberg inequality.
The strategy developed here is based on the interplay between the mass-preserving scaling and a sharp Gagliardo-Nirenberg inequality. It therefore suggests possible extensions to other prescribed-mass problems with competing homogeneous terms and threshold-mass phenomena.
\end{abstract}

\section{Introduction and main results}

In this paper, we consider the following equation with combined nonlinearities
\begin{equation}\label{G}
\begin{aligned}
-\Delta_{\gamma}u+\omega u+\alpha|u|^{p-2}u=|u|^{q-2}u
\ \
\mbox{ain }\mathbb{R}^{N},
\end{aligned}
\tag{$\mathcal{G}_{\gamma}$}
\end{equation}
where $\omega,\alpha\in\mathbb{R}$, $\gamma\geqslant0$ and $z=(x,y)\in \mathbb{R}^{k}\times \mathbb{R}^{l}=\mathbb{R}^N$ with $k,l\geqslant1$ and $N=k+l$. The exponents $p$ and $q$ satisfy 
\begin{equation} \label{H}
\begin{aligned}
&\mathrm{either}~~
2<p<q<2+\frac{4}{N_{\gamma}},\\
&\mathrm{or}~~
2+\frac{4}{N_{\gamma}}<q<p<2_{\gamma}^{*},
\end{aligned} \tag{H}
\end{equation}
where $N_{\gamma}=k+(1+\gamma)l$ is the so-called homogeneous dimension attached to the Grushin operator.
The exponent $2+\frac{4}{N_{\gamma}}$  is the mass critical exponent. In this paper, the term critical mass refers to the threshold value of the prescribed mass, see the below discussion.
The exponent $2_{\gamma}^{*}$  associated to $N_{\gamma}$ is
\begin{equation*}
\begin{aligned}
\begin{cases}
2_{\gamma}^{*}=\infty,&k=l=1\mathrm{~and~} \gamma=0,\\
2_{\gamma}^{*}=\frac{2N_{\gamma}}{N_{\gamma}-2},&\mathrm{others}.
\end{cases}
\end{aligned}
\end{equation*}

Grushin \cite{Grushin1,Grushin2} and Baouendi \cite{Baouendi1967} introduced the Grushin operator $\Delta_{\gamma}$ as follows
\begin{equation*}
\begin{aligned}
\Delta_{\gamma}v:=\Delta_xv+(1+\gamma)^{2}|x|^{2\gamma}\Delta_yv,
\ \
(x,y)\in \mathbb{R}^{k}\times \mathbb{R}^{l},
\ \
\gamma>0,
\end{aligned}
\end{equation*} 
where $\Delta_x$ and $\Delta_y$ are the Laplace operators in the variables
$x$ and $y$, respectively.
Furthermore, if $\gamma=0$, the operator $\Delta_{0}$ is the well-known Laplace operator. When $\gamma=1$, $k=1$ and
$l=2i$ ($i\in \mathbb{N}$),
the operator $\Delta_{\gamma}$ is the Heisenberg group Laplace operator \cite{Monti2006Duke}.
For $k\geqslant1$ and $\gamma>0$,
the operator $\Delta_{\gamma}$ is elliptic for $x\in\mathbb{R}^{k}$, $|x|\not=0$ and degenerates on the manifold $\{0\}\times\mathbb{R}^{l}$ \cite{Beckner2001PAMS,Monti2006CPDE,Monti-Morbidelli2008PAMS}.
For more related works,
we refer the readers to \cite{Egorov1975RMS,Hormander1967Acta,Kogoj-Sonner2016CVEE,Laptev-Ruzhansky-Yessirkegenov2019MS}.

The solutions of equation \eqref{G} with a prescribed mass $c>0$ are the critical points of the energy functional
\begin{equation*}
\begin{aligned}
E_{\gamma}(u):=
\frac{1}{2}
\int_{\mathbb{R}^{N}}
|\nabla_\gamma u|^{2}
\mathrm{d}z
+
\frac{\alpha}{p}
\int_{\mathbb{R}^{N}}
|u|^p
\mathrm{d}x
-
\frac{1}{q}
\int_{\mathbb{R}^{N}}
|u|^{q}
\mathrm{d}x,
\end{aligned}
\end{equation*}
constrained on the $L^{2}$-sphere of $H_\gamma^{1,2}(\mathbb{R}^{N})$ (the weighted Sobolev space, see the beginning of the next section)  described by
\begin{equation*}
\begin{aligned}
S_{c}:=\left\{u\in H_\gamma^{1,2}(\mathbb{R}^{N})
:
\int_{\mathbb{R}^{N}}|u|^{2}\mathrm{d}z=c
\right\}.
\end{aligned}
\end{equation*}
So, by a solution of \eqref{G} we mean a couple $(u,\omega)\in H_\gamma^{1,2}(\mathbb{R}^{N})\times \mathbb{R}$, where $\omega$ is the Lagrange multiplier associated to the critical point $u$ on $S_{c}$.

\begin{definition}
A function $u\in S_{c}$ is called an {\it energy ground state solution} (energy GSS for short) if
\begin{equation*}
\begin{aligned}
E_{\gamma}(u)=m_{c,\gamma},
\end{aligned}
\end{equation*}
where ground state energy is
\begin{equation*}
\begin{aligned}
m_{c,\gamma}:=\inf_{v\in S_{c}}E_{\gamma}(v).
\end{aligned}
\end{equation*}
\end{definition}

For $\gamma=0$ and $\alpha=0$, Cazenave-Lions \cite{Cazenave-Lions1982CMP} studied the existence and orbital stability of energy GSS to equation \eqref{G} with $q\in(2,2+\frac{4}{N})$ via the concentration compactness principle. Weinstein \cite{Weinstein1983CMP} proved the existence and non-existence results to equation \eqref{G} for $q=2+\frac{4}{N}$. 

After the seminal works by Tao-Visan-Zhang \cite{Tao-Visan-Zhang2007CPDE} on the nonlinear Schr\"{o}dinger equation with combined nonlinearities, the interest on this type of equation rapidly increased, see \cite{Bellazzini-Forcella-Georgiev2023Pisa,FuGY-ZhuSH2026JMPA,Jeanjean-Jendrej-Le-ViscigliaSS2022JMPA,MiaoCX-XuGX-ZhaoLF2013CMP,MiaoCX-ZhaoTF-ZhengJQ2017CVPDE}. 
We recall some known results.

For $\gamma=0$ and $\alpha<0$, Soave \cite{Soave2020JDE} investigated the existence and stability of energy GSS to equation \eqref{G} with  $2<p\leqslant2+\frac{4}{N}\leqslant q<2_{0}^{*}:=\frac{2N}{N-2}$. 
Subsequently, Soave \cite{Soave2020JFA} extended the results in \cite{Soave2020JDE} to the Sobolev critical case for $q=2_{0}^{*}$. 
For the existence and multiplicity results to equation \eqref{G} with $\gamma=0$, we refer to \cite{KangJC-TangCL2026CVPDE,WeiJC-WuYZ2020JMPA,WeiJC-WuYZ2022JFA,WeiJC-WuYZ2023Edinburgh}. 

For $\gamma=0$ and $\alpha>0$, Killip-Oh-Pocovnicu-Visan 
\cite{Killip-Oh-Pocovnicu-Visan2017ARMA} proved that equation \eqref{G} admits a energy GSS. 
Bellazzini-Forcella-Georgiev
\cite{Bellazzini-Forcella-Georgiev2026arXiv} obtained the existence and non-existence results for the energy GSS of equation \eqref{G} with $2<p<q<2+\frac{4}{N}$.
Lewin-Nodari \cite{Lewin-Nodari2020CVPDE} studied the existence of positive and radial decreasing energy GSS of equation \eqref{G}. 
Moreover, 
Jeanjean-Lu
\cite{Jeanjean-LuSS2022CVPDE} investigated the relation between energy GSS and action GSS. 
For more related works,
we refer the readers to \cite{Jeanjean-LuSS2022MMMAS}.
We summarize several known results in the following table. Here and in the sequel, $c_{0}$ and $c_{\gamma}$ are called the critical mass. 
\begin{center}
\renewcommand\arraystretch{1.5}
\begin{tabular}[h]{c|@{\hskip 1em}c@{\hskip 1em}|@{\hskip 1em}c@{\hskip 1em}|c}
\hline
$p,q$& Ground state energy&Attainability of $m_{c,0}$&Ref\\
\hline
$2<p<q<2+\frac{4}{N}$&$m_{c,0}=0$ if $c<c_{0}$&No&\cite{Jeanjean-LuSS2022CVPDE}\\
&$m_{c,0}=0$ if $c=c_{0}$&Yes& \cite{Bellazzini-Forcella-Georgiev2026arXiv}\\
&$m_{c,0}<0$ if $c>c_{0}$&Yes&\cite{Jeanjean-LuSS2022CVPDE}\\
\hline
$2<p<2+\frac{4}{N}<q<\frac{2N}{N-2}$&$m_{c,0}<0$ &Yes& \cite{Lewin-Nodari2020CVPDE}\\
\hline
$2+\frac{4}{N}=q<p<\frac{2N}{N-2}$&$m_{c,0}=0$ if $c<{c}_{0}$&No& \cite{Lewin-Nodari2020CVPDE}\\
&$m_{c,0}=0$ if $c={c}_{0}$&No&\cite{Lewin-Nodari2020CVPDE}\\
&$m_{c,0}<0$ if $c>{c}_{0}$&Yes&\cite{Lewin-Nodari2020CVPDE}\\
\hline
$2+\frac{4}{N}<q<p<\frac{2N}{N-2}$&$m_{c,0}>0$ if $c<{c}_{0}$&Yes& \cite{Lewin-Nodari2020CVPDE}\\
&$m_{c,0}=0$ if $c={c}_{0}$&Yes&\cite{Lewin-Nodari2020CVPDE}\\
&$m_{c,0}<0$ if $c>{c}_{0}$&Yes&\cite{Lewin-Nodari2020CVPDE}\\
\hline
\end{tabular}

\smallskip
Table 1. The case $\gamma=0$ and $\alpha >0$.
\end{center}

\begin{definition}
A function $u\in \tilde{S}_{c}$ is called a {\it $x$-radial energy GSS} if
\begin{equation*}
\begin{aligned}
E_{\gamma}(u)=\tilde{m}_{c,\gamma}:=\inf_{v\in \tilde{S}_{c}}E_{\gamma}(v),
\end{aligned}
\end{equation*}
where 
\begin{equation*}
\begin{aligned}
\tilde{S}_{c}:=\left\{u\in H^{1,2}_{\gamma,rad_{x}}(\mathbb{R}^{N})
:
\int_{\mathbb{R}^{N}}|u|^{2}\mathrm{d}z=c
\right\}.
\end{aligned}
\end{equation*}
\end{definition}

\begin{remark}
When $\gamma=0$,
we know that the energy GSS is radially via rearrangement inequality.
Then one has
$m_{c,0}=\tilde{m}_{c,0}$.
However,
for $\gamma>0$,
we just know $m_{c,\gamma}\leqslant\tilde{m}_{c,\gamma}$,
since the rearrangement arguments are still not available, see \cite{Giang-Tri-Tuan2026JMAA}.
\end{remark}

For $\gamma>0$ and $\alpha>0$, Yu-Yang-Tang \cite{YuSB-YangC-TangCL2026JGA} studied the existence of $x$-radial energy GSS of equation \eqref{G}.
We summarize their results in the following table.
\begin{center}
\renewcommand\arraystretch{1.5}
\begin{tabular}[h]{c|@{\hskip 1em}c@{\hskip 1em}|@{\hskip 1em}c@{\hskip 1em}|c}
\hline
$p,q$& Ground state energy&Attainability of $\tilde{m}_{c,\gamma}$&Ref\\
\hline
$2<p<q<2+\frac{4}{N_{\gamma}}$&$\tilde{m}_{c,\gamma}=0$ if $c<c_{\gamma}$&No&\cite{YuSB-YangC-TangCL2026JGA}\\
&$\tilde{m}_{c,\gamma}=0$ if $c=c_{\gamma}$&Not addressed& \\
&$\tilde{m}_{c,\gamma}<0$ if $c>c_{\gamma}$&Yes&\\
\hline
$2+\frac{4}{N_{\gamma}}<q<p<2_{\gamma}^{*}$&$\tilde{m}_{c,\gamma}=0$ if $c<{c}_{\gamma}$&No& \cite{YuSB-YangC-TangCL2026JGA}\\
&$\tilde{m}_{c,\gamma}=0$ if $c={c}_{\gamma}$&Not addressed&\\
&$\tilde{m}_{c,\gamma}<0$ if $c>{c}_{\gamma}$&Yes&\\
\hline
\end{tabular}

\smallskip
Table 2. The case $\gamma>0$ and $\alpha >0$.
\end{center}

The critical mass effects not only  the existence of energy GSS but also the global existence, stability and scattering results.
The energy GSS with critical mass also plays a key role.
For example,
when $\gamma=0$, $\alpha=0$ and $q=2+\frac{4}{N}$,
Weinstein \cite[Theorem A]{Weinstein1983CMP} showed the global existence result with the help of the critical mass and the energy GSS with critical mass.
When $\gamma=0$, $\alpha>0$ and $2<p<q<2+\frac{4}{N}$,
Bellazzini-Forcella-Georgiev
\cite[Theorem 4]{Bellazzini-Forcella-Georgiev2026arXiv} considerded the small data scattering by using critical mass.
When $\gamma=0$, $\alpha>0$ and $2+\frac{4}{N}<q<p=\frac{2N}{N-2}$,
Killip-Oh-Pocovnicu-Visan \cite[Theorem 1.3]{Killip-Oh-Pocovnicu-Visan2017ARMA} proved the scattering result by virtue of the critical mass and the energy GSS with critical mass.
Hence,
it is important to characterize the critical mass and the energy GSS with critical mass.

For $2+\frac{4}{N}<q<p<\frac{2N}{N-2}$,
Lewin-Nodari \cite[Page 39]{Lewin-Nodari2020CVPDE} characterized the critical mass  by the best constant of Gagliardo-Nirenberg inequality as follows
\begin{equation*}
\begin{aligned}
c_{0}=\left[
\frac{p}{q}
\frac{(2-q)N+4}{(2-p)N+4}
\right]
^{\frac{(2-p)N+4}{2(p-q)}}
\left[
\frac{2}{p}
\frac{N(q-p)}{(2-q)N+4}
\right]
^{\frac{-N}{2}}
K_{N,p,q}
^{-\frac{(2-p)N+4}{2(q-p)}},
\end{aligned}
\end{equation*}
where $K_{N,p,q}$ is the best constant of Gagliardo-Nirenberg inequality
\begin{equation}\label{1.1}
\begin{aligned}
&
K_{N,p,q}^{-1}
\int_{\mathbb{R}^{N}}
|u|^{q}
\mathrm{d}z\\
\leqslant&
\left(
\int_{\mathbb{R}^{N}}
|u|^{2}
\mathrm{d}z
\right)
^{\frac{2(q-p)}{2N+4-pN}}
\left(
\int_{\mathbb{R}^{N}}
|u|^{p}
\mathrm{d}z
\right)
^{\frac{2N+4-qN}{2N+4-pN}}
\left(
\int_{\mathbb{R}^{N}}
|\nabla u|^{2}
\mathrm{d}z
\right)
^{\frac{N(q-p)}{2N+4-pN}}.
\end{aligned}
\end{equation}
We explain now our motivations and our results. We first observe that Bellazzini-Forcella-Georgiev
\cite[Theorems 2 and 3]{Bellazzini-Forcella-Georgiev2026arXiv} and Yu-Yang-Tang \cite[Theorem 1.2]{YuSB-YangC-TangCL2026JGA}  showed the existence of a critical mass but with no characterization. Thus, a natural question is whether one can get qualitative and quantitative information about the critical mass. This is one of our aims.

On the other hand, Lewin-Nodari \cite{Lewin-Nodari2020CVPDE} established the existence, uniqueness, non-degeneracy of energy GSS, again with no further characterization.
Yu-Yang-Tang \cite[Theorem 1.2]{YuSB-YangC-TangCL2026JGA} showed the existence of $x$-radial energy GSS,
and the relation between $x$-radial action GSS and $x$-radial energy GSS in \cite[Corollary 1.4]{YuSB-YangC-TangCL2026JGA}.
Bellazzini-Forcella-Georgiev
\cite[Theorems 2 and 3]{Bellazzini-Forcella-Georgiev2026arXiv} are concenerned with the existence of energy GSS.
Our second aim is therefore to obtain qualitative and quantitative information about $x$-radial energy GSS with critical mass.
To the best of our knowledge, there are no  previous results in this direction. 

\subsection{Main results}
In order to characterize the critical mass and the $x$-radial energy GSS with critical mass, we start with establishing the attainability of the following Gagliardo-Nirenberg inequality and estimating its best constant.
\begin{theorem}
\label{Theorem1.1}
Let $N=k+l$,
$k,l\geqslant 1$,
$\gamma\geqslant0$
and $p,q$ satisfy \eqref{H}.
Then there exists $C_{N_{\gamma},p,q}>0$ such that,
for all $u\in H^{1,2}_{\gamma,rad_{x}}(\mathbb{R}^{N})$,
\begin{equation}\label{1.2}
\begin{aligned}
&C_{N_{\gamma},p,q}^{-1}
\int_{\mathbb{R}^{N}}
|u|^{q}
\mathrm{d}z\\
\leqslant&
\left(
\int_{\mathbb{R}^{N}}
|u|^{2}
\mathrm{d}z
\right)
^{\frac{2(q-p)}{2N_{\gamma}+4-pN_{\gamma}}}
\left(
\int_{\mathbb{R}^{N}}
|u|^{p}
\mathrm{d}z
\right)
^{\frac{2N_{\gamma}+4-qN_{\gamma}}{2N_{\gamma}+4-pN_{\gamma}}}
\left(
\int_{\mathbb{R}^{N}}
|\nabla_{\gamma}u|^{2}
\mathrm{d}z
\right)
^{\frac{N_{\gamma}(q-p)}{2N_{\gamma}+4-pN_{\gamma}}}.
\end{aligned}
\end{equation}
Moreover,
there exists $Q\in H^{1,2}_{\gamma,rad_{x}}(\mathbb{R}^{N})$ such that
\begin{equation*}
\begin{aligned}
W_{N_{\gamma},p,q}(Q)=C_{N_{\gamma},p,q}^{-1}~\mathrm{and}~
\int_{\mathbb{R}^{N}}
|\nabla_{\gamma}Q|^{2}
\mathrm{d}z
=
\int_{\mathbb{R}^{N}}|Q|^{p}\mathrm{d}z=1,
\end{aligned}
\end{equation*}
where $W_{N_{\gamma},p,q}: H^{1,2}_{\gamma,rad_{x}}(\mathbb{R}^{N}) \to \mathbb{R}^{+}$ is the Weinstein type functional defined by
\begin{equation*}
\begin{aligned}
W_{N_{\gamma},p,q}(u)
=
\frac{
\left(
\int_{\mathbb{R}^{N}}
|u|^{2}
\mathrm{d}z
\right)
^{\frac{2(q-p)}{2N_{\gamma}+4-pN_{\gamma}}}
\left(
\int_{\mathbb{R}^{N}}
|u|^{p}
\mathrm{d}z
\right)
^{\frac{2N_{\gamma}+4-qN_{\gamma}}{2N_{\gamma}+4-pN_{\gamma}}}
\left(
\int_{\mathbb{R}^{N}}
|\nabla_{\gamma}u|^{2}
\mathrm{d}z
\right)
^{\frac{N_{\gamma}(q-p)}{2N_{\gamma}+4-pN_{\gamma}}}
}{
\int_{\mathbb{R}^{N}}
|u|^{q}
\mathrm{d}z
}.
\end{aligned}
\end{equation*}
In particular, the best constant of inequality \eqref{1.2} is achieved by $Q$.
Furthermore, the best constant $C_{N_{\gamma},p,q}$ satisfies
\begin{equation*}
\begin{aligned}
&\left(\frac{p}{q}\right)^{\frac{N}{2}}
\left[
\frac{2^{1+\frac{N}{2}}(k+\gamma l)}
{kN_{\gamma}p^{\frac{N}{2}}\pi^{\frac{N}{N_{\gamma}}}}
\left(
\frac{k}
{(1+\gamma)^{2}(k+\gamma l)
\left(\frac{k}{2}+\gamma\right)^{\gamma}}
\right)^{\frac{l}{N_{\gamma}}}
\right]^{\frac{N_{\gamma}(q-p)}{2N_{\gamma}+4-pN_{\gamma}}}
\\
\leqslant&
C_{N_{\gamma},p,q}\\
<&\left(\frac{N_{\gamma}+2}{N_{\gamma}}\right)
^{\frac{N_{\gamma}(q-p)}{2N_{\gamma}+4-pN_{\gamma}}}
\left(\frac{k+\gamma l}{k}\right)
^{\frac{k(q-p)}{2N_{\gamma}+4-pN_{\gamma}}}
\left[
\frac{(N_{\gamma}+2)(1+\gamma)
\Gamma\left(\frac{(1+\gamma)l}{2}\right)}
{2^{N+1}\pi^{\frac{N}{2}}
\Gamma\left(\frac{l}{2}\right)
\Gamma\left(\frac{N_{\gamma}+2}{2}\right)}
\right]^{\frac{2(q-p)}{2N_{\gamma}+4-pN_{\gamma}}}
\\&\times
\left[
\frac{2\gamma+1}{(1+\gamma)^{\gamma+3}}
\left(\frac{4(k+\gamma l)}{k^{2}l}\right)^{\gamma}
\right]^{\frac{l(q-p)}{2N_{\gamma}+4-pN_{\gamma}}},
\end{aligned}
\end{equation*}
where $\Gamma$ denotes the Euler Gamma function.
\end{theorem}

\begin{remark}
For $\gamma=0$,
it is easy to prove the existence of optimizer for Gagliardo-Nirenberg inequality \eqref{1.2} by  rearrangement inequalities and inequality \eqref{1.2} is reduced to \eqref{1.1}.
However,
for $\gamma>0$, the situation is completely different since we can not apply rearrangement arguments.
\end{remark}

\begin{remark}
The estimates for $C_{N_{\gamma},p,q}$ seem to be new also for the Euclidean space  case $\gamma=0$.
\end{remark}

In order to get the optimizer of inequality \eqref{1.2}, the idea is to scale the minimizing sequence in a proper way.
Usually,
in order to get the boundedness of minimizing sequence and rule out the vanishing of it,
one scales as
\begin{equation*}
\begin{aligned}
\|u_{n}\|_{D^{1,2}_{\gamma}(\mathbb{R}^{N})}
=
\|u\|_{L^{2}(\mathbb{R}^{N})}
=1.
\end{aligned}
\end{equation*}
Consider
the powers of  the $L^{2}$ and $D^{1,2}_{\gamma}$-terms in the Weinstein type functional $W_{N_{\gamma},p,q}$ and observe that
\begin{equation*}
\begin{aligned}
\frac{2(q-p)}{2N_{\gamma}+4-pN_{\gamma}}+\frac{N_{\gamma}(q-p)}{2N_{\gamma}+4-pN_{\gamma}}\geqslant1
\Leftrightarrow
\begin{cases}
&\mathrm{either}~~
2+\frac{2p}{N_{\gamma}+2}\leqslant q<2+\frac{4}{N_{\gamma}},\\
&\mathrm{or}~~
2+\frac{4}{N_{\gamma}}<q\leqslant2+\frac{2p}{N_{\gamma}+2},
\end{cases}
\end{aligned}
\end{equation*}
and
\begin{equation*}
\begin{aligned}
\frac{2(q-p)}{2N_{\gamma}+4-pN_{\gamma}}+\frac{N_{\gamma}(q-p)}{2N_{\gamma}+4-pN_{\gamma}}<1
\Leftrightarrow
\begin{cases}
&\mathrm{either}~~
2<p<q<2+\frac{2p}{N_{\gamma}+2},\\
&\mathrm{or}~~
2+\frac{2p}{N_{\gamma}+2}<q<p.
\end{cases}
\end{aligned}
\end{equation*}
In this scaling, when 
$\frac{2(q-p)}{2N_{\gamma}+4-pN_{\gamma}}+\frac{N_{\gamma}(q-p)}{2N_{\gamma}+4-pN_{\gamma}}\geqslant1$, it is easy to rule out the dichotomy of the minimizing sequence. This is not the case when we are in the complementary situation.

Instead, observing the powers of  the $L^{p}$ and $D^{1,2}_{\gamma}$-terms we notice that
\begin{equation*}
\begin{aligned}
\frac{2N_{\gamma}+4-qN_{\gamma}}{2N_{\gamma}+4-pN_{\gamma}}
+
\frac{N_{\gamma}(q-p)}{2N_{\gamma}+4-pN_{\gamma}}
=1.
\end{aligned}
\end{equation*}
The idea is then to scale accordingly:
\begin{equation*}
\begin{aligned}
\|u_{n}\|_{D^{1,2}_{\gamma}(\mathbb{R}^{N})}
=
\|u\|_{L^{p}(\mathbb{R}^{N})}
=1.
\end{aligned}
\end{equation*}
The dichotomy can be then easily ruled out. It seems that this kind of normalization has not previously been used for the mixed-power Grushin functional considered here.
On the other hand,
in this scaling,
it is hard to get boundedness of minimizing sequence and rule out the vanishing of it.
We apply the Gagliardo-Nirenberg inequality to handle these aspects,
see Lemma \ref{Lemma3.1}.

Before discussing the existence and non-existence of $x$-radial energy GSS, 
we discuss conditions for $\tilde{m}_{c,\gamma}<0$ and $\tilde{m}_{c,\gamma}=0$. We obtain here a quantitative refinement of the results in \cite{YuSB-YangC-TangCL2026JGA}.
\begin{theorem}\label{Theorem1.2}
Let $N=k+l$,
$k,l\geqslant 1$,
$\gamma\geqslant0$,
$\alpha=1$
and $p,q$ satisfy \eqref{H}. Let
\begin{equation*}
\begin{aligned}
c_{\gamma}:=
C_{N_{\gamma},p,q}
^{-\frac{4-N_{\gamma}(p-2)}{2(q-p)}}
\left[
\frac{p}{q}
\frac{4-N_{\gamma}(q-2)}{4-N_{\gamma}(p-2)}
\right]
^{\frac{4-N_{\gamma}(p-2)}{2(p-q)}}
\left[
\frac{2}{p}
\frac{N_{\gamma}(q-p)}{4-N_{\gamma}(q-2)}
\right]
^{\frac{-N_{\gamma}}{2}}.
\end{aligned}
\end{equation*}
Then, it holds
$$
\tilde{m}_{c,\gamma}<0 \mbox{ if and only if } c>c_{\gamma}
$$
and 
$$
\tilde{m}_{c,\gamma}=0 \mbox{ if and only if } c\leqslant c_{\gamma}.
$$
\end{theorem}

\smallskip

\begin{remark}
By Theorem \ref{Theorem1.1} we can estimate $c_{\gamma}$ as follows:
\begin{equation*}
\begin{aligned}
c_{\gamma}>\left(C_{N_{\gamma},2+\frac{4}{N_{\gamma}}}\right)
^{-\frac{N_{\gamma}}{2}}
\left[
\frac{p}{q}
\frac{4-N_{\gamma}(q-2)}{4-N_{\gamma}(p-2)}
\right]
^{\frac{4-N_{\gamma}(p-2)}{2(p-q)}}
\left[
\frac{2}{p}
\frac{N_{\gamma}(q-p)}{4-N_{\gamma}(q-2)}
\right]
^{\frac{-N_{\gamma}}{2}},
\end{aligned}
\end{equation*}
and
\begin{equation*}
\begin{aligned}
c_{\gamma}
\leqslant&
\left(\frac{p}{q}\right)^{-\frac{4-N_{\gamma}(p-2)}{2(q-p)}\frac{N}{2}}
\left[
\frac{2^{1+\frac{N}{2}}(k+\gamma l)}
{kN_{\gamma}p^{\frac{N}{2}}\pi^{\frac{N}{N_{\gamma}}}}
\left(
\frac{k}
{(1+\gamma)^{2}(k+\gamma l)
\left(\frac{k}{2}+\gamma\right)^{\gamma}}
\right)^{\frac{l}{N_{\gamma}}}
\right]^{-\frac{N_{\gamma}}{2}}\\
&\times
\left[
\frac{p}{q}
\frac{4-N_{\gamma}(q-2)}{4-N_{\gamma}(p-2)}
\right]
^{\frac{4-N_{\gamma}(p-2)}{2(p-q)}}
\left[
\frac{2}{p}
\frac{N_{\gamma}(q-p)}{4-N_{\gamma}(q-2)}
\right]
^{\frac{-N_{\gamma}}{2}}.
\end{aligned}
\end{equation*}
\end{remark}

Next, with the help of Theorem \ref{Theorem1.2},
we show the existence and non-existence of $x$-radial energy GSS, giving an alternative quantitative proof of the results in \cite{YuSB-YangC-TangCL2026JGA}, extending them to the case $c=c_{\gamma}$.
\begin{theorem}\label{Theorem1.3}
Let $N=k+l$,
$k,l\geqslant 1$,
$\gamma\geqslant0$,
$\alpha=1$
and $p,q$ satisfy \eqref{H}.
Then, for $c\in(0,c_{\gamma})$ there are no $x$-radial energy GSS, while for $c\in[c_{\gamma},\infty)$ there exists a $x$-radial energy GSS $u_{c}\in \tilde{S}_{c}$ such that $\tilde{m}_{c,\gamma}=E_{\gamma}(u_{c})$.

Furthermore,
the corresponding Lagrange multiplier is such that $\omega>0$.
If $2+\frac{4}{N_{\gamma}}<q<p<2_{\gamma}^{*}$,
then we also know that
$\omega<\omega_{\sharp}:=
\frac{2(p-q)}{q(p-2)}
\left[
\frac{p(q-2)}{q(p-2)}
\right]
^{\frac{q-2}{p-q}}$.
\end{theorem}

\smallskip 

\begin{remark}
For $\gamma=0$,
the $x$-radial energy GSS is radial via rearrangements,
and it is also an energy GSS.
In this case,
Theorem \ref{Theorem1.3} covers the results of \cite[Theorems 2 and 3]{Bellazzini-Forcella-Georgiev2026arXiv} and \cite[Theorem 6]{Lewin-Nodari2020CVPDE}, while for $\gamma>0$ extends \cite[Theorem 1.2]{YuSB-YangC-TangCL2026JGA}.
\end{remark}

\begin{definition}
We say that $u_{\omega}$ is an action ground state solution (action GSS for short) of equation \eqref{G}, if $u_{\omega}$ achieves
\begin{equation*}
\begin{aligned}
e_{\omega}:=
\inf\big\{
I_{\omega}(u):I'_{\omega}(u)=0,u\in H^{1,2}_{\gamma}(\mathbb{R}^N)\backslash\{0\}
\big\},
\end{aligned}
\end{equation*}
where the action functional is
\begin{equation*}
\begin{aligned}
I_{\omega}(u):=
\frac{1}{2}
\int_{\mathbb{R}^{N}}
|\nabla_{\gamma}u|^{2}
\mathrm{d}z
+
\frac{\omega}{2}
\int_{\mathbb{R}^{N}}
|u|^{2}
\mathrm{d}z
+
\frac{\alpha}{p}
\int_{\mathbb{R}^{N}}
|u|^p
\mathrm{d}z
-
\frac{1}{q}
\int_{\mathbb{R}^{N}}
|u|^q
\mathrm{d}z.
\end{aligned}
\end{equation*}
\end{definition}

\begin{definition}
We say $u_{\omega}$ is a $x$-radial action GSS of equation \eqref{G}, if $u_{\omega}$ achieves
\begin{equation*}
\begin{aligned}
\tilde{e}_{\omega}:=
\inf\big\{
I_{\omega}(u):I'_{\omega}(u)=0,u\in H^{1,2}_{\gamma,rad_{x}}(\mathbb{R}^N)\backslash\{0\}
\big\}.
\end{aligned}
\end{equation*}
\end{definition}

\begin{remark}
For $\gamma>0$,
Alves-de Holanda \cite{Alves-de-Holanda2021} established the existence of $x$-radial action GSS for equation \eqref{G} with Berestycki-Lions type condition.
For more related results, we refer to \cite{Abatangelo-Ferrero-Luzzini2025JDE,Alves-Gandal-Loiudice-Tyagi2024JGA,Bisci-Malanchini-Secchi2025PA,Loiudice2009NA,Kogoj-Lanconelli2012NA}.

Again, when $\gamma=0$,
one has  
$e_{\omega}=\tilde{e}_{\omega}$ and $x$-radial action GSS is the action GSS.
However,
for $\gamma>0$,
we just know $e_{\omega}\leqslant\tilde{e}_{\omega}$.
\end{remark}

In Theorem \ref{Theorem1.3},
we have a critical mass $c_{\gamma}$ and  $u_{c_{\gamma}}$ ($x$-radial energy GSS with critical mass).
In the following  
we discuss in more details the relation between the critical mass $c_{\gamma}$ and the best constant and the optimizers of inequality \eqref{1.2}.

\begin{theorem}\label{Theorem1.4}
Let $N=k+l$,
$k,l\geqslant 1$,
$\gamma\geqslant0$,
$\alpha=1$
and $p,q$ satisfy \eqref{H}.

Let $u_{c}(x,y)=\lambda Q(\mu x, \mu^{1+\gamma} y)$ with
\begin{equation*}
\begin{aligned}
\begin{cases}
\lambda=
\left[
\frac{p}{q}
\frac{4-N_{\gamma}(q-2)}{4-N_{\gamma}(p-2)}
\int_{\mathbb{R}^{N}}
|Q|^{q}
\mathrm{d}z
\right]^{\frac{1}{p-q}},\\
\mu
=
\left[
\frac{2}{p}
\frac{N_{\gamma}(q-p)}{4-N_{\gamma}(q-2)}
\right]^{\frac{1}{2}}
\left[
\frac{p}{q}
\frac{4-N_{\gamma}(q-2)}{4-N_{\gamma}(p-2)}
\int_{\mathbb{R}^{N}}
|Q|^{q}
\mathrm{d}z
\right]
^{\frac{p-2}{2(p-q)}},
\end{cases}
\end{aligned}
\end{equation*}
where $Q$ is the optimizer with $C_{N_{\gamma},p,q}^{-1}$ in Theorem \ref{Theorem1.1}.
Then the following results hold true.

\begin{enumerate}
\item [(i)]
$u_{c}$ is an optimizer with $C_{N_{\gamma},p,q}^{-1}$;

\item [(ii)]
$u_{c}$ is a $x$-radial energy GSS for critical mass $c_{\gamma}$ with the Lagrange multiplier $\omega_{*}$,
where
\begin{equation*}
\begin{aligned}
\omega_{*}=
\frac{4(q-p)}{p[4-N_{\gamma}(q-2)]}
\left[
\frac{p}{q}
\frac{4-N_{\gamma}(q-2)}{4-N_{\gamma}(p-2)}
\int_{\mathbb{R}^{N}}
|Q|^{q}
\mathrm{d}z
\right]
^{\frac{p-2}{p-q}}
\frac{1}
{\int_{\mathbb{R}^{N}}
|Q|^{2}
\mathrm{d}z};
\end{aligned}
\end{equation*}

\item [(iii)]
For $2<p<q<2+\frac{4}{N_{\gamma}}$,
$u_{c}$ is a $x$-radial action GSS with $\omega_{*}$.
\end{enumerate}
\end{theorem}

\begin{remark}
In Theorem \ref{Theorem1.4},
we not only construct a $x$-radial energy GSS with critical mass $c_{\gamma}$,
but also show it is an optimizer with $C_{N_{\gamma},p,q}^{-1}$.
\end{remark}

Actually, we have that all $x$-radial energy GSS with critical mass $c_{\gamma}$ are optimizers of \eqref{1.2}.
\begin{theorem}\label{Theorem1.5}
Let $N=k+l$,
$k,l\geqslant 1$,
$\gamma\geqslant0$,
$\alpha=1$
and $p,q$ satisfy \eqref{H}.
Let $U\in H^{1,2}_{\gamma,rad_{x}}(\mathbb{R}^{N})$ be any $x$-radial energy GSS with critical mass $c_{\gamma}$.
Then the following results hold true.

\begin{enumerate}
\item [(i)]
$U$ is the optimizer of \eqref{1.2};

\item [(ii)]
for $2<p<q<2+\frac{4}{N_{\gamma}}$,
if the corresponding Lagrange multiplier is $\omega_{*}$,
then $U$ is a $x$-radial action GSS with $\omega=\omega_{*}$.
\end{enumerate}
\end{theorem}

\begin{theorem}\label{Theorem1.6}
Let $N=k+l$,
$k,l\geqslant 1$,
$\gamma\geqslant0$,
$\alpha=1$
and 
$2<p<q<2+\frac{4}{N_{\gamma}}$.
For $\omega=\omega_{*}$, equation \eqref{G} has a $x$-radial action GSS, say $V\in H^{1,2}_{\gamma,rad_{x}}(\mathbb{R}^{N})$. 
Moreover, the following assertions hold true.

\begin{enumerate}
 \item [(i)]
$V$ is a $x$-radial energy GSS of \eqref{G} for the mass $c_{\gamma}$ with the Lagrange multiplier $\omega_{*}$;

\item [(ii)]
$V$ satisfies the Pohozaev normalization condition
\begin{equation*}
\begin{aligned}
\int_{\mathbb{R}^{N}}
|\nabla_{\gamma}V|^{2}
\mathrm{d}z
=&
\frac{2N_{\gamma}(q-p)}
{q[4-N_{\gamma}(p-2)]}
\int_{\mathbb{R}^{N}}
|V|^{q}
\mathrm{d}z,\\
\int_{\mathbb{R}^{N}}
|V|^{p}
\mathrm{d}z
=&
\frac{p[4-N_{\gamma}(q-2)]}
{q[4-N_{\gamma}(p-2)]}
\int_{\mathbb{R}^{N}}
|V|^{q}
\mathrm{d}z;
\end{aligned}
\end{equation*}

\item [(iii)]
$V$ is the optimizer of \eqref{1.2}.
\end{enumerate}
\end{theorem}

\begin{remark}
For $\gamma=0$,
the relation between energy GSS and action GSS has a long history, see \cite{Jeanjean-LuSS2022CVPDE}.
For $\gamma>0$,
the relation between $x$-radial energy GSS and $x$-radial action GSS is given in \cite{YuSB-YangC-TangCL2026JGA}.
Here, we derive the relation between the $x$-radial energy GSS, $x$-radial action GSS and optimizers of Gagliardo-Nirenberg inequality.
Our results clearly hold also for $\gamma=0$.
\end{remark}

The paper is organized as follows. 
In section \ref{section2}, we present some preliminary results. 
In sections \ref{section3}-\ref{section8}, we derive the main Theorems \ref{Theorem1.1}-\ref{Theorem1.6}, respectively.

\section{Preliminaries}
\label{section2}
In this section, we first recall some preliminaries on the Grushin operator and related results. For $\gamma\geqslant0$, let
\begin{equation*}
\begin{aligned}
\nabla_{\gamma}u=
(\nabla_{x}u,(1+\gamma)|x|^{\gamma}\nabla_{y}u)=
(u_{x_1},\cdots,u_{x_k},(1+\gamma)|x|^{\gamma}u_{y_1},\cdots,(1+\gamma)|x|^{\gamma}u_{y_l})
\end{aligned}
\end{equation*}
and
\begin{equation*}
\begin{aligned}
|\nabla_{\gamma}u|^{2}=
|\nabla_{x}u|^{2}
+
(1+\gamma)^{2}|x|^{2\gamma}
|\nabla_{y}u|^{2}.
\end{aligned}
\end{equation*}
Let us define the space $D^{1,2}_{\gamma}(\mathbb{R}^{N})$ as the closure of $C_0^\infty(\mathbb{R}^N)$ with respect to the norm
\begin{equation*}
\begin{aligned}
\|u\|_{D^{1,2}_{\gamma}(\mathbb{R}^{N})}^{2}:=
\int_{\mathbb{R}^N}
|\nabla_{\gamma}u|^{2}
\mathrm{d}z.
\end{aligned}
\end{equation*}
We let $H^{1,2}_{\gamma}(\mathbb{R}^{N})$ bne the weighted Sobolev space
\begin{equation*}
\begin{aligned}
H^{1,2}_{\gamma}(\mathbb{R}^{N}):=
\left\{
u\in L^{2}(\mathbb{R}^{N})
:
\frac{\partial u}{\partial x_i},|x|^{\gamma}\frac{\partial u}{\partial y_j}\in L^{2}(\mathbb{R}^{N}),i=1,\cdots k, j=1,\cdots l
\right\},
\end{aligned}
\end{equation*}
endowed with the scalar product  
\begin{equation*}
\begin{aligned}
\langle u,v\rangle_{H^{1,2}_{\gamma}(\mathbb{R}^{N})}:=
\int_{\mathbb{R}^N}
(\nabla_{\gamma}u\nabla_{\gamma}v+uv)
\mathrm{d}z
\end{aligned}
\end{equation*}
and the norm
\begin{equation*}
\begin{aligned}
\|u\|_{H^{1,2}_{\gamma}(\mathbb{R}^{N})}^{2}
:=
\int_{\mathbb{R}^N}
(|\nabla_{\gamma}u|^{2}+|u|^{2})
\mathrm{d}z.
\end{aligned}
\end{equation*}
From \cite{Monti2006Duke},
it is easy to see that the embedding $H^{1,2}_{\gamma}(\mathbb{R}^N)\hookrightarrow L^r(\mathbb{R}^N)$ is continuous for every $r\in[2,2_{\gamma}^{*}]$.

We introduce the subspace $H^{1,2}_{\gamma,rad_{x}}(\mathbb{R}^N)$ as follows
\begin{equation*}
\begin{aligned}
H^{1,2}_{\gamma,rad_{x}}(\mathbb{R}^{N}):=
\left\{
u\in H^{1,2}_{\gamma}(\mathbb{R}^N)
:
u(x,y)
=
u(|x|,y)
~\mathrm{for~all}~(x,y)\in
\mathbb{R}^{k}\times \mathbb{R}^{l}
\right\}.
\end{aligned}
\end{equation*}
\begin{lemma}\label{Lemma2.1}
{\rm \cite[Lemma 2.2]{Alves-de-Holanda2021}}
Let $R>0$ and $2< r<2_{\gamma}^{*}$.
Then the embedding of
$H^{1,2}_{\gamma,rad_{x}}(\mathbb{R}^N)$ into $L^{r}(\mathbb{R}^{k}\times B(0,R))$ is compact,
where $B(0,R)$ is the ball with radius $R$.
\end{lemma}

We recall the Gagliardo-Nirenberg inequality for the Grushin operator.
\begin{lemma}\label{Lemma2.2}
{\rm \cite{YuSB-YangC-TangCL2026JGA,SongML-LiWJ2019ANA}}
Let $N=k+l$,
$k,l\geqslant 1$,
$\gamma>0$
and
$2<r\leqslant 2_{\gamma}^{*}$.
Then for $u\in H^{1,2}_{\gamma}(\mathbb{R}^{N})$, there exists a constant $C_{N_{\gamma},r}>0$ such that
\begin{equation*}
\begin{aligned}
\|u\|_{L^{r}(\mathbb{R}^{N})}^{r}
\leqslant
C_{N_{\gamma},r}
\|u\|_{D_{\gamma}^{1,2}(\mathbb{R}^{N})}
^{\frac{N_{\gamma}(r-2)}{2}}
\|u\|_{L^{2}(\mathbb{R}^{N})}
^{r-\frac{N_{\gamma}(r-2)}{2}}.
\end{aligned}
\end{equation*}
\end{lemma}

\begin{remark}\label{remark2.1}
In \cite{Caffarelli-Kohn-Nirenberg1984Compos}, for the case $\gamma=0$,
we have the above inequality.
Moreover,
for $k=l=1$, $\gamma=0$ and $2<p<r$, the following inequality holds true
\begin{equation*}
\begin{aligned}
\|u\|_{L^{r}(\mathbb{R}^{N})}^{r}
\leqslant
C
\|u\|_{D^{1,2}(\mathbb{R}^{N})}
^{r-p}
\|u\|_{L^{p}(\mathbb{R}^{N})}^{p}.
\end{aligned}
\end{equation*}
\end{remark}

Next, we give an explicit upper bound for the constant in Lemma \ref{Lemma2.2}
at the mass critical exponent.
\begin{lemma}\label{Lemma2.3}
Let $N=k+l$, $k,l\geqslant1$ and $\gamma\geqslant0$.
For the mass critical exponent $r=2+\frac{4}{N_{\gamma}}$,
the Gagliardo-Nirenberg inequality in Lemma \ref{Lemma2.2}
holds with a constant $C_{N_{\gamma},2+\frac{4}{N_{\gamma}}}$ satisfying
\begin{equation}\label{2.1}
\begin{aligned}
C_{N_{\gamma},2+\frac{4}{N_{\gamma}}}
\leqslant&
\frac{N_{\gamma}+2}{N_{\gamma}}
\left(\frac{k+\gamma l}{k}\right)^{\frac{k}{N_{\gamma}}}
\left[
\frac{(N_{\gamma}+2)(1+\gamma)
\Gamma\left(\frac{(1+\gamma)l}{2}\right)}
{2^{N+1}\pi^{\frac{N}{2}}
\Gamma\left(\frac{l}{2}\right)
\Gamma\left(\frac{N_{\gamma}+2}{2}\right)}
\right]^{\frac{2}{N_{\gamma}}}
\\&\times
\left[
\frac{2\gamma+1}{(1+\gamma)^{\gamma+3}}
\left(\frac{4(k+\gamma l)}{k^{2}l}\right)^{\gamma}
\right]^{\frac{l}{N_{\gamma}}},
\end{aligned}
\end{equation}
where $\Gamma$ is the Gamma function.
In particular, when $\gamma=0$ \eqref{2.1} reduces to
\begin{equation*}
\begin{aligned}
C_{N,2+\frac{4}{N}}
\leqslant
\frac{N+2}{N}
\left[
\frac{N+2}{2^{N+1}\pi^{\frac{N}{2}}
\Gamma\left(\frac{N+2}{2}\right)}
\right]^{\frac{2}{N}}.
\end{aligned}
\end{equation*}
\end{lemma}

\begin{proof}
We divide the proof into six steps.

\noindent
{\bf Step 1}.
We first introduce the Fourier transforms and decompose the Grushin energy.
Let $u\in C_0^{\infty}(\mathbb{R}^{N})$ and set
\begin{equation*}
\begin{aligned}
w(x,y)&=(2\pi)^{-\frac l2}
\int_{\mathbb{R}^{l}}e^{-iy\cdot z}u(x,z)\mathrm{d}z,
\\
v(x,y)&=(2\pi)^{-\frac k2}
\int_{\mathbb{R}^{k}}e^{-ix\cdot z}w(z,y)\mathrm{d}z,
\end{aligned}
\end{equation*}
where $w$ is the Fourier transform of $u$ with respect to $y$, and $v$ is the full Fourier transform of $u$.
By Plancherel's identity, we have
\begin{equation*}
\begin{aligned}
\|v\|_{L^{2}(\mathbb{R}^{N})}^{2}
=\|w\|_{L^{2}(\mathbb{R}^{N})}^{2}
=\|u\|_{L^{2}(\mathbb{R}^{N})}^{2}
\end{aligned}
\end{equation*}
Making use of Plancherel's identity again, one can see that
\begin{equation*}
\begin{aligned}
\|u\|_{D_{\gamma}^{1,2}(\mathbb{R}^{N})}^{2}
=
\int_{\mathbb{R}^{N}}
\left(|\nabla_xw|^{2}+(1+\gamma)^{2}|x|^{2\gamma}|y|^{2}|w|^{2}\right)
\mathrm{d}x\mathrm{d}y.
\end{aligned}
\end{equation*}
For $\gamma>0$, we have
\begin{equation*}
\begin{aligned}
\|u\|_{D_{\gamma}^{1,2}(\mathbb{R}^{N})}^{2}
=&
\int_{\mathbb{R}^{l}}
\int_{\mathbb{R}^{k}}
\left(\frac{\gamma l}{k+\gamma l}|\nabla_xw|^{2}
+(1+\gamma)^{2}|x|^{2\gamma}|y|^{2}|w|^{2}\right)
\mathrm{d}x
\mathrm{d}y
\\&+
\frac{k}{k+\gamma l}
\int_{\mathbb{R}^{N}}|\nabla_xw|^{2}\mathrm{d}x\mathrm{d}y.
\end{aligned}
\end{equation*}

\noindent
{\bf Step 2}.
In this step, we establish an estimate relating the weighted $L^{2}$ norm
to the gradient norm.
Let $\gamma>0$ and $f\in H^1(\mathbb{R})$ with
$\|f\|_{L^{2}(\mathbb{R})}=1$.
Since the continuous representative of $f$ tends to zero as $x\to\pm\infty$,
we have, for every $x\in\mathbb{R}$,
\begin{equation*}
\begin{aligned}
|f(x)|^{2}
&=\operatorname{Re}\int_{-\infty}^{x}\overline{f(s)}f'(s)\mathrm{d}s
-\operatorname{Re}\int_x^{\infty}\overline{f(s)}f'(s)\mathrm{d}s
\\
&\leqslant\int_{\mathbb{R}}|f(s)||f'(s)|\mathrm{d}s
\leqslant\|f\|_{L^{2}(\mathbb{R})}\|f'\|_{L^{2}(\mathbb{R})},
\end{aligned}
\end{equation*}
Since $\|f\|_{L^{2}(\mathbb{R})}=1$, we have
\begin{equation*}
\begin{aligned}
|f(x)|^{2}\leqslant
\|f'\|_{L^{2}(\mathbb{R})}.
\end{aligned}
\end{equation*}
Taking the supremum over $x$, we obtain
\begin{equation*}
\begin{aligned}
\|f\|_{L^{\infty}(\mathbb{R})}^{2}
\leqslant
\|f'\|_{L^{2}(\mathbb{R})}.
\end{aligned}
\end{equation*}
It follows that, for every $t>0$,
\begin{equation*}
\begin{aligned}
\int_{|x|\leqslant t}
|f(x)|^{2}
\mathrm{d}x
\leqslant
2t\|f\|_{L^\infty(\mathbb{R})}^{2}
\leqslant
2t\|f'\|_{L^{2}(\mathbb{R})}.
\end{aligned}
\end{equation*}
Then,
\begin{equation*}
\begin{aligned}
\int_{|x|>t}|f|^{2}\mathrm{d}x
=
1-
\int_{|x|\leqslant t}|f|^{2}\mathrm{d}x
\geqslant
1-2t\|f'\|_{L^{2}(\mathbb{R})},
\ \ t>0.
\end{aligned}
\end{equation*}
By Tonelli's theorem, we have
\begin{equation*}
\begin{aligned}
\int_{\mathbb{R}}|x|^{2\gamma}|f|^{2}\mathrm{d}x
&=
\int_{\mathbb{R}}
\int_{0}^{|x|}
2\gamma t^{2\gamma-1}
|f(x)|^{2}
\mathrm{d}t
\mathrm{d}x\\
&=\int_0^{\infty}2\gamma t^{2\gamma-1}
\left(\int_{|x|>t}|f|^{2}\mathrm{d}x\right)\mathrm{d}t
\\
&\geqslant\int_0^{\frac{1}{2\|f'\|_{L^{2}(\mathbb{R})}}}
2\gamma t^{2\gamma-1}
\left(1-2t\|f'\|_{L^{2}(\mathbb{R})}\right)\mathrm{d}t
\\
&=\left[
t^{2\gamma}-\frac{4\gamma}{2\gamma+1}
\|f'\|_{L^{2}(\mathbb{R})}t^{2\gamma+1}
\right]_0^{\frac{1}{2\|f'\|_{L^{2}(\mathbb{R})}}}
\\
&=\frac{1}{\left(2\|f'\|_{L^{2}(\mathbb{R})}\right)^{2\gamma}}
\left(1-\frac{2\gamma}{2\gamma+1}\right)
\\
&=\frac{1}{4^{\gamma}(2\gamma+1)
\|f'\|_{L^{2}(\mathbb{R})}^{2\gamma}}.
\end{aligned}
\end{equation*}
Note that $f'$ does not vanish identically, since otherwise $f$ would be
constant, contrary to $\|f\|_{L^{2}(\mathbb{R})}=1$.

Let $\phi\in C_0^{\infty}(\mathbb{R}^{k})$ with
$\|\phi\|_{L^{2}(\mathbb{R}^{k})}=1$.
Fix $j\in\{1,\ldots,k\}$, for $k\geqslant2$, we define
\begin{equation*}
\begin{aligned}
f_j(t)=\left.
\left(\int_{\mathbb{R}^{k-1}}|\phi(x)|^{2}
\prod_{i\ne j}\mathrm{d}x_i\right)^{\frac12}
\right|_{x_j=t}.
\end{aligned}
\end{equation*}
Since the norm on $L^{2}(\mathbb{R}^{k-1})$ is $1$-Lipschitz,
the function $f_j$ belongs to $H^1(\mathbb{R})$.
By Fubini's theorem, we have
\begin{equation*}
\begin{aligned}
\int_{\mathbb{R}}|f_j(t)|^{2}\mathrm{d}t
=\int_{\mathbb{R}^{k}}|\phi(x)|^{2}\mathrm{d}x=1.
\end{aligned}
\end{equation*}
Note that
$
\partial_{x_j}|\phi|^{2}
=\partial_{x_j}(\phi\overline \phi)
=2\operatorname{Re}\bigl(\overline \phi\,\partial_{x_j}\phi\bigr)
$.
Differentiating $f_j^{2}$, we obtain
\begin{equation*}
\begin{aligned}
2f_j(t)f_j'(t)
=
\left.
2\operatorname{Re}\int_{\mathbb R^{k-1}}
\overline{\phi(x)}\,\partial_{x_j}\phi(x)
\prod_{i\ne j}\mathrm{d}x_i
\right|_{x_j=t}.
\end{aligned}
\end{equation*}
For $f_j(t)>0$, by the Cauchy-Schwarz inequality, we have
\begin{equation*}
\begin{aligned}
|f_j'(t)|^{2}
=\left.
\frac{\left|\operatorname{Re}\int_{\mathbb{R}^{k-1}}
\overline{\phi(x)}\partial_{x_j}\phi(x)
\prod_{i\ne j}\mathrm{d}x_i\right|^{2}}
{\int_{\mathbb{R}^{k-1}}|\phi(x)|^{2}
\prod_{i\ne j}\mathrm{d}x_i}
\right|_{x_j=t}
\leqslant\left.
\int_{\mathbb{R}^{k-1}}|\partial_{x_j}\phi(x)|^{2}
\prod_{i\ne j}\mathrm{d}x_i\right|_{x_j=t}.
\end{aligned}
\end{equation*}
Since $f_j'=0$ almost everywhere on the zero set of $f_j$,
we integrate the above estimate in $t$ and obtain
\begin{equation*}
\begin{aligned}
\|f_j'\|_{L^{2}(\mathbb{R})}^{2}
\leqslant\int_{\mathbb{R}^{k}}|\partial_{x_j}\phi|^{2}\mathrm{d}x
\ \
\mathrm{and}
\ \
\int_{\mathbb{R}}|t|^{2\gamma}|f_j(t)|^{2}\mathrm{d}t
=\int_{\mathbb{R}^{k}}|x_j|^{2\gamma}|\phi|^{2}\mathrm{d}x.
\end{aligned}
\end{equation*}
For $k=1$, take $f_1=|\phi|$ and use $|(|\phi|)'|\leqslant|\phi'|$
almost everywhere. It follows that, for every $k\geqslant1$,
\begin{equation*}
\begin{aligned}
\int_{\mathbb{R}^{k}}|x_j|^{2\gamma}|\phi|^{2}\mathrm{d}x
\geqslant\frac{1}{4^{\gamma}(2\gamma+1)}
\left(\int_{\mathbb{R}^{k}}|\partial_{x_j}\phi|^{2}\mathrm{d}x\right)^{-\gamma},
\ \ j=1,\ldots,k.
\end{aligned}
\end{equation*}
Note that all the integrals of $|\partial_{x_j}\phi|^{2}$ are positive.

For $0<\gamma\leqslant1$, it follows from the concavity of $t\mapsto t^{\gamma}$ that
\begin{equation*}
\begin{aligned}
|x|^{2\gamma}
=k^{\gamma}\left(\frac{1}{k}\sum_{j=1}^{k}|x_j|^{2}\right)^{\gamma}
\geqslant k^{\gamma-1}\sum_{j=1}^{k}|x_j|^{2\gamma}.
\end{aligned}
\end{equation*}
By the convexity of $t\mapsto t^{-\gamma}$, we also have
\begin{equation*}
\begin{aligned}
\sum_{j=1}^{k}
\left(\int_{\mathbb{R}^{k}}|\partial_{x_j}\phi|^{2}\mathrm{d}x\right)^{-\gamma}
\geqslant
k\left(\frac{1}{k}\int_{\mathbb{R}^{k}}|\nabla \phi|^{2}\mathrm{d}x\right)^{-\gamma}
=k^{\gamma+1}
\left(\int_{\mathbb{R}^{k}}|\nabla \phi|^{2}\mathrm{d}x\right)^{-\gamma}.
\end{aligned}
\end{equation*}
It follows that
\begin{equation*}
\begin{aligned}
\int_{\mathbb{R}^{k}}|x|^{2\gamma}|\phi|^{2}\mathrm{d}x
\geqslant
\frac{k^{\gamma-1}}{4^{\gamma}(2\gamma+1)}
\sum_{j=1}^{k}
\left(\int_{\mathbb{R}^{k}}|\partial_{x_j}\phi|^{2}\mathrm{d}x\right)^{-\gamma}
\geqslant
\frac{k^{2\gamma}}{4^{\gamma}(2\gamma+1)}
\left(\int_{\mathbb{R}^{k}}|\nabla \phi|^{2}\mathrm{d}x\right)^{-\gamma}.
\end{aligned}
\end{equation*}
For $\gamma\geqslant1$, since $\phi$ has compact support, we have
\begin{equation*}
\begin{aligned}
k
&=k\int_{\mathbb{R}^{k}}|\phi|^{2}
\mathrm{d}x\\&=
-\int_{\mathbb{R}^{k}}x\cdot\nabla(|\phi|^{2})\mathrm{d}x\\&=
-2\operatorname{Re}\int_{\mathbb{R}^{k}}
\overline \phi\,x\cdot\nabla \phi
\mathrm{d}x\\&\leqslant
2\left(\int_{\mathbb{R}^{k}}|x|^{2}|\phi|^{2}\mathrm{d}x\right)^{\frac{1}{2}}
\left(\int_{\mathbb{R}^{k}}|\nabla \phi|^{2}\mathrm{d}x\right)^{\frac{1}{2}}.
\end{aligned}
\end{equation*}
Hence,
\begin{equation*}
\begin{aligned}
\int_{\mathbb{R}^{k}}|x|^{2}|\phi|^{2}\,\mathrm{d}x
\geqslant
\frac{k^{2}}{4}
\left(\int_{\mathbb{R}^{k}}|\nabla \phi|^{2}\,\mathrm{d}x\right)^{-1}.
\end{aligned}
\end{equation*}
By Jensen's inequality and $\|\phi\|_{L^{2}(\mathbb{R}^{k})}=1$, we have
\begin{equation}\label{2.2}
\begin{aligned}
\int_{\mathbb{R}^{k}}|x|^{2\gamma}|\phi|^{2}\mathrm{d}x
\geqslant
\left(\int_{\mathbb{R}^{k}}|x|^{2}|\phi|^{2}\mathrm{d}x\right)^{\gamma}\geqslant
\frac{k^{2\gamma}}{4^{\gamma}}
\left(\int_{\mathbb{R}^{k}}|\nabla \phi|^{2}\mathrm{d}x\right)^{-\gamma}.
\end{aligned}
\end{equation}
Therefore, for every $\gamma>0$, we have
\begin{equation}\label{2.3}
\begin{aligned}
\left(\int_{\mathbb{R}^{k}}|\nabla \phi|^{2}\mathrm{d}x\right)^{\gamma}
\int_{\mathbb{R}^{k}}|x|^{2\gamma}|\phi|^{2}\mathrm{d}x
\geqslant\frac{k^{2\gamma}}{4^{\gamma}(2\gamma+1)}.
\end{aligned}
\end{equation}

Next, we use \eqref{2.3} to establish a lower bound for the weighted energy.
Let $\gamma>0$. 
We first compute the following minimum.
For $t>0$, we have
\begin{equation*}
\begin{aligned}
\frac{\mathrm{d}}{\mathrm{d}t}
\left[t+\frac{k^{2\gamma}}{4^{\gamma}(2\gamma+1)}t^{-\gamma}\right]
&=1-\frac{\gamma k^{2\gamma}}{4^{\gamma}(2\gamma+1)}t^{-\gamma-1},
\\
\frac{\mathrm{d}^{2}}{\mathrm{d}t^{2}}
\left[t+\frac{k^{2\gamma}}{4^{\gamma}(2\gamma+1)}t^{-\gamma}\right]
&=\frac{\gamma(\gamma+1)k^{2\gamma}}
{4^{\gamma}(2\gamma+1)}t^{-\gamma-2}>0.
\end{aligned}
\end{equation*}
Hence the minimum is attained at
$t=\left[\frac{\gamma k^{2\gamma}}{4^{\gamma}(2\gamma+1)}
\right]^{\frac{1}{1+\gamma}}$.
Therefore,
\begin{equation*}
\begin{aligned}
\inf_{t>0}\left[t+\frac{k^{2\gamma}}{4^{\gamma}(2\gamma+1)}t^{-\gamma}\right]
=
\frac{1+\gamma}{\gamma}
\left[\frac{\gamma k^{2\gamma}}{4^{\gamma}(2\gamma+1)}
\right]^{\frac{1}{1+\gamma}}
=
\frac{1+\gamma}{(2\gamma+1)^{\frac{1}{1+\gamma}}}
\left(\frac{k^{2}}{4\gamma}\right)^{\frac{\gamma}{1+\gamma}}.
\end{aligned}
\end{equation*}
For nonzero $\psi\in C_0^{\infty}(\mathbb{R}^{k})$, applying \eqref{2.3}
to $\frac{\psi}{\|\psi\|_{L^{2}(\mathbb{R}^{k})}}$, we obtain
\begin{equation*}
\begin{aligned}
\int_{\mathbb{R}^{k}}|x|^{2\gamma}|\psi|^{2}\mathrm{d}x
\geqslant\frac{k^{2\gamma}}{4^{\gamma}(2\gamma+1)}
\|\psi\|_{L^{2}(\mathbb{R}^{k})}^{2\gamma+2}
\left(\int_{\mathbb{R}^{k}}|\nabla \psi|^{2}\mathrm{d}x\right)^{-\gamma}.
\end{aligned}
\end{equation*}
It follows that, for $t,s>0$,
\begin{equation*}
\begin{aligned}
&\int_{\mathbb{R}^{k}}
\left(t|\nabla \psi|^{2}+s|x|^{2\gamma}|\psi|^{2}\right)
\mathrm{d}x\\&\geqslant
t\int_{\mathbb{R}^{k}}|\nabla \psi|^{2}
\mathrm{d}x
+\frac{sk^{2\gamma}}{4^\gamma(2\gamma+1)}
\|\psi\|_{L^{2}(\mathbb{R}^{k})}^{2\gamma+2}
\left(\int_{\mathbb{R}^{k}}|\nabla \psi|^{2}
\mathrm{d}x\right)
^{-\gamma}\\&=
\|\psi\|_{L^{2}(\mathbb{R}^{k})}^{2}
\left[
t\frac{\displaystyle\int_{\mathbb{R}^{k}}|\nabla \psi|^{2}\,\mathrm{d}x}
{\|\psi\|_{L^{2}(\mathbb{R}^{k})}^{2}}
+\frac{sk^{2\gamma}}{4^\gamma(2\gamma+1)}
\left(
\frac{\int_{\mathbb{R}^{k}}|\nabla \psi|^{2}\,\mathrm{d}x}
{\|\psi\|_{L^{2}(\mathbb{R}^{k})}^{2}}
\right)^{-\gamma}
\right]\\
&\geqslant
\|\psi\|_{L^{2}(\mathbb{R}^{k})}^{2}
\inf_{r>0}
\left[tr+\frac{sk^{2\gamma}}{4^\gamma(2\gamma+1)}r^{-\gamma}\right].
\end{aligned}
\end{equation*}
Taking $r=\left(\frac{s}{t}\right)^{\frac{1}{1+\gamma}}\mu$, we obtain
\begin{equation*}
\begin{aligned}
\inf_{r>0}\left[tr+
\frac{sk^{2\gamma}}{4^{\gamma}(2\gamma+1)}r^{-\gamma}\right]
=&
t^{\frac{\gamma}{1+\gamma}}s^{\frac{1}{1+\gamma}}
\inf_{\mu>0}\left[\mu+
\frac{k^{2\gamma}}{4^{\gamma}(2\gamma+1)}\mu^{-\gamma}\right]
\\=&\frac{1+\gamma}{(2\gamma+1)^{\frac{1}{1+\gamma}}}
\left(\frac{k^{2}}{4\gamma}\right)^{\frac{\gamma}{1+\gamma}}
t^{\frac{\gamma}{1+\gamma}}s^{\frac{1}{1+\gamma}}.
\end{aligned}
\end{equation*}
Therefore,
\begin{equation}\label{2.4}
\begin{aligned}
\int_{\mathbb{R}^{k}}
\left(t|\nabla \psi|^{2}+s|x|^{2\gamma}|\psi|^{2}\right)\mathrm{d}x
\geqslant
\frac{1+\gamma}{(2\gamma+1)^{\frac{1}{1+\gamma}}}
\left(\frac{k^{2}}{4\gamma}\right)^{\frac{\gamma}{1+\gamma}}
t^{\frac{\gamma}{1+\gamma}}s^{\frac{1}{1+\gamma}}
\|\psi\|_{L^{2}(\mathbb{R}^{k})}^{2}.
\end{aligned}
\end{equation}
Clearly, \eqref{2.4} also holds for $\psi=0$ or $s=0$.

\noindent
{\bf Step 3}.
We now estimate a weighted $L^{2}$ norm of the Fourier transform
by the Grushin operator.
For each fixed $y$, the function $w(\cdot,y)$ belongs to $C_0^{\infty}(\mathbb{R}^{k})$.
Its support is contained in the $x$-projection of the support of $u$.
For $\gamma>0$, applying \eqref{2.4} to $w(\cdot,y)$ with
$t=\frac{\gamma l}{k+\gamma l}$ and $s=(1+\gamma)^{2}|y|^{2}$, we deduce
\begin{equation*}
\begin{aligned}
\|u\|_{D_{\gamma}^{1,2}(\mathbb{R}^{N})}^{2}
\geqslant&
\frac{1+\gamma}{(2\gamma+1)^{\frac{1}{1+\gamma}}}
\left(\frac{k^{2}}{4\gamma}\right)^{\frac{\gamma}{1+\gamma}}
\left(\frac{\gamma l}{k+\gamma l}\right)^{\frac{\gamma}{1+\gamma}}
(1+\gamma)^{\frac{2}{1+\gamma}}
\int_{\mathbb{R}^{l}}
|y|^{\frac{2}{1+\gamma}}
\left(
\int_{\mathbb{R}^{k}}
|w|^{2}\mathrm{d}x
\right)
\mathrm{d}y
\\&+
\frac{k}{k+\gamma l}
\int_{\mathbb{R}^{N}}|\nabla_xw|^{2}\mathrm{d}x
\mathrm{d}y.
\end{aligned}
\end{equation*}
A direct calculation gives
\begin{equation*}
\begin{aligned}
\frac{1+\gamma}{(2\gamma+1)^{\frac{1}{1+\gamma}}}
\left(\frac{k^{2}}{4\gamma}\right)^{\frac{\gamma}{1+\gamma}}
\left(\frac{\gamma l}{k+\gamma l}\right)^{\frac{\gamma}{1+\gamma}}
(1+\gamma)^{\frac{2}{1+\gamma}}
&=\left[
\frac{(1+\gamma)^{1+\gamma+2}}{2\gamma+1}
\left(\frac{k^{2}}{4\gamma}\frac{\gamma l}{k+\gamma l}\right)^{\gamma}
\right]^{\frac{1}{1+\gamma}}
\\&=\left[
\frac{(1+\gamma)^{\gamma+3}}{2\gamma+1}
\left(\frac{k^{2}l}{4(k+\gamma l)}\right)^{\gamma}
\right]^{\frac{1}{1+\gamma}}.
\end{aligned}
\end{equation*}
By Plancherel's identity in the $x$-variables, we have
\begin{equation*}
\begin{aligned}
\int_{\mathbb{R}^{N}}|\nabla_xw|^{2}\mathrm{d}x\mathrm{d}y
&=\int_{\mathbb{R}^{N}}|x|^{2}|v|^{2}\mathrm{d}x\mathrm{d}y,
\\
\int_{\mathbb{R}^{N}}|y|^{\frac{2}{1+\gamma}}|w|^{2}\mathrm{d}x\mathrm{d}y
&=\int_{\mathbb{R}^{N}}|y|^{\frac{2}{1+\gamma}}|v|^{2}\mathrm{d}x\mathrm{d}y.
\end{aligned}
\end{equation*}
Therefore,
\begin{equation}\label{2.5}
\begin{aligned}
&\|u\|_{D_{\gamma}^{1,2}(\mathbb{R}^{N})}^{2}\\
\geqslant&\int_{\mathbb{R}^{N}}\Bigg\{
\frac{k}{k+\gamma l}|x|^{2}
+\left[
\frac{(1+\gamma)^{\gamma+3}}{2\gamma+1}
\left(\frac{k^{2}l}{4(k+\gamma l)}\right)^{\gamma}
\right]^{\frac{1}{1+\gamma}}
|y|^{\frac{2}{1+\gamma}}
\Bigg\}|v|^{2}
\mathrm{d}x
\mathrm{d}y.
\end{aligned}
\end{equation}
For $\gamma=0$, \eqref{2.5} follows from Plancherel's identity.
Hence \eqref{2.5} holds for every $\gamma\geqslant0$.

\noindent
{\bf Step 4}.
In this step, we estimate the  $L^{2+\frac4{N_\gamma}}$ norm by Hausdorff-Young's
and H\"{o}lder's inequalities.
It is easy to see that
\begin{equation*}
\begin{aligned}
\frac{N_{\gamma}+4}{2(N_{\gamma}+2)}
+\frac{N_{\gamma}}{2(N_{\gamma}+2)}=1
\ \
\mathrm{and}
\ \
\frac{N_{\gamma}+4}{2(N_{\gamma}+2)}-\frac12
=\frac1{N_{\gamma}+2}.
\end{aligned}
\end{equation*}
By virtue of Hausdorff-Young's inequality for the inverse Fourier transform
with the unitary normalization, we have
\begin{equation}\label{2.6}
\begin{aligned}
\|u\|_{L^{2+\frac4{N_{\gamma}}}(\mathbb{R}^{N})}
\leqslant(2\pi)^{-\frac{N}{N_{\gamma}+2}}
\|v\|_{L^{\frac{2(N_{\gamma}+2)}{N_{\gamma}+4}}(\mathbb{R}^{N})}.
\end{aligned}
\end{equation}
For $t>0$, it follows from H\"{o}lder's inequality  that
\begin{equation*}
\begin{aligned}
&\int_{\mathbb{R}^{N}}|v|^{\frac{2(N_{\gamma}+2)}{N_{\gamma}+4}}
\mathrm{d}x
\mathrm{d}y\\
=&
\int_{\mathbb{R}^{N}}
\Bigg(
\Bigg\{
t+\frac{k}{k+\gamma l}|x|^{2}
+\left[
\frac{(1+\gamma)^{\gamma+3}}{2\gamma+1}
\left(\frac{k^{2}l}{4(k+\gamma l)}\right)^\gamma
\right]^{\frac1{1+\gamma}}
|y|^{\frac2{1+\gamma}}
\Bigg\}|v|^{2}
\Bigg)^{\frac{N_{\gamma}+2}{N_{\gamma}+4}}\\&\times
\Bigg\{
t+\frac{k}{k+\gamma l}|x|^{2}
+\left[
\frac{(1+\gamma)^{\gamma+3}}{2\gamma+1}
\left(\frac{k^{2}l}{4(k+\gamma l)}\right)^\gamma
\right]^{\frac1{1+\gamma}}
|y|^{\frac2{1+\gamma}}
\Bigg\}^{-\frac{N_{\gamma}+2}{N_{\gamma}+4}}
\mathrm{d}x
\mathrm{d}y\\\leqslant&
\Bigg(\!
\int_{\mathbb{R}^{N}}
\Bigg\{
t+\frac{k}{k+\gamma l}|x|^{2}
+\left[
\frac{(1+\gamma)^{\gamma+3}}{2\gamma+1}
\left(\frac{k^{2}l}{4(k+\gamma l)}\right)^\gamma
\right]^{\frac1{1+\gamma}}
|y|^{\frac2{1+\gamma}}
\Bigg\}|v|^{2}
\mathrm{d}x
\mathrm{d}y\!
\Bigg)^{\frac{N_{\gamma}+2}{N_{\gamma}+4}}\\&\times
\Bigg(\!
\int_{\mathbb{R}^{N}}
\Bigg\{
t+\frac{k}{k+\gamma l}|x|^{2}
+\left[
\frac{(1+\gamma)^{\gamma+3}}{2\gamma+1}
\left(\frac{k^{2}l}{4(k+\gamma l)}\right)^\gamma
\right]^{\frac1{1+\gamma}}
|y|^{\frac2{1+\gamma}}
\Bigg\}^{-\frac{N_{\gamma}+2}{2}}
\mathrm{d}x
\mathrm{d}y\!
\Bigg)^{\frac2{N_{\gamma}+4}},
\end{aligned}
\end{equation*}
which gives
\begin{equation}\label{2.7}
\begin{aligned}
&\|v\|_{L^{\frac{2(N_\gamma+2)}{N_\gamma+4}}(\mathbb{R}^N)}^{2+\frac4{N_\gamma}}\\\leqslant&
\Bigg(\int_{\mathbb{R}^N}\Bigg\{
t+\frac{k}{k+\gamma l}|x|^{2}
+\left[
\frac{(1+\gamma)^{\gamma+3}}{2\gamma+1}
\left(\frac{k^{2}l}{4(k+\gamma l)}\right)^\gamma
\right]^{\frac1{1+\gamma}}|y|^{\frac2{1+\gamma}}
\Bigg\}|v|^{2}\mathrm{d}x\mathrm{d}y\Bigg)^{\frac{N_\gamma+2}{N_\gamma}}\\&\times
\Bigg(\int_{\mathbb{R}^N}\Bigg\{
t+\frac{k}{k+\gamma l}|x|^{2}
+
\left[
\frac{(1+\gamma)^{\gamma+3}}{2\gamma+1}
\left(\frac{k^{2}l}{4(k+\gamma l)}\right)^\gamma
\right]^{\frac1{1+\gamma}}|y|^{\frac2{1+\gamma}}
\Bigg\}^{-\frac{N_\gamma+2}{2}}\mathrm{d}x\mathrm{d}y\Bigg)^{\frac2{N_\gamma}}.
\end{aligned}
\end{equation}
\noindent
{\bf Step 5}.
In this step, we compute the second integral on the right-hand side of \eqref{2.7}.

For $s>0$, by polar coordinates in $\mathbb{R}^{k}$ with
$r_x=|x|$ and $\mu_x=s\frac{k}{k+\gamma l}r_x^{2}$, we have
\begin{equation*}
\begin{aligned}
\int_{\mathbb{R}^{k}}
e^{-s\frac{k}{k+\gamma l}|x|^{2}}\mathrm{d}x
&=\frac{2\pi^{\frac k2}}{\Gamma\left(\frac k2\right)}
\int_0^{\infty}r_x^{k-1}e^{-s\frac{k}{k+\gamma l}r_x^{2}}\mathrm{d}r_x
\\
&=\frac{2\pi^{\frac k2}}{\Gamma\left(\frac k2\right)}
\frac12s^{-\frac k2}
\left(\frac{k+\gamma l}{k}\right)^{\frac k2}
\int_0^{\infty}\mu_x^{\frac k2-1}e^{-\mu_x}\mathrm{d}\mu_x
\\
&=\pi^{\frac k2}s^{-\frac k2}
\left(\frac{k+\gamma l}{k}\right)^{\frac k2}.
\end{aligned}
\end{equation*}
Set $r_y=|y|$ and
\begin{equation*}
\begin{aligned}
\mu_y=s\left[
\frac{(1+\gamma)^{\gamma+3}}{2\gamma+1}
\left(\frac{k^{2}l}{4(k+\gamma l)}\right)^{\gamma}
\right]^{\frac{1}{1+\gamma}}r_y^{\frac{2}{1+\gamma}}.
\end{aligned}
\end{equation*}
Then we have
\begin{equation*}
\begin{aligned}
r_y
&=s^{-\frac{1+\gamma}{2}}\mu_y^{\frac{1+\gamma}{2}}
\left[\frac{2\gamma+1}{(1+\gamma)^{\gamma+3}}
\left(\frac{4(k+\gamma l)}{k^{2}l}\right)^{\gamma}\right]^{\frac12},
\\
r_y^{l-1}\mathrm{d}r_y
&=\frac{1+\gamma}{2}s^{-\frac{(1+\gamma)l}{2}}
\left[\frac{2\gamma+1}{(1+\gamma)^{\gamma+3}}
\left(\frac{4(k+\gamma l)}{k^{2}l}\right)^{\gamma}\right]^{\frac l2}
\mu_y^{\frac{(1+\gamma)l}{2}-1}\mathrm{d}\mu_y.
\end{aligned}
\end{equation*}
By polar coordinates, we obtain
\begin{equation*}
\begin{aligned}
&\int_{\mathbb{R}^{l}}
\exp\Bigg\{-s\left[
\frac{(1+\gamma)^{\gamma+3}}{2\gamma+1}
\left(\frac{k^{2}l}{4(k+\gamma l)}\right)^{\gamma}
\right]^{\frac{1}{1+\gamma}}
|y|^{\frac{2}{1+\gamma}}\Bigg\}
\mathrm{d}y
\\&=
\frac{2\pi^{\frac l2}}{\Gamma\left(\frac l2\right)}
\frac{1+\gamma}{2}s^{-\frac{(1+\gamma)l}{2}}
\left[\frac{2\gamma+1}{(1+\gamma)^{\gamma+3}}
\left(\frac{4(k+\gamma l)}{k^{2}l}\right)^{\gamma}\right]^{\frac l2}
\int_{0}^{\infty}
\mu_y^{\frac{(1+\gamma)l}{2}-1}e^{-\mu_y}
\mathrm{d}\mu_y
\\&=\frac{\pi^{\frac l2}(1+\gamma)
\Gamma\left(\frac{(1+\gamma)l}{2}\right)}
{\Gamma\left(\frac l2\right)}s^{-\frac{(1+\gamma)l}{2}}
\left[\frac{2\gamma+1}{(1+\gamma)^{\gamma+3}}
\left(\frac{4(k+\gamma l)}{k^{2}l}\right)^{\gamma}\right]^{\frac l2}.
\end{aligned}
\end{equation*}
For $t>0$, by the Gamma integral and Tonelli's theorem, we have
\begin{equation*}
\begin{aligned}
&\int_{\mathbb{R}^{N}}\Bigg\{
t+\frac{k}{k+\gamma l}|x|^{2}
+\left[
\frac{(1+\gamma)^{\gamma+3}}{2\gamma+1}
\left(\frac{k^{2}l}{4(k+\gamma l)}\right)^{\gamma}
\right]^{\frac{1}{1+\gamma}}|y|^{\frac{2}{1+\gamma}}
\Bigg\}^{-\frac{N_{\gamma}+2}{2}}\mathrm{d}x\mathrm{d}y
\\=&
\frac{1}{\Gamma\left(\frac{N_{\gamma}+2}{2}\right)}
\int_0^{\infty}s^{\frac{N_{\gamma}}2}e^{-st}
\left(\int_{\mathbb{R}^{k}}
e^{-s\frac{k}{k+\gamma l}|x|^{2}}\mathrm{d}x\right)
\\&\times
\left(\int_{\mathbb{R}^{l}}
\exp\Bigg\{-s\left[
\frac{(1+\gamma)^{\gamma+3}}{2\gamma+1}
\left(\frac{k^{2}l}{4(k+\gamma l)}\right)^{\gamma}
\right]^{\frac{1}{1+\gamma}}|y|^{\frac{2}{1+\gamma}}
\Bigg\}\mathrm{d}y\right)\mathrm{d}s.
\end{aligned}
\end{equation*}
Since $N_\gamma=k+(1+\gamma)l$, we have
\begin{equation*}
\begin{aligned}
\frac{N_{\gamma}}2-\frac k2-\frac{(1+\gamma)l}{2}=0
\ \
\mathrm{and}
\ \
\int_0^{\infty}e^{-st}\mathrm{d}s=\frac1t.
\end{aligned}
\end{equation*}
It follows that
\begin{equation}\label{2.8}
\begin{aligned}
&\int_{\mathbb{R}^{N}}\Bigg\{
t+\frac{k}{k+\gamma l}|x|^{2}
+\left[
\frac{(1+\gamma)^{\gamma+3}}{2\gamma+1}
\left(\frac{k^{2}l}{4(k+\gamma l)}\right)^{\gamma}
\right]^{\frac{1}{1+\gamma}}|y|^{\frac{2}{1+\gamma}}
\Bigg\}^{-\frac{N_{\gamma}+2}{2}}\mathrm{d}x\mathrm{d}y
\\&=\frac{\pi^{\frac N2}(1+\gamma)
\Gamma\left(\frac{(1+\gamma)l}{2}\right)}
{t\Gamma\left(\frac l2\right)
\Gamma\left(\frac{N_{\gamma}+2}{2}\right)}
\left(\frac{k+\gamma l}{k}\right)^{\frac k2}
\left[
\frac{2\gamma+1}{(1+\gamma)^{\gamma+3}}
\left(\frac{4(k+\gamma l)}{k^{2}l}\right)^{\gamma}
\right]^{\frac l2}.
\end{aligned}
\end{equation}

\noindent
{\bf Step 6}.
In the last step, we combine the above estimates.
Note that
\begin{equation*}
\begin{aligned}
(2\pi)^{-\frac{2N}{N_{\gamma}}}
\left(\pi^{\frac N2}\right)^{\frac2{N_{\gamma}}}
=2^{-\frac{2N}{N_{\gamma}}}\pi^{-\frac N{N_{\gamma}}}
=
\left(\frac{1}{2^N\pi^{\frac N2}}\right)^{\frac2{N_{\gamma}}}.
\end{aligned}
\end{equation*}
By \eqref{2.5}-\eqref{2.8} and Plancherel's identity, we obtain, for every $t>0$,
\begin{equation}\label{2.9}
\begin{aligned}
\|u\|_{L^{2+\frac4{N_{\gamma}}}(\mathbb{R}^{N})}
^{2+\frac4{N_{\gamma}}}
\leqslant&
\left[\frac{(1+\gamma)\Gamma\left(\frac{(1+\gamma)l}{2}\right)}
{2^N\pi^{\frac N2}\Gamma\left(\frac l2\right)
\Gamma\left(\frac{N_{\gamma}+2}{2}\right)}\right]^{\frac2{N_{\gamma}}}
\left(\frac{k+\gamma l}{k}\right)^{\frac k{N_{\gamma}}}\\
&\times
\left[\frac{2\gamma+1}{(1+\gamma)^{\gamma+3}}
\left(\frac{4(k+\gamma l)}{k^{2}l}\right)^{\gamma}\right]^{\frac l{N_{\gamma}}}\\
&\times
t^{-\frac2{N_{\gamma}}}
\left(\|u\|_{D_{\gamma}^{1,2}(\mathbb{R}^{N})}^{2}
+t\|u\|_{L^{2}(\mathbb{R}^{N})}^{2}\right)^{\frac{N_{\gamma}+2}{N_{\gamma}}}.
\end{aligned}
\end{equation}
Finally, we minimize the right-hand side in \eqref{2.9} to obtain an explicit upper bound for $C_{N_\gamma,2+\frac4{N_\gamma}}$.
For $u\not\equiv0$, both $\|u\|_{L^{2}(\mathbb{R}^{N})}$ and
$\|u\|_{D_{\gamma}^{1,2}(\mathbb{R}^{N})}$ are positive.
Differentiating with respect to $t$, we have
\begin{equation*}
\begin{aligned}
&\frac{\mathrm{d}}{\mathrm{d}t}
\log\Bigg[
t^{-\frac2{N_{\gamma}}}
\left(\|u\|_{D_{\gamma}^{1,2}(\mathbb{R}^{N})}^{2}
+t\|u\|_{L^{2}(\mathbb{R}^{N})}^{2}\right)^{\frac{N_{\gamma}+2}{N_{\gamma}}}
\Bigg]\\
=&
-\frac2{N_{\gamma}t}
+\frac{N_{\gamma}+2}{N_{\gamma}}
\frac{\|u\|_{L^{2}(\mathbb{R}^{N})}^{2}}
{\|u\|_{D_{\gamma}^{1,2}(\mathbb{R}^{N})}^{2}
+t\|u\|_{L^{2}(\mathbb{R}^{N})}^{2}}\\
=&
\frac{N_{\gamma}t\|u\|_{L^{2}(\mathbb{R}^{N})}^{2}
-2\|u\|_{D_{\gamma}^{1,2}(\mathbb{R}^{N})}^{2}}
{N_{\gamma}t\left(\|u\|_{D_{\gamma}^{1,2}(\mathbb{R}^{N})}^{2}
+t\|u\|_{L^{2}(\mathbb{R}^{N})}^{2}\right)}.
\end{aligned}
\end{equation*}
The above derivative changes sign from negative to positive at
\begin{equation*}
\begin{aligned}
t=\frac{2\|u\|_{D_{\gamma}^{1,2}(\mathbb{R}^{N})}^{2}}
{N_{\gamma}\|u\|_{L^{2}(\mathbb{R}^{N})}^{2}}.
\end{aligned}
\end{equation*}
Therefore,
\begin{equation}\label{2.10}
\begin{aligned}
&\inf_{t>0}t^{-\frac2{N_{\gamma}}}
\left(\|u\|_{D_{\gamma}^{1,2}(\mathbb{R}^{N})}^{2}
+t\|u\|_{L^{2}(\mathbb{R}^{N})}^{2}\right)^{\frac{N_{\gamma}+2}{N_{\gamma}}}\\
=&\left(
\frac{N_{\gamma}\|u\|_{L^{2}(\mathbb{R}^{N})}^{2}}
{2\|u\|_{D_{\gamma}^{1,2}(\mathbb{R}^{N})}^{2}}
\right)^{\frac2{N_{\gamma}}}
\left(\frac{N_{\gamma}+2}{N_{\gamma}}
\|u\|_{D_{\gamma}^{1,2}(\mathbb{R}^{N})}^{2}
\right)^{\frac{N_{\gamma}+2}{N_{\gamma}}}\\
=&\left(\frac{N_{\gamma}}2\right)^{\frac2{N_{\gamma}}}
\left(\frac{N_{\gamma}+2}{N_{\gamma}}\right)^{\frac{N_{\gamma}+2}{N_{\gamma}}}
\|u\|_{D_{\gamma}^{1,2}(\mathbb{R}^{N})}^{2}
\|u\|_{L^{2}(\mathbb{R}^{N})}^{\frac4{N_{\gamma}}}\\
=&
\frac{N_{\gamma}+2}{N_{\gamma}}
\left(\frac{N_{\gamma}+2}{2}\right)^{\frac2{N_{\gamma}}}
\|u\|_{D_{\gamma}^{1,2}(\mathbb{R}^{N})}^{2}
\|u\|_{L^{2}(\mathbb{R}^{N})}^{\frac4{N_{\gamma}}}.
\end{aligned}
\end{equation}
Note that
\begin{equation*}
\begin{aligned}
\left(\frac{N_{\gamma}}2\right)^{\frac2{N_{\gamma}}}
\left(\frac{N_{\gamma}+2}{N_{\gamma}}\right)^{\frac{N_{\gamma}+2}{N_{\gamma}}}
=\frac{N_{\gamma}+2}{N_{\gamma}}
\left(\frac{N_{\gamma}+2}{2}\right)^{\frac2{N_{\gamma}}}.
\end{aligned}
\end{equation*}
Moreover, we have
\begin{equation*}
\begin{aligned}
&\left(\frac{N_{\gamma}+2}{2}\right)^{\frac2{N_{\gamma}}}
\left[\frac{(1+\gamma)\Gamma\left(\frac{(1+\gamma)l}{2}\right)}
{2^N\pi^{\frac N2}\Gamma\left(\frac l2\right)
\Gamma\left(\frac{N_{\gamma}+2}{2}\right)}\right]^{\frac2{N_{\gamma}}}
=
\left[\frac{(N_{\gamma}+2)(1+\gamma)
\Gamma\left(\frac{(1+\gamma)l}{2}\right)}
{2^{N+1}\pi^{\frac N2}\Gamma\left(\frac l2\right)
\Gamma\left(\frac{N_{\gamma}+2}{2}\right)}\right]^{\frac2{N_{\gamma}}}.
\end{aligned}
\end{equation*}
By \eqref{2.9}-\eqref{2.10}, we obtain
\begin{equation*}
\begin{aligned}
\|u\|_{L^{2+\frac4{N_{\gamma}}}(\mathbb{R}^{N})}^{2+\frac4{N_{\gamma}}}
\leqslant&
\frac{N_{\gamma}+2}{N_{\gamma}}
\left(\frac{k+\gamma l}{k}\right)^{\frac k{N_{\gamma}}}
\left[\frac{(N_{\gamma}+2)(1+\gamma)
\Gamma\left(\frac{(1+\gamma)l}{2}\right)}
{2^{N+1}\pi^{\frac N2}\Gamma\left(\frac l2\right)
\Gamma\left(\frac{N_{\gamma}+2}{2}\right)}\right]^{\frac2{N_{\gamma}}}\\
&\times
\left[\frac{2\gamma+1}{(1+\gamma)^{\gamma+3}}
\left(\frac{4(k+\gamma l)}{k^{2}l}\right)^{\gamma}\right]^{\frac l{N_{\gamma}}}
\|u\|_{D_{\gamma}^{1,2}(\mathbb{R}^{N})}^{2}
\|u\|_{L^{2}(\mathbb{R}^{N})}^{\frac4{N_{\gamma}}}.
\end{aligned}
\end{equation*}
Clearly, the above inequality holds for $u=0$.
By density and the continuous embedding into $L^{2+\frac4{N_\gamma}}(\mathbb{R}^{N})$, the above inequality holds for every $u\in H_{\gamma}^{1,2}(\mathbb{R}^{N})$.
This proves \eqref{2.1} and gives an upper bound for the constant $C_{N_{\gamma},2+\frac{4}{N_{\gamma}}}$.

When $\gamma=0$, \eqref{2.1} reduces to
\begin{equation*}
\begin{aligned}
C_{N,2+\frac4N}
\leqslant
\frac{N+2}{N}
\left[
\frac{N+2}{2^{N+1}\pi^{\frac N2}
\Gamma\left(\frac{N+2}{2}\right)}
\right]^{\frac2N}.
\end{aligned}
\end{equation*}
The proof is completed.
\end{proof}

\begin{lemma}\label{Lemma2.4}
Let $N=k+l$,
$k,l\geqslant 1$,
$\gamma\geqslant0$
and $p,q$ satisfy \eqref{H}.
Then for $u\in H^{1,2}_{\gamma}(\mathbb{R}^{N})$,
we have
\begin{equation*}
\begin{aligned}
&\int_{\mathbb{R}^{N}}
|u|^{q}
\mathrm{d}z\\
\leqslant&
C
\left(
\int_{\mathbb{R}^{N}}
|u|^{p}
\mathrm{d}z
\right)
^{\frac{2N_{\gamma}+4-qN_{\gamma}}{2N_{\gamma}+4-pN_{\gamma}}}
\left(
\int_{\mathbb{R}^{N}}
|\nabla_{\gamma}u|^{2}
\mathrm{d}z
\right)
^{\frac{N_{\gamma}(q-p)}{2N_{\gamma}+4-pN_{\gamma}}}
\left(
\int_{\mathbb{R}^{N}}
|u|^{2}
\mathrm{d}z
\right)
^{\frac{2(q-p)}{2N_{\gamma}+4-pN_{\gamma}}},
\end{aligned}
\end{equation*}
where the constant $C>0$ is independent of $u$.
\end{lemma}

\begin{proof}
Applying the H\"{o}lder inequality and Lemma \ref{Lemma2.2},
we have that
\begin{equation*}
\begin{aligned}
&\int_{\mathbb{R}^{N}}
|u|^{q}
\mathrm{d}z\\
\leqslant&
\left(
\int_{\mathbb{R}^{N}}
|u|^{p}
\mathrm{d}z
\right)
^{\frac{2+\frac{4}{N_{\gamma}}-q}{2+\frac{4}{N_{\gamma}}-p}}
\left(
\int_{\mathbb{R}^{N}}
|u|^{2+\frac{4}{N_{\gamma}}}
\mathrm{d}z
\right)
^{\frac{q-p}{2+\frac{4}{N_{\gamma}}-p}}\\
\leqslant&
C
\left(
\int_{\mathbb{R}^{N}}
|u|^{p}
\mathrm{d}z
\right)
^{\frac{2N_{\gamma}+4-qN_{\gamma}}{2N_{\gamma}+4-pN_{\gamma}}}
\left(
\int_{\mathbb{R}^{N}}
|\nabla_{\gamma}u|^{2}
\mathrm{d}z
\right)
^{\frac{N_{\gamma}(q-p)}{2N_{\gamma}+4-pN_{\gamma}}}
\left(
\int_{\mathbb{R}^{N}}
|u|^{2}
\mathrm{d}z
\right)
^{\frac{2(q-p)}{2N_{\gamma}+4-pN_{\gamma}}}.
\end{aligned}
\end{equation*}
\end{proof}

\begin{lemma}
\label{Lemma2.5}
{\rm\cite[Proposition 5.5]
{YuSB-YangC-TangCL2026JGA}}
Let $N=k+l$,
$k,l\geqslant 1$,
$\gamma\geqslant0$,
$\alpha=1$
and $p,q$ satisfy \eqref{H}.
Suppose that $u\in H_{\gamma}^{1,2}(\mathbb{R}^N)$ is a weak solution of equation \eqref{G}. Then u satisfies the following Nehari identity $N(u)=0$ and the Pohozaev identity $P(u)=0$, where
\begin{equation*}
\begin{aligned}
N(u):=
\|u\|_{D_{\gamma}^{1,2}(\mathbb{R}^{N})}^{2}
+
\omega
\|u\|_{L^{2}(\mathbb{R}^{N})}^{2}
+
\|u\|_{L^{p}(\mathbb{R}^{N})}^{p}
-
\|u\|_{L^{q}(\mathbb{R}^{N})}^{q}
\end{aligned}
\end{equation*}
and
\begin{equation*}
\begin{aligned}
P(u):=
\frac{N_{\gamma}-2}{2}
\|u\|_{D_{\gamma}^{1,2}(\mathbb{R}^{N})}^{2}
+
\frac{\omega N_{\gamma}}{2}
\|u\|_{L^{2}(\mathbb{R}^{N})}^{2}
+
\frac{N_{\gamma}}{p}
\|u\|_{L^{p}(\mathbb{R}^{N})}^{p}
-
\frac{N_{\gamma}}{q}
\|u\|_{L^{q}(\mathbb{R}^{N})}^{q}.
\end{aligned}
\end{equation*}

\end{lemma}

By using Lemma \ref{Lemma2.5},
we know that
\begin{equation}\label{2.11}
\begin{aligned}
0=J(u):=&2N(u)-(p-2)P(u)
\\=&
\frac{2p-N_{\gamma}(p-2)}{2}
\|u\|_{D_{\gamma}^{1,2}(\mathbb{R}^{N})}^{2}
+
\frac{4-N_{\gamma}(p-2)}{2}
\omega\|u\|_{L^{2}(\mathbb{R}^{N})}^{2}
\\&+
\frac{2p-N_{\gamma}(p-2)}{p}
\|u\|_{L^{p}(\mathbb{R}^{N})}^{p}
-
\frac{2q-N_{\gamma}(p-2)}{q}
\|u\|_{L^{q}(\mathbb{R}^{N})}^{q}.
\end{aligned}
\end{equation}

\begin{lemma}\label{Lemma2.6}
Let $N=k+l$,
$k,l\geqslant 1$,
$\gamma\geqslant0$,
$\alpha=1$
and $2<p<q<2+\frac{4}{N_{\gamma}}$.
Suppose that $u\in H_{\gamma}^{1,2}(\mathbb{R}^N)$ is a weak solution of equation \eqref{G}. Then it holds that
\begin{equation*}
\begin{aligned}
I_{\omega}(u)-I_{\omega}(u^{t})
\geqslant&
\frac{f_{1}(t)}{2[2p-N_{\gamma}(p-2)]}
\omega
\|u\|_{L^{2}(\mathbb{R}^{N})}^{2}
+
\frac{f_{2}(t)}{q[2p-N_{\gamma}(p-2)]}
\|u\|_{L^{q}(\mathbb{R}^{N})}^{q}\\
\geqslant&0,
\end{aligned}
\end{equation*}
where $u^{t}(z)=t^{2}u(t^{p-2}x,t^{(p-2)(1+\gamma)}y)$, $t>0$ and
\begin{equation*}
\begin{aligned}
\begin{cases}
f_{1}(t):=
2(p-2)-[2p-N_{\gamma}(p-2)]t^{4-N_{\gamma}(p-2)}+[4-N_{\gamma}(p-2)]t^{2p-N_{\gamma}(p-2)},\\
f_{2}(t):=2(q-p)-[2q-N_{\gamma}(p-2)]t^{2p-N_{\gamma}(p-2)}
+[2p-N_{\gamma}(p-2)]t^{2q-N_{\gamma}(p-2)}.
\end{cases}
\end{aligned}
\end{equation*}
\end{lemma}
\begin{proof}
By \eqref{2.11}, we know that
\begin{equation*}
\begin{aligned}
&I_{\omega}(u)-I_{\omega}(u^{t})\\
=&
\frac{1-t^{2p-N_{\gamma}(p-2)}}{2p-N_{\gamma}(p-2)}J(u)\\
&+
\frac{2(p-2)-[2p-N_{\gamma}(p-2)]t^{4-N_{\gamma}(p-2)}+[4-N_{\gamma}(p-2)]t^{2p-N_{\gamma}(p-2)}}{2[2p-N_{\gamma}(p-2)]}
\omega
\|u\|_{L^{2}(\mathbb{R}^{N})}^{2}
\\&+
\frac{2(q-p)-[2q-N_{\gamma}(p-2)]t^{2p-N_{\gamma}(p-2)}
+[2p-N_{\gamma}(p-2)]t^{2q-N_{\gamma}(p-2)}}{q[2p-N_{\gamma}(p-2)]}
\|u\|_{L^{q}(\mathbb{R}^{N})}^{q}.
\end{aligned}
\end{equation*}
We set
\begin{equation*}
\begin{aligned}
\begin{cases}
f_{1}(t):=
2(p-2)-[2p-N_{\gamma}(p-2)]t^{4-N_{\gamma}(p-2)}+[4-N_{\gamma}(p-2)]t^{2p-N_{\gamma}(p-2)},\\
f_{2}(t):=2(q-p)-[2q-N_{\gamma}(p-2)]t^{2p-N_{\gamma}(p-2)}
+[2p-N_{\gamma}(p-2)]t^{2q-N_{\gamma}(p-2)}.
\end{cases}
\end{aligned}
\end{equation*}
Then 
\begin{equation*}
\begin{aligned}
\begin{cases}
f_{1}'(t)
=
[2N_{\gamma}-(N_{\gamma}-2)p][4-N_{\gamma}(p-2)]t^{2p-N_{\gamma}(p-2)-1}(1-t^{2(2-p)}),\\
f_{2}'(t)
=
[2N_{\gamma}-(N_{\gamma}-2)p][2q-N_{\gamma}(p-2)]t^{2q-N_{\gamma}(p-2)-1}(1-t^{2(p-q)}).
\end{cases}
\end{aligned}
\end{equation*}
If $2<p<q<2+\frac{4}{N_{\gamma}}$, then $f_{1}(1)=f_{2}(1)=0$ is global minimum on $(0,\infty)$.
\end{proof}

\section{Proof of Theorem \ref{Theorem1.1}}
\label{section3}
In this section,
we prove Theorem \ref{Theorem1.1}, i.e., the attainability of Gagliardo-Nirenberg inequality \eqref{1.2}
and the two-sided estimates for its best constant.

We set the Weinstein type functional
\begin{equation*}
\begin{aligned}
W_{N_{\gamma},p,q}(u)=
\frac{
\left(
\|u\|_{L^{2}(\mathbb{R}^N)}^{2}
\right)
^{\frac{2(q-p)}{2N_{\gamma}+4-pN_{\gamma}}}
\left(
\|u\|_{L^{p}(\mathbb{R}^N)}^{p}
\right)
^{\frac{2N_{\gamma}+4-qN_{\gamma}}{2N_{\gamma}+4-pN_{\gamma}}}
\left(
\|u\|_{D^{1,2}_{\gamma}(\mathbb{R}^N)}^{2}
\right)
^{\frac{N_{\gamma}(q-p)}{2N_{\gamma}+4-pN_{\gamma}}}
}
{\|u\|_{L^{q}(\mathbb{R}^N)}^{q}}.
\end{aligned}
\end{equation*}
We define
\begin{equation*}
\begin{aligned}
C_{N_{\gamma},p,q}^{-1}:=
\inf_{u\in H_{\gamma,rad_{x}}^{1,2}(\mathbb{R}^N)\backslash\{0\}}W_{N_{\gamma},p,q}(u).
\end{aligned}
\end{equation*}
Our aim is to prove the attainability of $C_{N_{\gamma},p,q}^{-1}$
and obtain the upper and lower bounds for $C_{N_{\gamma},p,q}$.
\begin{lemma}\label{Lemma3.1}
Assume that $\gamma\geqslant0$ and $p,q$ satisfy \eqref{H}.
Then, there exists a bounded sequence $\{u_n\}\subset H_{\gamma,rad_{x}}^{1,2}(\mathbb{R}^N)$ such that
\begin{equation*}
\begin{aligned}
\lim_{n\to\infty}
W_{N_{\gamma},p,q}(u_n)=C_{N_{\gamma},p,q}^{-1},
\ \
\|u_n\|_{D_{\gamma}^{1,2}(\mathbb{R}^{N})}
=
\|u_n\|_{L^{p}(\mathbb{R}^{N})}=1,
\end{aligned}
\end{equation*}
and
\begin{equation*}
\begin{aligned}
C_1\leqslant \|u_n\|_{L^{2}(\mathbb{R}^{N})}
\leqslant C_2
\ \
and
\ \
C_3\leqslant
\|u_n\|_{L^{q}(\mathbb{R}^{N})},
\end{aligned}
\end{equation*}
for some constants $C_1,C_2,C_3>0$.
\end{lemma}

\begin{proof}
\textbf{Step 1.}
Let $\{v_n\}\subset H_{\gamma,rad_{x}}^{1,2}(\mathbb{R}^N)\backslash\{0\}$ be any minimizing sequence of $C_{N_{\gamma},p,q}^{-1}$ and put
\begin{equation*}
\begin{aligned}
\tilde{u}_n(z):=
\lambda_{n}
v_n(\mu_nx,\mu_n^{1+\gamma}y),
\end{aligned}
\end{equation*}
where
\begin{equation*}
\begin{aligned}
\begin{cases}
\mu_n
=
\left[
\frac{\int_{\mathbb{R}^{N}}
|v_n|^{p}
\mathrm{d}z}
{\left(
\int |\nabla_\gamma v_n|^{2} \mathrm{d}x\mathrm{d}y
\right)^{\frac{p}{2}}}
\right]
^{\frac{2}{2p-(p-2)[k+l(1+\gamma)]}},
\\
\lambda_{n}
=
\left[
\frac{\int_{\mathbb{R}^{N}}
|v_n|^{p}
\mathrm{d}z}
{\left(
\int |\nabla_\gamma v_n|^{2} \mathrm{d}x\mathrm{d}y
\right)^{\frac{p}{2}}}
\right]
^{\frac{2[k+l(1+\gamma)]}{p\{2p-(p-2)[k+l(1+\gamma)]\}}
}
\left(
\int_{\mathbb{R}^{N}}
|v_n|^{p}
\mathrm{d}z
\right)
^{-\frac{1}{p}}.
\end{cases}
\end{aligned}
\end{equation*}
Then we have that
\begin{equation*}
\begin{aligned}
&\int |\nabla_\gamma\tilde{u}_n|^{2} \mathrm{d}x\mathrm{d}y
=
\lambda_n^{2} \mu_n^{2-k-l(1+\gamma)} 
\int|\nabla_\gamma v_n|^{2} \mathrm{d}x\mathrm{d}y
=1,\\
&\|\tilde{u}_n\|_{L^{p}(\mathbb{R}^{N})}^p
=
\lambda_{n}^p
\mu_n^{-k-l(1+\gamma)}
\int_{\mathbb{R}^{N}}
|v_n|^{p}
\mathrm{d}z=1,\\
&\|\tilde{u}_n\|_{L^{2}(\mathbb{R}^{N})}^{2}=
\lambda_{n}^{2}
\mu_n^{-k-l(1+\gamma)}
\int_{\mathbb{R}^{N}}
|v_n|^{2}
\mathrm{d}z,\\
&\|\tilde{u}_n\|_{L^{q}(\mathbb{R}^{N})}^q=
\lambda_{n}^q
\mu_n^{-k-l(1+\gamma)}
\int_{\mathbb{R}^{N}}
|v_n|^{q}
\mathrm{d}z.
\end{aligned}
\end{equation*}
It is easy to see that
\begin{equation*}
\begin{aligned}
W_{N_{\gamma},p,q}(\tilde{u}_n)
=
W_{N_{\gamma},p,q}(v_n)=C_{N_{\gamma},p,q}^{-1}+o(1),
\end{aligned}
\end{equation*}
which yields that $\{\tilde{u}_n\}\subset H_{\gamma,rad_{x}}^{1,2}(\mathbb{R}^N)\backslash\{0\}$ is also a minimizing sequence.

By virtue of
$\int_{\mathbb{R}^{N}}
|\nabla_\gamma\tilde{u}_n|^{2} \mathrm{d}z
=\int_{\mathbb{R}^{N}}
|\tilde{u}_n|^p \mathrm{d}z=1$,
we get that
\begin{equation}\label{3.1}
\begin{aligned}
C_{N_{\gamma},p,q}^{-1}+o(1)=
\frac{
\left(
\|\tilde{u}_n\|_{L^{2}(\mathbb{R}^N)}^{2}
\right)
^{\frac{2(q-p)}{2N_{\gamma}+4-pN_{\gamma}}}}
{\|\tilde{u}_n\|_{L^{q}(\mathbb{R}^N)}^{q}}.
\end{aligned}
\end{equation}

\textbf{Step 2.}
Next we claim that there exists $C>0$ such that $\|\tilde{u}_n\|_{L^{2}(\mathbb{R}^N)}\geqslant C$.
It follows from Lemma \ref{Lemma2.2} and $\int_{\mathbb{R}^{N}}
|\nabla_\gamma\tilde{u}_n|^{2} \mathrm{d}z=1$ that
\begin{equation*}
\begin{aligned}
C_{N_{\gamma},p,q}^{-1}+o(1)
\geqslant& C
\frac{
\left(
\|\tilde{u}_n\|_{L^{2}(\mathbb{R}^N)}^{2}
\right)
^{\frac{2(q-p)}{2N_{\gamma}+4-pN_{\gamma}}}}
{\left(
\int_{\mathbb{R}^{N}}
|\nabla_{\gamma}\tilde{u}_n|^{2}
\mathrm{d}z
\right)
^{\frac{N_{\gamma}(q-2)}{4}}
\left(
\|\tilde{u}_n\|_{L^{2}(\mathbb{R}^N)}^{2}
\right)
^{\frac{q}{2}-\frac{N_{\gamma}(q-2)}{4}}}\\
=& C
\left(
\|\tilde{u}_n\|_{L^{2}(\mathbb{R}^N)}^{2}
\right)
^{\frac{2(q-p)}{2N_{\gamma}+4-pN_{\gamma}}-\frac{q}{2}+\frac{N_{\gamma}(q-2)}{4}}.
\end{aligned}
\end{equation*}
From $\frac{2(q-p)}{2N_{\gamma}+4-pN_{\gamma}}<\frac{q}{2}-\frac{N_{\gamma}(q-2)}{4}$,
we know that $\|\tilde{u}_n\|_{L^{2}(\mathbb{R}^N)}\geqslant C$.

\textbf{Step 3.}
Next we claim that there exists $C>0$ such that $\|\tilde{u}_n\|_{L^{q}(\mathbb{R}^N)}\geqslant C$.
It follows from Lemma \ref{Lemma2.2}, $\int_{\mathbb{R}^{N}}
|\nabla_\gamma\tilde{u}_n|^{2} \mathrm{d}z=1$ and $\frac{2(q-p)}{2N_{\gamma}+4-pN_{\gamma}}<\frac{q}{2}-\frac{N_{\gamma}(q-2)}{4}$ that
\begin{equation*}
\begin{aligned}
C_{N_{\gamma},p,q}^{-1}+o(1)
=&
\frac{
\left(
\|\tilde{u}_n\|_{L^{2}(\mathbb{R}^N)}^{2}
\right)
^{\frac{2(q-p)}{2N_{\gamma}+4-pN_{\gamma}}}}
{\|\tilde{u}_n\|_{L^{q}(\mathbb{R}^N)}^{q}}\\
=&
\frac
{\left[
\left(
\int_{\mathbb{R}^{N}}
|\nabla_{\gamma}\tilde{u}_n|^{2}
\mathrm{d}z
\right)
^{\frac{N_{\gamma}(q-2)}{4}}
\left(
\|\tilde{u}_n\|_{L^{2}(\mathbb{R}^N)}^{2}
\right)
^{\frac{q}{2}-\frac{N_{\gamma}(q-2)}{4}}
\right]
^{
\frac{\frac{2(q-p)}{2N_{\gamma}+4-pN_{\gamma}}}{\frac{q}{2}-\frac{N_{\gamma}(q-2)}{4}}
}
}
{\|\tilde{u}_n\|_{L^{q}(\mathbb{R}^N)}^{q}}\\
\geqslant& C
\frac
{\left[
\|\tilde{u}_n\|_{L^{q}(\mathbb{R}^N)}^{q}
\right]
^{
\frac{\frac{2(q-p)}{2N_{\gamma}+4-pN_{\gamma}}}{\frac{q}{2}-\frac{N_{\gamma}(q-2)}{4}}
}
}
{\|\tilde{u}_n\|_{L^{q}(\mathbb{R}^N)}^{q}}\\
=& C
\left(
\|\tilde{u}_n\|_{L^{q}(\mathbb{R}^N)}^{q}
\right)
^{\frac{\frac{2(q-p)}{2N_{\gamma}+4-pN_{\gamma}}}{\frac{q}{2}-\frac{N_{\gamma}(q-2)}{4}}-1}.
\end{aligned}
\end{equation*}
From $\frac{\frac{2(q-p)}{2N_{\gamma}+4-pN_{\gamma}}}{\frac{q}{2}-\frac{N_{\gamma}(q-2)}{4}}<1$,
we know that $\|\tilde{u}_n\|_{L^{q}(\mathbb{R}^N)}\geqslant C$.

\textbf{Step 4.}
In this step, we claim that there exists $C>0$ such that $\|\tilde{u}_n\|_{L^{2}(\mathbb{R}^N)}\leqslant C$.
We now 
divide the proof into three cases:
\begin{enumerate}
\item $2<p<q<2+\frac{4}{N_{\gamma}}$, $k=l=1$ and $\gamma=0$;

\item $2<p<q<2+\frac{4}{N_{\gamma}}$ and others;

\item $2+\frac{4}{N_{\gamma}}<q<p<2_{\gamma}^{*}$.
\end{enumerate}

\textbf{Step 4.1.}
For case (1),
it follows from the inequality in Remark 2.1 and $\int_{\mathbb{R}^{N}}
|\nabla_\gamma\tilde{u}_n|^{2} \mathrm{d}z
=\int_{\mathbb{R}^{N}}
|\tilde{u}_n|^p \mathrm{d}z=1$ that
\begin{equation}\label{3.2}
\begin{aligned}
\|\tilde{u}_n\|_{L^{q}(\mathbb{R}^N)}^{q}
\leqslant&
C
\|\tilde{u}_n\|_{L^{p}(\mathbb{R}^N)}^{p}
\left(
\|\tilde{u}_n\|_{D_{\gamma}^{1,2}(\mathbb{R}^N)}^{2}
\right)^{\frac{q-p}{2}}=C.
\end{aligned}
\end{equation}
Putting \eqref{3.2} into \eqref{3.1},
we have
\begin{equation*}
\begin{aligned}
\|\tilde{u}_n\|_{L^{q}(\mathbb{R}^N)}^{q}[C_{N_{\gamma},p,q}^{-1}+o(1)]
=
\left(
\|\tilde{u}_n\|_{L^{2}(\mathbb{R}^N)}^{2}
\right)
^{\frac{2(q-p)}{2N_{\gamma}+4-pN_{\gamma}}},
\end{aligned}
\end{equation*}
which shows $\|\tilde{u}_{n}\|_{L^{2}(\mathbb{R}^N)}\leqslant C$.

\textbf{Step 4.2.}
For case (2),
it follows from the H\"{o}lder inequality, Lemma \ref{Lemma2.2} and $\int_{\mathbb{R}^{N}}
|\nabla_\gamma\tilde{u}_n|^{2} \mathrm{d}z
=\int_{\mathbb{R}^{N}}
|\tilde{u}_n|^p \mathrm{d}z=1$ that
\begin{equation}\label{3.3}
\begin{aligned}
\|\tilde{u}_n\|_{L^{q}(\mathbb{R}^N)}^{q}
\leqslant&
\left(
\|\tilde{u}_n\|_{L^{p}(\mathbb{R}^N)}^{p}
\right)^{\frac{2_{\gamma}^{*}-q}{2_{\gamma}^{*}-p}}
\left(
\|\tilde{u}_n\|_{L^{2_{\gamma}^{*}}(\mathbb{R}^N)}^{2_{\gamma}^{*}}
\right)^{\frac{q-p}{2_{\gamma}^{*}-p}}\\
\leqslant&
C
\left(
\|\tilde{u}_n\|_{L^{p}(\mathbb{R}^N)}^{p}
\right)^{\frac{2_{\gamma}^{*}-q}{2_{\gamma}^{*}-p}}
\left(
\|\tilde{u}_n\|_{D_{\gamma}^{1,2}(\mathbb{R}^N)}^{2_{\gamma}^{*}}
\right)^{\frac{q-p}{2_{\gamma}^{*}-p}}=C.
\end{aligned}
\end{equation}
Similarly,
putting \eqref{3.3} into \eqref{3.1},
we can show that $\|\tilde{u}_{n}\|_{L^{2}(\mathbb{R}^N)}\leqslant C$.

\textbf{Step 4.3.}
For case (3),
it follows from the H\"{o}lder inequality, Lemma \ref{Lemma2.2} and $\int_{\mathbb{R}^{N}}
|\tilde{u}_n|^p \mathrm{d}z=1$ that
\begin{equation}\label{3.4}
\begin{aligned}
\|\tilde{u}_n\|_{L^{q}(\mathbb{R}^N)}^{q}
\leqslant
\left(
\|\tilde{u}_n\|_{L^{p}(\mathbb{R}^N)}^{p}
\right)^{\frac{q-2}{p-2}}
\left(
\|\tilde{u}_n\|_{L^{2}(\mathbb{R}^N)}^{2}
\right)^{\frac{p-q}{p-2}}
=
\left(
\|\tilde{u}_n\|_{L^{2}(\mathbb{R}^N)}^{2}
\right)^{\frac{p-q}{p-2}}.
\end{aligned}
\end{equation}
Combining \eqref{3.1} and \eqref{3.4},
we get that
\begin{equation*}
\begin{aligned}
C_{N_{\gamma},p,q}^{-1}+o(1)
\geqslant
\frac{
\left(
\|\tilde{u}_n\|_{L^{2}(\mathbb{R}^N)}^{2}
\right)
^{\frac{2(q-p)}{2N_{\gamma}+4-pN_{\gamma}}}}
{\left(
\|\tilde{u}_n\|_{L^{2}(\mathbb{R}^N)}^{2}
\right)^{\frac{p-q}{p-2}}}
=
\left(
\|\tilde{u}_n\|_{L^{2}(\mathbb{R}^N)}^{2}
\right)^{\frac{2(q-p)}{2N_{\gamma}+4-pN_{\gamma}}-\frac{p-q}{p-2}}.
\end{aligned}
\end{equation*}
From $\frac{2(q-p)}{2N_{\gamma}+4-pN_{\gamma}}>\frac{p-q}{p-2}$,
we obtain that $\|\tilde{u}_{n}\|_{L^{2}(\mathbb{R}^N)}\leqslant C$.

\textbf{Step 5.}
By Lions' vanishing lemma \cite[Lemma 2.1]{Alves-de-Holanda2021} and Lemma \ref{Lemma2.1}, we know that there exists a sequence $\{y_{n}\}\subset \mathbb{R}^{l}$ such that
\begin{equation*}
\begin{aligned}
u_{n}:=\tilde{u}_{n}(x,y+y_{n})
~~\mathrm{and}~~
u_{n}\rightharpoonup Q\not\equiv 0,
\end{aligned}
\end{equation*}
and $\{u_{n}\}$ is a bounded minimizing sequence of $C_{N_{\gamma},p,q}^{-1}$. The proof is complete.
\end{proof}

\begin{proof}[Proof of Theorem \ref{Theorem1.1}]
\textbf{Step 1.}
We first prove the attainability of $C_{N_{\gamma},p,q}^{-1}$.
It follows from Lemma \ref{Lemma3.1}, $\|u_n\|_{L^{p}(\mathbb{R}^N)}^{p}
=
\|u_n\|_{D_{\gamma}^{1,2}(\mathbb{R}^N)}
=1$ and Br\'{e}zis-Lieb Lemma \cite{Brezis-Lieb1983PAMS} that
\begin{equation*}
\begin{aligned}
C_{N_{\gamma},p,q}+o(1)
=&
\frac
{\|u_{n}\|_{L^{q}(\mathbb{R}^N)}^{q}}
{
\left(
\|u_{n}\|_{L^{2}(\mathbb{R}^N)}^{2}
\right)
^{\frac{2(q-p)}{4-N_{\gamma}(p-2)}}
}
\\=&
\frac
{\|u_{n}-Q\|_{L^{q}(\mathbb{R}^N)}^{q}}
{
\left(
\|u_{n}\|_{L^{2}(\mathbb{R}^N)}^{2}
\right)
^{\frac{2(q-p)}{4-N_{\gamma}(p-2)}}
}
+
\frac
{\|Q\|_{L^{q}(\mathbb{R}^N)}^{q}}
{
\left(
\|u_{n}\|_{L^{2}(\mathbb{R}^N)}^{2}
\right)
^{\frac{2(q-p)}{4-N_{\gamma}(p-2)}}
}
+o(1).
\end{aligned}
\end{equation*}
By applying \eqref{1.2},
$\frac{\|Q\|_{L^{2}(\mathbb{R}^{N})}}
{\|u_{n}\|_{L^{2}(\mathbb{R}^{N})}}\leqslant1$
and $\frac{\|u_{n}-Q\|_{L^{2}(\mathbb{R}^{N})}}{\|u_{n}\|_{L^{2}(\mathbb{R}^{N})}}\leqslant1$, one has
\begin{equation*}
\begin{aligned}
&C_{N_{\gamma},p,q}+o(1)
\\\leqslant&
C_{N_{\gamma},p,q}
\frac{
\left(
\|Q\|_{L^{p}(\mathbb{R}^N)}^{p}
\right)
^{\frac{4-N_{\gamma}(q-2)}{4-N_{\gamma}(p-2)}}
\left(
\|Q\|_{D_{\gamma}^{1,2}(\mathbb{R}^N)}^{2}
\right)
^{\frac{N_{\gamma}(q-p)}{4-N_{\gamma}(p-2)}}
\left(
\|Q\|_{L^{2}(\mathbb{R}^N)}^{2}
\right)
^{\frac{2(q-p)}{4-N_{\gamma}(p-2)}}
}
{
\left(
\|u_{n}\|_{L^{2}(\mathbb{R}^N)}^{2}
\right)
^{\frac{2(q-p)}{4-N_{\gamma}(p-2)}}
}
\\&+
C_{N_{\gamma},p,q}
\frac{
\left(
\|u_{n}-Q\|_{L^{p}(\mathbb{R}^N)}^{p}
\right)
^{\frac{4-N_{\gamma}(q-2)}{4-N_{\gamma}(p-2)}}
\left(
\|u_n-Q\|_{D_{\gamma}^{1,2}(\mathbb{R}^N)}^{2}
\right)
^{\frac{N_{\gamma}(q-p)}{4-N_{\gamma}(p-2)}}
\left(
\|u_{n}-Q\|_{L^{2}(\mathbb{R}^N)}^{2}
\right)
^{\frac{2(q-p)}{4-N_{\gamma}(p-2)}}
}
{
\left(
\|u_{n}\|_{L^{2}(\mathbb{R}^N)}^{2}
\right)
^{\frac{2(q-p)}{4-N_{\gamma}(p-2)}}
}
\\\leqslant&
C_{N_{\gamma},p,q}(1+o(1))
\left(
\|Q\|_{L^{p}(\mathbb{R}^N)}^{p}
\right)
^{\frac{4-N_{\gamma}(q-2)}{4-N_{\gamma}(p-2)}}
\left(
\|Q\|_{D_{\gamma}^{1,2}(\mathbb{R}^N)}^{2}
\right)
^{\frac{N_{\gamma}(q-p)}{4-N_{\gamma}(p-2)}}
\\&+
C_{N_{\gamma},p,q}(1+o(1))
\left(
\|u_{n}-Q\|_{L^{p}(\mathbb{R}^N)}^{p}
\right)
^{\frac{4-N_{\gamma}(q-2)}{4-N_{\gamma}(p-2)}}
\left(
\|u_n-Q\|_{D_{\gamma}^{1,2}(\mathbb{R}^N)}^{2}
\right)
^{\frac{N_{\gamma}(q-p)}{4-N_{\gamma}(p-2)}}.
\end{aligned}
\end{equation*}
Using the Young inequality, we get
\begin{equation*}
\begin{aligned}
&C_{N_{\gamma},p,q}+o(1)
\\\leqslant&
C_{N_{\gamma},p,q}(1+o(1))
\left[
\frac{4-N_{\gamma}(q-2)}{4-N_{\gamma}(p-2)}
\|Q\|_{L^{p}(\mathbb{R}^N)}^{p}
+
\frac{N_{\gamma}(q-p)}{4-N_{\gamma}(p-2)}
\|Q\|_{D_{\gamma}^{1,2}(\mathbb{R}^N)}^{2}
\right]
\\&+
C_{N_{\gamma},p,q}(1+o(1))
\left[
\frac{4-N_{\gamma}(q-2)}{4-N_{\gamma}(p-2)}
\|u_{n}-Q\|_{L^{p}(\mathbb{R}^N)}^{p}
+
\frac{N_{\gamma}(q-p)}{4-N_{\gamma}(p-2)}
\|u_n-Q\|_{D_{\gamma}^{1,2}(\mathbb{R}^N)}^{2}
\right]
\\=&
C_{N_{\gamma},p,q}(1+o(1)).
\end{aligned}
\end{equation*}
Thus, all the inequalities above have to be equalities.
By the equality of $\frac{\|Q\|_{L^{2}(\mathbb{R}^{N})}}{\|u_{n}\|_{L^{2}(\mathbb{R}^{N})}}=1$
or $\frac{\|u_{n}-Q\|_{L^{2}(\mathbb{R}^{N})}}{\|u_{n}\|_{L^{2}(\mathbb{R}^{N})}}=1$,
from Br\'{e}zis-Lieb Lemma \cite{Brezis-Lieb1983PAMS},
we know that
\begin{equation*}
\begin{aligned}
\|Q\|_{L^{2}(\mathbb{R}^{N})}=0
\ \
\mathrm{or}
\ \
\lim\limits_{n\to\infty}
\|u_{n}-Q\|_{L^{2}(\mathbb{R}^{N})}=0.
\end{aligned}
\end{equation*}
Note that $Q\not\equiv0$. Then
\begin{equation}\label{3.5}
\begin{aligned}
\lim\limits_{n\to\infty}
\|u_{n}-Q\|_{L^{2}(\mathbb{R}^{N})}=0.
\end{aligned}
\end{equation}
By the equality of the Young inequality, we know that
\begin{equation}\label{3.6}
\begin{aligned}
&\int_{\mathbb{R}^{N}}
|Q|^{p}
\mathrm{d}z
=
\int_{\mathbb{R}^{N}}
|\nabla_{\gamma}Q|^{2}
\mathrm{d}z
\ \
\mathrm{and}
\ \
\int_{\mathbb{R}^{N}}
|u_{n}-Q|^{p}
\mathrm{d}z
=
\int_{\mathbb{R}^{N}}
|\nabla_{\gamma}(u_{n}-Q)|^{2}
\mathrm{d}z.
\end{aligned}
\end{equation}
From Lemma \ref{Lemma2.2} and \eqref{3.5},
we obtain that
\begin{equation*}
\begin{aligned}
\lim\limits_{n\to\infty}
\|u_{n}-Q\|_{L^{p}(\mathbb{R}^{N})}^{p}
=0.
\end{aligned}
\end{equation*}
Putting this result back to \eqref{3.6},
one has
\begin{equation*}
\begin{aligned}
\lim\limits_{n\to\infty}
\int_{\mathbb{R}^{N}}
|\nabla_{\gamma}(u_{n}-Q)|^{2}
\mathrm{d}z
=0.
\end{aligned}
\end{equation*}
Hence, $u_n\to Q$ strongly in $H^{1,2}_{\gamma}(\mathbb{R}^{N})$.
By Lemma \ref{Lemma2.2} and Remark \ref{remark2.1}, we also have
$u_n\to Q$ strongly in $L^q(\mathbb{R}^{N})$.
It follows that
\begin{equation*}
\begin{aligned}
W_{N_{\gamma},p,q}(Q)=C_{N_{\gamma},p,q}^{-1}
\ \ 
\mathrm{and}
\ \ 
\int_{\mathbb{R}^{N}}|\nabla_{\gamma}Q|^{2}\mathrm{d}z
=\int_{\mathbb{R}^{N}}|Q|^p\mathrm{d}z=1.
\end{aligned}
\end{equation*}

\textbf{Step 2.}
We next prove the upper bound of $C_{N_{\gamma},p,q}$.
For both ranges of $p,q$, we have
\begin{equation*}
\begin{aligned}
0<\frac{N_{\gamma}(q-p)}{2N_{\gamma}+4-pN_{\gamma}}<1
\ \
\mathrm{and}
\ \
\frac{2N_{\gamma}+4-qN_{\gamma}}{2N_{\gamma}+4-pN_{\gamma}}
+\frac{N_{\gamma}(q-p)}{2N_{\gamma}+4-pN_{\gamma}}=1.
\end{aligned}
\end{equation*}
Applying the H\"{o}lder inequality, for every
$u\in H^{1,2}_{\gamma,rad_{x}}(\mathbb{R}^{N})\backslash\{0\}$, we get
\begin{equation}\label{3.7}
\begin{aligned}
\int_{\mathbb{R}^{N}}|u|^q\mathrm{d}z
<&
\left(\int_{\mathbb{R}^{N}}|u|^p\mathrm{d}z\right)
^{\frac{2N_{\gamma}+4-qN_{\gamma}}{2N_{\gamma}+4-pN_{\gamma}}}
\left(\int_{\mathbb{R}^{N}}|u|^{2+\frac{4}{N_{\gamma}}}
\mathrm{d}z\right)
^{\frac{N_{\gamma}(q-p)}{2N_{\gamma}+4-pN_{\gamma}}}.
\end{aligned}
\end{equation}
We point out that the equality in the H\"{o}lder inequality is not valid.
Since the equality holds when $|u|$ is either zero or a fixed positive constant.
From $u\in L^{2}(\mathbb{R}^{N})$,
we know that $u$ takes only the value zero.
Moreover,
from $u\in H^{1,2}_{\gamma,rad_{x}}(\mathbb{R}^{N})\backslash\{0\}$,
one deduces that the equality in the H\"{o}lder inequality is not valid.

It follows that \eqref{3.7} and Lemma \ref{Lemma2.2}, we obtain
\begin{equation*}
\begin{aligned}
\int_{\mathbb{R}^{N}}|u|^q\mathrm{d}z
<&
\left(C_{N_{\gamma},2+\frac{4}{N_{\gamma}}}\right)
^{\frac{N_{\gamma}(q-p)}{2N_{\gamma}+4-pN_{\gamma}}}
\left(\int_{\mathbb{R}^{N}}|u|^{2}\mathrm{d}z\right)
^{\frac{2(q-p)}{2N_{\gamma}+4-pN_{\gamma}}}
\left(\int_{\mathbb{R}^{N}}|\nabla_{\gamma}u|^{2}\mathrm{d}z\right)
^{\frac{N_{\gamma}(q-p)}{2N_{\gamma}+4-pN_{\gamma}}}\\
&\times
\left(\int_{\mathbb{R}^{N}}|u|^p\mathrm{d}z\right)
^{\frac{2N_{\gamma}+4-qN_{\gamma}}{2N_{\gamma}+4-pN_{\gamma}}}.
\end{aligned}
\end{equation*}
Taking $u=Q$, where $Q$ is the optimizer obtained in Step 1, and using
$W_{N_{\gamma},p,q}(Q)=C_{N_{\gamma},p,q}^{-1}$,
we obtain the strict upper bound of $C_{N_{\gamma},p,q}$.

Finally, choosing the mass critical constant as in
Lemma \ref{Lemma2.3} gives the explicit upper bound stated in
Theorem \ref{Theorem1.1}, since
$\frac{N_{\gamma}(q-p)}{2N_{\gamma}+4-pN_{\gamma}}>0$
in both ranges of $p,q$.
For $\gamma\geqslant1$, the improved estimate in the same lemma
gives the claimed replacement of $2\gamma+1$ by $1$.

\textbf{Step 3.}
We next prove the lower bound of $C_{N_{\gamma},p,q}$.
For $t>0$, let
\begin{equation*}
\begin{aligned}
u(x,y)=e^{-\frac{|x|^{2}+t|y|^{2}}{2}}.
\end{aligned}
\end{equation*}
Then $u$ is radial in $x$ and
\begin{equation*}
\begin{aligned}
\nabla_xu=-xu,
\ \ 
\nabla_yu=-tyu,
\ \ 
|\nabla_{\gamma}u|^{2}
=\left[|x|^{2}+(1+\gamma)^{2}t^{2}|x|^{2\gamma}|y|^{2}\right]
e^{-|x|^{2}-t|y|^{2}}.
\end{aligned}
\end{equation*}
For every $r>0$, by the Fubini theorem and the Gaussian integral, we have
\begin{equation*}
\begin{aligned}
\int_{\mathbb{R}^{N}}|u|^r\mathrm{d}z
&=\left(\int_{\mathbb{R}^{k}}e^{-\frac{r}{2}|x|^{2}}\mathrm{d}x\right)
\left(\int_{\mathbb{R}^{l}}e^{-\frac{rt}{2}|y|^{2}}\mathrm{d}y\right)
\\
&=\left(\frac{2\pi}{r}\right)^{\frac{k}{2}}
\left(\frac{2\pi}{rt}\right)^{\frac{l}{2}}
=\left(\frac{2\pi}{r}\right)^{\frac{N}{2}}t^{-\frac{l}{2}}.
\end{aligned}
\end{equation*}
Then
\begin{equation}\label{3.8}
\begin{aligned}
&\int_{\mathbb{R}^{N}}|u|^{2}\mathrm{d}z
=
\pi^{\frac{N}{2}}t^{-\frac{l}{2}},\\
&\int_{\mathbb{R}^{N}}|u|^p\mathrm{d}z
=
\pi^{\frac{N}{2}}t^{-\frac{l}{2}}
\left(\frac{2}{p}\right)^{\frac{N}{2}},\\
&\int_{\mathbb{R}^{N}}|u|^q\mathrm{d}z
=
\pi^{\frac{N}{2}}t^{-\frac{l}{2}}
\left(\frac{2}{q}\right)^{\frac{N}{2}}.
\end{aligned}
\end{equation}
In order to show that $u\in H^{1,2}_{\gamma,rad_{x}}(\mathbb{R}^{N})\backslash\{0\}$,
we need to estimate the gradient term.
Integration by parts gives
\begin{equation*}
\begin{aligned}
\int_{\mathbb{R}^{k}}|x|^{2}e^{-|x|^{2}}\mathrm{d}x
&=-\frac{1}{2}\sum_{i=1}^{k}
\int_{\mathbb{R}^{k}}x_i
\frac{\partial}{\partial x_i}e^{-|x|^{2}}\mathrm{d}x
\\
&=\frac{k}{2}\int_{\mathbb{R}^{k}}e^{-|x|^{2}}\mathrm{d}x
=\frac{k}{2}\pi^{\frac{k}{2}},
\\
\int_{\mathbb{R}^{l}}|y|^{2}e^{-t|y|^{2}}\mathrm{d}y
&=-\frac{1}{2t}\sum_{j=1}^{l}
\int_{\mathbb{R}^{l}}y_j
\frac{\partial}{\partial y_j}e^{-t|y|^{2}}\mathrm{d}y
\\
&=\frac{l}{2t}\int_{\mathbb{R}^{l}}e^{-t|y|^{2}}\mathrm{d}y
=\frac{l}{2}\pi^{\frac{l}{2}}t^{-\frac{l}{2}-1}.
\end{aligned}
\end{equation*}
For $\gamma>0$ and $x\neq0$, the elementary inequality
\begin{equation*}
\begin{aligned}
\log\left(\frac{|x|^{2}}{\frac{k}{2}+\gamma}\right)
\leqslant
\frac{|x|^{2}}{\frac{k}{2}+\gamma}-1
\end{aligned}
\end{equation*}
which implies that
\begin{equation*}
\begin{aligned}
|x|^{2\gamma}e^{-|x|^{2}}
=&
\left(\frac{k}{2}+\gamma\right)^{\gamma}
e^{
\gamma\log\left(\frac{|x|^{2}}{\frac{k}{2}+\gamma}\right)-|x|^{2}}\\
\leqslant&
\left(\frac{k}{2}+\gamma\right)^{\gamma}
e^{
\gamma\left(\frac{|x|^{2}}{\frac{k}{2}+\gamma}-1\right)-|x|^{2}}\\
=&
\left(\frac{k}{2}+\gamma\right)^{\gamma}
e^{-\gamma}
e^{
\left(\frac{\gamma}{\frac{k}{2}+\gamma}-1\right)|x|^{2}}\\
=&
\left(\frac{k}{2}+\gamma\right)^{\gamma}
e^{-\gamma}
e^{-\frac{k}{k+2\gamma}|x|^{2}}.
\end{aligned}
\end{equation*}
For $x=0$, 
we have
\begin{equation*}
\begin{aligned}
|x|^{2\gamma}e^{-|x|^{2}}
=0\leqslant
\left(\frac{k}{2}+\gamma\right)^{\gamma}
e^{-\gamma}
=
\left(\frac{k}{2}+\gamma\right)^{\gamma}
e^{-\gamma}
e^{-\frac{k}{k+2\gamma}|x|^{2}}.
\end{aligned}
\end{equation*}
Since
\begin{equation*}
\begin{aligned}
\frac{k}{2}\log\left(1+\frac{2\gamma}{k}\right)\leqslant\gamma,
\end{aligned}
\end{equation*}
we obtain that
\begin{equation*}
\begin{aligned}
\int_{\mathbb{R}^{k}}|x|^{2\gamma}e^{-|x|^{2}}\mathrm{d}x
\leqslant&
\left(\frac{k}{2}+\gamma\right)^{\gamma}e^{-\gamma}
\int_{\mathbb{R}^{k}}
e^{-\frac{k}{k+2\gamma}|x|^{2}}
\mathrm{d}x\\
=&\pi^{\frac{k}{2}}\left(\frac{k}{2}+\gamma\right)^{\gamma}
e^{-\gamma}\left(1+\frac{2\gamma}{k}\right)^{\frac{k}{2}}
\\
\leqslant&
\pi^{\frac{k}{2}}\left(\frac{k}{2}+\gamma\right)^{\gamma}.
\end{aligned}
\end{equation*}
For $\gamma=0$, the same estimate holds with equality.
Consequently,
\begin{equation*}
\begin{aligned}
\int_{\mathbb{R}^{N}}|\nabla_xu|^{2}\mathrm{d}z
=
\left(\int_{\mathbb{R}^{k}}|x|^{2}e^{-|x|^{2}}\mathrm{d}x\right)
\left(\int_{\mathbb{R}^{l}}e^{-t|y|^{2}}\mathrm{d}y\right)
=
\frac{k}{2}\pi^{\frac{N}{2}}t^{-\frac{l}{2}}
\end{aligned}
\end{equation*}
and
\begin{equation*}
\begin{aligned}
(1+\gamma)^{2}\int_{\mathbb{R}^{N}}|x|^{2\gamma}|\nabla_yu|^{2}
\mathrm{d}z
=&(1+\gamma)^{2}t^{2}
\left(\int_{\mathbb{R}^{k}}|x|^{2\gamma}e^{-|x|^{2}}\mathrm{d}x\right)
\left(\int_{\mathbb{R}^{l}}|y|^{2}e^{-t|y|^{2}}\mathrm{d}y\right)
\\
\leqslant&
\frac{(1+\gamma)^{2}l}{2}
\pi^{\frac{N}{2}}t^{1-\frac{l}{2}}
\left(\frac{k}{2}+\gamma\right)^{\gamma}.
\end{aligned}
\end{equation*}
From the above two estimates, we get that
\begin{equation}\label{3.9}
\begin{aligned}
\int_{\mathbb{R}^{N}}|\nabla_{\gamma}u|^{2}\mathrm{d}z
&=\int_{\mathbb{R}^{N}}|\nabla_xu|^{2}\mathrm{d}z
+(1+\gamma)^{2}\int_{\mathbb{R}^{N}}|x|^{2\gamma}|\nabla_yu|^{2}
\mathrm{d}z
\\
&\leqslant
\frac{\pi^{\frac{N}{2}}t^{-\frac{l}{2}}}{2}
\left[k+(1+\gamma)^{2}lt
\left(\frac{k}{2}+\gamma\right)^{\gamma}\right].
\end{aligned}
\end{equation}
We show that
$u\in H^{1,2}_{\gamma,rad_{x}}(\mathbb{R}^{N})\backslash\{0\}$.

By the definition of $C_{N_{\gamma},p,q}$, we have
$C_{N_{\gamma},p,q}\geqslant W_{N_{\gamma},p,q}(u)^{-1}$.
Note that
\begin{equation*}
\begin{aligned}
1-\frac{2(q-p)}{2N_{\gamma}+4-pN_{\gamma}}
-\frac{2N_{\gamma}+4-qN_{\gamma}}{2N_{\gamma}+4-pN_{\gamma}}
-\frac{N_{\gamma}(q-p)}{2N_{\gamma}+4-pN_{\gamma}}
=-\frac{2(q-p)}{2N_{\gamma}+4-pN_{\gamma}}
\end{aligned}
\end{equation*}
and
\begin{equation*}
\begin{aligned}
1-\frac{2N_{\gamma}+4-qN_{\gamma}}{2N_{\gamma}+4-pN_{\gamma}}
=
\frac{N_{\gamma}(q-p)}{2N_{\gamma}+4-pN_{\gamma}}>0.
\end{aligned}
\end{equation*}
Substituting \eqref{3.8} and  \eqref{3.9} into $W_{N_{\gamma},p,q}(u)^{-1}$,
we obtain that
\begin{equation*}
\begin{aligned}
W_{N_{\gamma},p,q}(u)^{-1}
\geqslant&
\left(\frac{2}{q}\right)^{\frac{N}{2}}
\left(\frac{2}{p}\right)
^{-\frac{N(2N_{\gamma}+4-qN_{\gamma})}{2(2N_{\gamma}+4-pN_{\gamma})}}
\left(\pi^{\frac{N}{2}}t^{-\frac{l}{2}}\right)
^{-\frac{2(q-p)}{2N_{\gamma}+4-pN_{\gamma}}}
\\
&\times
\left[\frac{k+(1+\gamma)^{2}lt
\left(\frac{k}{2}+\gamma\right)^{\gamma}}{2}\right]
^{-\frac{N_{\gamma}(q-p)}{2N_{\gamma}+4-pN_{\gamma}}}
\\
=&
\left(\frac{p}{q}\right)^{\frac{N}{2}}
\left(\frac{2}{p}\right)
^{\frac{NN_{\gamma}(q-p)}{2(2N_{\gamma}+4-pN_{\gamma})}}
\pi^{-\frac{N(q-p)}{2N_{\gamma}+4-pN_{\gamma}}}
t^{\frac{l(q-p)}{2N_{\gamma}+4-pN_{\gamma}}}
\\
&\times
\left[\frac{2}{k+(1+\gamma)^{2}lt
\left(\frac{k}{2}+\gamma\right)^{\gamma}}\right]
^{\frac{N_{\gamma}(q-p)}{2N_{\gamma}+4-pN_{\gamma}}}
\\
=&
\left(\frac{p}{q}\right)^{\frac{N}{2}}
\left[
\frac{2^{1+\frac{N}{2}}t^{\frac{l}{N_{\gamma}}}}
{p^{\frac{N}{2}}\pi^{\frac{N}{N_{\gamma}}}
\left[k+(1+\gamma)^{2}lt
\left(\frac{k}{2}+\gamma\right)^{\gamma}\right]}
\right]^{\frac{N_{\gamma}(q-p)}{2N_{\gamma}+4-pN_{\gamma}}}.
\end{aligned}
\end{equation*}

It remains to choose $t$.
Since $N_{\gamma}-l=k+\gamma l$, we have
\begin{equation*}
\begin{aligned}
\frac{\mathrm{d}}{\mathrm{d}t}
\log\left[\frac{t^{\frac{l}{N_{\gamma}}}}
{k+(1+\gamma)^{2}lt\left(\frac{k}{2}+\gamma\right)^{\gamma}}\right]
=&\frac{l}{N_{\gamma}t}
-\frac{(1+\gamma)^{2}l\left(\frac{k}{2}+\gamma\right)^{\gamma}}
{k+(1+\gamma)^{2}lt\left(\frac{k}{2}+\gamma\right)^{\gamma}}
\\
=&\frac{l\left[k-(1+\gamma)^{2}(k+\gamma l)t
\left(\frac{k}{2}+\gamma\right)^{\gamma}\right]}
{N_{\gamma}t\left[k+(1+\gamma)^{2}lt
\left(\frac{k}{2}+\gamma\right)^{\gamma}\right]}.
\end{aligned}
\end{equation*}
We know that
\begin{equation*}
\begin{aligned}
\begin{cases}
\frac{\mathrm{d}}{\mathrm{d}t}
\log\left[\frac{t^{\frac{l}{N_{\gamma}}}}
{k+(1+\gamma)^{2}lt\left(\frac{k}{2}+\gamma\right)^{\gamma}}\right]>0,&t\in(0,t_{1}),\\
\frac{\mathrm{d}}{\mathrm{d}t}
\log\left[\frac{t^{\frac{l}{N_{\gamma}}}}
{k+(1+\gamma)^{2}lt\left(\frac{k}{2}+\gamma\right)^{\gamma}}\right]=0,&t=t_{1}=\frac{k}
{(1+\gamma)^{2}(k+\gamma l)
\left(\frac{k}{2}+\gamma\right)^{\gamma}},\\
\frac{\mathrm{d}}{\mathrm{d}t}
\log\left[\frac{t^{\frac{l}{N_{\gamma}}}}
{k+(1+\gamma)^{2}lt\left(\frac{k}{2}+\gamma\right)^{\gamma}}\right]<0,&t\in(t_{1},+\infty).
\end{cases}
\end{aligned}
\end{equation*}
For $t=t_{1}$, we have
\begin{equation*}
\begin{aligned}
k+(1+\gamma)^{2}lt\left(\frac{k}{2}+\gamma\right)^{\gamma}
=
k+\frac{kl}{k+\gamma l}
=
\frac{kN_{\gamma}}{k+\gamma l}
\end{aligned}
\end{equation*}
and
\begin{equation*}
\begin{aligned}
t^{\frac{l}{N_{\gamma}}}
=
\left[\frac{k}
{(1+\gamma)^{2}(k+\gamma l)
\left(\frac{k}{2}+\gamma\right)^{\gamma}}\right]^{\frac{l}{N_{\gamma}}}.
\end{aligned}
\end{equation*}
Therefore,
\begin{equation*}
\begin{aligned}
C_{N_{\gamma},p,q}
\geqslant
\left(\frac{p}{q}\right)^{\frac{N}{2}}
\left[
\frac{2^{1+\frac{N}{2}}(k+\gamma l)}
{kN_{\gamma}p^{\frac{N}{2}}\pi^{\frac{N}{N_{\gamma}}}}
\left(
\frac{k}
{(1+\gamma)^{2}(k+\gamma l)
\left(\frac{k}{2}+\gamma\right)^{\gamma}}
\right)^{\frac{l}{N_{\gamma}}}
\right]^{\frac{N_{\gamma}(q-p)}{2N_{\gamma}+4-pN_{\gamma}}},
\end{aligned}
\end{equation*}
which gives the required lower bound of $C_{N_{\gamma},p,q}$.

For $\gamma=0$, the lower bound becomes
\begin{equation*}
\begin{aligned}
C_{N,p,q}\geqslant\left(\frac{p}{q}\right)^{\frac{N}{2}}
\left(\frac{2^{1+\frac{N}{2}}}{\pi Np^{\frac{N}{2}}}\right)
^{\frac{N(q-p)}{2N+4-pN}}.
\end{aligned}
\end{equation*}
This completes the proof.
\end{proof}

\section{Proof of Theorem \ref{Theorem1.2}}
\label{section4}
In this section,
we prove Theorem \ref{Theorem1.2},
i.e.,
characterizing the conditions for the minimum of $E_{\gamma}$ taking negative values on $\tilde{S}_{c}$.

\begin{lemma}\label{Lemma4.1}
Let $N=k+l$,
$k,l\geqslant 1$,
$\gamma\geqslant0$,
$\alpha=1$,
$u\in \tilde{S}_{c}$
and $p,q$ satisfy \eqref{H}.
Then the infimum of $E_{\gamma}$ takes negative values in $\tilde{S}_{c}$ if and only if, there exists $u\in \tilde{S}_{c}$ such that 
\begin{equation*}
\begin{aligned}
\|u\|_{L^{p}(\mathbb{R}^N)}^{\frac{p[N_{\gamma}(q-2)-4]}{N_{\gamma}(p-2)-4}}
\|u\|_{D_{\gamma}^{1,2}(\mathbb{R}^N)}
^{\frac{2[(p-q)N_{\gamma}]}{N_{\gamma}(p-2)-4}}
<
\frac{p}{q}
\frac{N_{\gamma}(q-2)-4}{N_{\gamma}(p-2)-4}
\left[
\frac{2}{p}
\frac{N_{\gamma}(q-p)}{4-N_{\gamma}(q-2)}
\right]^{\frac{(q-p)N_{\gamma}}{4-N_{\gamma}(p-2)}}
\|u\|_{L^{q}(\mathbb{R}^N)}^{q},
\end{aligned}
\end{equation*}
\end{lemma}

\begin{proof}
For any $c>0$ and $u\in \tilde{S}_{c}$, set
\begin{equation*}
\begin{aligned}
u_{t}(x,y)=t^{\frac{N_{\gamma}}{2}}u(tx,t^{1+\gamma}y),\qquad t>0.
\end{aligned}
\end{equation*}
Then $u_t\in\tilde{S}_{c}$ and
\begin{equation*}
\begin{aligned}
E_{\gamma}(u_{t}):=
\frac{t^{2}}{2}
\int_{\mathbb{R}^{N}}
|\nabla_{\gamma} u|^{2}
\mathrm{d}z
+
\frac{t^{\frac{N_{\gamma}(p-2)}{2}}}{p}
\int_{\mathbb{R}^{N}}
|u|^{p}
\mathrm{d}z
-\frac{t^{\frac{N_{\gamma}(q-2)}{2}}}{q}
\int_{\mathbb{R}^{N}}
|u|^{q}
\mathrm{d}z.
\end{aligned}
\end{equation*}
We divide the proof into two cases.

(1) For $2<p<q<2+\frac{4}{N_{\gamma}}$,
we set
\begin{equation*}
\begin{aligned}
h_{u}(t):=
\frac{E_{\gamma}(u_{t})}{t^{2}}
=\frac{1}{2}
\int_{\mathbb{R}^{N}}
|\nabla_{\gamma} u|^{2}
\mathrm{d}z
+
\frac{t^{\frac{N_{\gamma}(p-2)-4}{2}}}{p}
\int_{\mathbb{R}^{N}}
|u|^{p}
\mathrm{d}z
-\frac{t^{\frac{N_{\gamma}(q-2)-4}{2}}}{q}
\int_{\mathbb{R}^{N}}
|u|^{q}
\mathrm{d}z.
\end{aligned}
\end{equation*}
Clearly,
\begin{equation*}
\begin{aligned}
h_{u}'(t)
=
\frac{N_{\gamma}(p-2)-4}{2}
\frac{t^{\frac{N_{\gamma}(p-2)-6}{2}}}{p}
\int_{\mathbb{R}^{N}}
|u|^{p}
\mathrm{d}z
-
\frac{N_{\gamma}(q-2)-4}{2}
\frac{t^{\frac{N_{\gamma}(q-2)-6}{2}}}{q}
\int_{\mathbb{R}^{N}}
|u|^{q}
\mathrm{d}z.
\end{aligned}
\end{equation*}
Thus $h'_{u}(t)=0$ if and only if
\begin{equation*}
\begin{aligned}
t=t_{u}
=
\left[
\frac{q}{p}
\frac{N_{\gamma}(p-2)-4}{N_{\gamma}(q-2)-4}
\frac{
\int_{\mathbb{R}^{N}}
|u|^{p}
\mathrm{d}z
}
{
\int_{\mathbb{R}^{N}}
|u|^{q}
\mathrm{d}z
}
\right]
^{\frac{2}{(q-p)N_{\gamma}}}.
\end{aligned}
\end{equation*}
From $p<q$, we have
$\frac{N_{\gamma}(q-2)-4}{2}>\frac{N_{\gamma}(p-2)-4}{2}$ and
\begin{equation*}
\begin{aligned}
\begin{cases}
h'_{u}(t)<0,&t<t_{u},\\
h'_{u}(t)=0,&t=t_{u},\\
h'_{u}(t)>0,&t>t_{u}.
\end{cases}
\end{aligned}
\end{equation*}
This shows that $t_{u}$ is the unique critical point of $h_{u}(\cdot)$, and the minimum
\begin{equation*}
\begin{aligned}
h_{u}(t_{u})=&
\frac{1}{p}
\left[
\frac{q}{p}
\frac{N_{\gamma}(p-2)-4}{N_{\gamma}(q-2)-4}
\right]
^{\frac{N_{\gamma}(p-2)-4}{(q-p)N_{\gamma}}}
\frac{N_{\gamma}(q-p)}{N_{\gamma}(q-2)-4}
\|u\|_{L^{p}(\mathbb{R}^{N})}
^{\frac{p[N_{\gamma}(q-2)-4]}{(q-p)N_{\gamma}}}
\|u\|_{L^{q}(\mathbb{R}^{N})}
^{-\frac{q[N_{\gamma}(p-2)-4]}{(q-p)N_{\gamma}}}\\
&+
\frac{1}{2}
\|u\|_{D_{\gamma}^{1,2}(\mathbb{R}^N)}^{2},
\end{aligned}
\end{equation*}
Thus $h_{u}(t_{u})<0$ if and only if
\begin{equation*}
\begin{aligned}
\|u\|_{L^{p}(\mathbb{R}^{N})}
^{\frac{p[4-N_{\gamma}(q-2)]}{4-N_{\gamma}(p-2)}}
\|u\|_{D_{\gamma}^{1,2}(\mathbb{R}^N)}^{\frac{2(q-p)N_{\gamma}}{4-N_{\gamma}(p-2)}}
<
\frac{p}{q}
\frac{N_{\gamma}(q-2)-4}{N_{\gamma}(p-2)-4}
\left[
\frac{2}{p}
\frac{N_{\gamma}(q-p)}{4-N_{\gamma}(q-2)}
\right]^{\frac{(q-p)N_{\gamma}}{4-N_{\gamma}(p-2)}}
\|u\|_{L^{q}(\mathbb{R}^{N})}^{q}.
\end{aligned}
\end{equation*}

(2) 
For $2+\frac{4}{N_{\gamma}}<q<p<2_{\gamma}^{*}$,
we set
\begin{equation*}
\begin{aligned}
g_{u}(t):=&
\frac{E_{\gamma}(u_{t})}{t^{\frac{N_{\gamma}(p-2)}{2}}}
=
\frac{t^{\frac{4-N_{\gamma}(p-2)}{2}}}{2}
\int_{\mathbb{R}^{N}}
|\nabla_{\gamma} u|^{2}
\mathrm{d}z
+
\frac{1}{p}
\int_{\mathbb{R}^{N}}
|u|^{p}
\mathrm{d}z
-\frac{t^{\frac{(q-p)N_{\gamma}}{2}}}{q}
\int_{\mathbb{R}^{N}}
|u|^{q}
\mathrm{d}z.
\end{aligned}
\end{equation*}
Then
\begin{equation*}
\begin{aligned}
g'_{u}(t)=&
\frac{4-N_{\gamma}(p-2)}{2}
\frac{t^{\frac{4-N_{\gamma}(p-2)}{2}-1}}{2}
\int_{\mathbb{R}^{N}}
|\nabla_{\gamma} u|^{2}
\mathrm{d}z
-
\frac{(q-p)N_{\gamma}}{2}\frac{t^{\frac{(q-p)N_{\gamma}}{2}-1}}{q}
\int_{\mathbb{R}^{N}}
|u|^{q}
\mathrm{d}z.
\end{aligned}
\end{equation*}
Let $g'_{u}(t)=0$, we have
\begin{equation*}
\begin{aligned}
\tilde{t}_{u}=
\left[
\frac{q(4-N_{\gamma}(p-2))}{2(q-p)N_{\gamma}}
\frac{\int_{\mathbb{R}^{N}}
|\nabla_{\gamma} u|^{2}
\mathrm{d}z}
{\int_{\mathbb{R}^{N}}
|u|^{q}
\mathrm{d}z}
\right]
^{\frac{2}{N_{\gamma}(q-2)-4}}.
\end{aligned}
\end{equation*}
So we have 
\begin{equation*}
\begin{aligned}
\begin{cases}
g'_{u}(t)<0,&t<\tilde{t}_{u},
\\
g'_{u}(t)=0,&t=\tilde{t}_{u},
\\
g'_{u}(t)>0,&t>\tilde{t}_{u}.
\end{cases}
\end{aligned}
\end{equation*}
This shows that $\tilde{t}_{u}$ is the unique critical point of $g_{u}(\cdot)$, and the minimum
\begin{equation*}
\begin{aligned}
g_{u}(\tilde{t}_{u})
=&
\frac{N_{\gamma}(q-2)-4}{2(q-p)N_{\gamma}}
\left[
\frac{q(4-N_{\gamma}(p-2))}{2(q-p)N_{\gamma}}
\right]
^{\frac{4-N_{\gamma}(p-2)}{N_{\gamma}(q-2)-4}}
\|u\|_{D_{\gamma}^{1,2}(\mathbb{R}^N)}
^{\frac{2(q-p)N_{\gamma}}{N_{\gamma}(q-2)-4}}
\|u\|_{L^{q}(\mathbb{R}^N)}
^{\frac{q[4-N_{\gamma}(p-2)]}{4-N_{\gamma}(q-2)}}
\\&+
\frac{1}{p}
\|u\|_{L^{p}(\mathbb{R}^N)}^{p}.
\end{aligned}
\end{equation*}
Thus $g_{u}(\tilde{t}_{u})<0$ if and only if
\begin{equation*}
\begin{aligned}
\|u\|_{L^{p}(\mathbb{R}^N)}^{\frac{p\left[N_{\gamma}(q-2)-4\right]}{N_{\gamma}(p-2)-4}}
\|u\|_{D_{\gamma}^{1,2}(\mathbb{R}^N)}
^{\frac{2(p-q)N_{\gamma}}{N_{\gamma}(p-2)-4}}
<
\frac{p}{q}
\frac{N_{\gamma}(q-2)-4}{N_{\gamma}(p-2)-4}
\left[
\frac{2}{p}
\frac{N_{\gamma}(q-p)}{4-N_{\gamma}(q-2)}
\right]^{\frac{(q-p)N_{\gamma}}{4-N_{\gamma}(p-2)}}
\|u\|_{L^{q}(\mathbb{R}^N)}^{q},
\end{aligned}
\end{equation*}
which ends the proof.
\end{proof}

\begin{lemma}\label{Lemma4.2}
Let $N=k+l$,
$k,l\geqslant 1$,
$\gamma\geqslant0$,
$\alpha=1$,
$u\in \tilde{S}_{c}$
and $p,q$ satisfy \eqref{H}.
If $c\leq c_\gamma$,
then $\tilde{m}_{c,\gamma}=0$.
\end{lemma}

\begin{proof}
Using \eqref{1.2}, for all $u\in \tilde{S}_{c}$, we know that
\begin{equation}\label{4.1}
\begin{aligned}
\|u\|_{L^{q}(\mathbb{R}^N)}^{q}
\leqslant c^{\frac{2(q-p)}{4-N_{\gamma}(p-2)}}
C_{N_{\gamma},p,q}
\|u\|_{L^{p}(\mathbb{R}^N)}^{\frac{p\left[N_{\gamma}(q-2)-4\right]}{N_{\gamma}(p-2)-4}}
\|u\|_{D_{\gamma}^{1,2}(\mathbb{R}^N)}
^{\frac{2(p-q)N_{\gamma}}{N_{\gamma}(p-2)-4}}.
\end{aligned}
\end{equation}
By using \eqref{4.1} and
\begin{equation*}
\begin{aligned}
c^{\frac{2(q-p)}{4-N_{\gamma}(p-2)}}
C_{N_{\gamma},p,q}
\leqslant
\frac{q[4-N_{\gamma}(p-2)]}{p[4-N_{\gamma}(q-2)]}
\left[
\frac{p}{2}
\frac{4-N_{\gamma}(q-2)}{N_{\gamma}(q-p)}
\right]
^{\frac{N_{\gamma}(q-p)}{4-N_{\gamma}(p-2)}},
\end{aligned}
\end{equation*}
one deduces
\begin{equation*}
\begin{aligned}
\frac{p}{q}
\frac{N_{\gamma}(q-2)-4}{N_{\gamma}(p-2)-4}
\left[
\frac{2}{p}
\frac{N_{\gamma}(q-p)}{4-N_{\gamma}(q-2)}
\right]^{\frac{(q-p)N_{\gamma}}{4-N_{\gamma}(p-2)}}
\|u\|_{L^{q}(\mathbb{R}^N)}^{q}
\leqslant
\|u\|_{L^{p}(\mathbb{R}^N)}^{\frac{p\left[N_{\gamma}(q-2)-4\right]}{N_{\gamma}(p-2)-4}}
\|u\|_{D_{\gamma}^{1,2}(\mathbb{R}^N)}
^{\frac{2(p-q)N_{\gamma}}{N_{\gamma}(p-2)-4}}.
\end{aligned}
\end{equation*}
According to Lemma \ref{Lemma4.1}, we know that $\tilde{m}_{c,\gamma}=0$. The proof is complete.
\end{proof}

\begin{lemma}\label{Lemma4.4}
If $c>c_\gamma$,
then $\tilde{m}_{c,\gamma}<0$.
\end{lemma}

\begin{proof}
Let $Q$ be the optimizer of $C_{N_{\gamma},p,q}^{-1}$.
Set $Q_{c}:=\frac{\sqrt{c}}{\|Q\|_{L^{2}(\mathbb{R}^{N})}}Q$.
Then
\begin{equation*}
\begin{aligned}
C_{N_{\gamma},p,q}^{-1}
=
\frac{
\|Q_{c}\|_{L^{p}(\mathbb{R}^N)}
^{\frac{p[4-N_{\gamma}(q-2)]}{4-N_{\gamma}(p-2)}}
\|Q_{c}\|_{D_{\gamma}^{1,2}(\mathbb{R}^N)}
^{\frac{2(q-p)N_{\gamma}}{4-N_{\gamma}(p-2)}}
\|Q_{c}\|_{L^{2}(\mathbb{R}^N)}
^{\frac{4(q-p)}{4-N_{\gamma}(p-2)}}}
{\|Q_{c}\|_{L^{q}(\mathbb{R}^N)}^{q}}
.
\end{aligned}
\end{equation*}
This shows that $Q_{c}$ is also a optimizer of $C_{N_{\gamma},p,q}^{-1}$.
From $Q_{c}\in \tilde{S}_{c}$,
we have
\begin{equation*}
\begin{aligned}
\|Q_{c}\|_{L^{q}(\mathbb{R}^{N})}^{q}
=&c^{\frac{2(q-p)}{4-N_{\gamma}(p-2)}}
C_{N_{\gamma},p,q}
\|Q_{c}\|_{L^{p}(\mathbb{R}^{N})}
^{\frac{p[4-N_{\gamma}(q-2)]}{4-N_{\gamma}(p-2)}}
\|Q_{c}\|_{D_{\gamma}^{1,2}(\mathbb{R}^{N})}
^{\frac{2N_{\gamma}(q-p)}{4-N_{\gamma}(p-2)}}
\\>&
\left[
\frac{q}{p}
\frac{4-N_{\gamma}(p-2)}{4-N_{\gamma}(q-2)}
\right]
\left[
\frac{p}{2}
\frac{4-N_{\gamma}(q-2)}{N_{\gamma}(q-p)}
\right]
^{\frac{N_{\gamma}(q-p)}{4-N_{\gamma}(p-2)}}
\|Q_{c}\|_{L^{p}(\mathbb{R}^{N})}
^{\frac{p[4-N_{\gamma}(q-2)]}{4-N_{\gamma}(p-2)}}
\|Q_{c}\|_{D_{\gamma}^{1,2}(\mathbb{R}^{N})}
^{\frac{2N_{\gamma}(q-p)}{4-N_{\gamma}(p-2)}}.
\end{aligned}
\end{equation*}
From $Q_{c}\in \tilde{S}_{c}$ and Lemma \ref{Lemma4.1}, we know $\tilde{m}_{c,\gamma}<0$. The proof is complete.
\end{proof}

\begin{proof}[Proof of Theorem \ref{Theorem1.2}]
Combining Lemmas \ref{Lemma4.2}-\ref{Lemma4.4}, observing also that $\tilde{m}_{c,\gamma}\leqslant0$ for all $c>0$, we get the desired result.
\end{proof}

\section{Proof of Theorem \ref{Theorem1.3}}
\label{section5}
In this section, we prove Theorem \ref{Theorem1.3}, i.e., the existence and non-existence of $x$-radial energy GSS of equation \eqref{G}.
\begin{lemma}\label{Lemma5.1}
Let $N=k+l$,
$k,l\geqslant 1$,
$\gamma\geqslant0$,
$\alpha=1$
and $p,q$ satisfy \eqref{H}.
Then we know that $E_{\gamma}$ is coercive on $\tilde{S}_{c}$, and so $\tilde{m}_{c,\gamma}>-\infty$.
\end{lemma}

\begin{proof}
For $2<p<q<2+\frac{4}{N_{\gamma}}$,
it follows from $u\in \tilde{S}_{c}$ and Lemma \ref{Lemma2.2} that
\begin{equation}\label{5.1}
\begin{aligned}
E_{\gamma}(u)
\geqslant&
\frac{1}{2}
\|u\|_{D_{\gamma}^{1,2}(\mathbb{R}^{N})}^{2}
-
\frac{C_{N_{\gamma},q}}{q}
\|u\|_{D_{\gamma}^{1,2}(\mathbb{R}^{N})}
^{\frac{N_{\gamma}(q-2)}{2}}
\|u\|_{L^{2}(\mathbb{R}^{N})}
^{q-\frac{N_{\gamma}(q-2)}{2}}\\
\geqslant&
\frac{1}{2}
\|u\|_{D_{\gamma}^{1,2}(\mathbb{R}^{N})}^{2}
-
\frac{C_{N_{\gamma},q}}{q}
\|u\|_{D_{\gamma}^{1,2}(\mathbb{R}^{N})}
^{\frac{N_{\gamma}(q-2)}{2}}
c
^{\frac{q}{2}-\frac{N_{\gamma}(q-2)}{4}}.
\end{aligned}
\end{equation}
We see that $E_{\gamma}$ is coercive on $\tilde{S}_{c}$, and so $\tilde{m}_{c,\gamma}>-\infty$.

For $2+\frac{4}{N_{\gamma}}<q<p<2_{\gamma}^{*}$,
it follows from $u\in \tilde{S}_{c}$ and Young inequality with $\varepsilon>0$ small enough that
\begin{equation}\label{5.2}
\begin{aligned}
E_{\gamma}(u)
\geqslant&
\frac{1}{2}
\|u\|_{D_{\gamma}^{1,2}(\mathbb{R}^{N})}^{2}
+
\frac{1}{p}
\|u\|_{L^{p}(\mathbb{R}^{N})}^{p}
-\frac{1}{q}
\left(
\|u\|_{L^{2}(\mathbb{R}^{N})}^{2}
\right)^{\frac{p-q}{p-2}}
\left(
\|u\|_{L^{p}(\mathbb{R}^{N})}^{p}
\right)^{\frac{q-2}{p-2}}\\
\geqslant&
\frac{1}{2}
\|u\|_{D_{\gamma}^{1,2}(\mathbb{R}^{N})}^{2}
+
\frac{1}{p}
\|u\|_{L^{p}(\mathbb{R}^{N})}^{p}\\
&-\frac{1}{q}
\left[
\frac{\varepsilon^{-\frac{q-2}{p-q}}(p-q)}{p-2}
\|u\|_{L^{2}(\mathbb{R}^{N})}^{2}
+
\frac{\varepsilon(q-2)}{p-2}
\|u\|_{L^{p}(\mathbb{R}^{N})}^{p}
\right]
\\=&
\frac{1}{2}
\|u\|_{D_{\gamma}^{1,2}(\mathbb{R}^{N})}^{2}
+
\left[
\frac{1}{p}
-
\frac{\varepsilon(q-2)}{q(p-2)}
\right]
\|u\|_{L^{p}(\mathbb{R}^{N})}^{p}
-\frac{\varepsilon^{-\frac{q-2}{p-q}}(p-q)c}{q(p-2)}
\\=&
\frac{1}{2}
\|u\|_{D_{\gamma}^{1,2}(\mathbb{R}^{N})}^{2}
-\frac{\varepsilon^{-\frac{q-2}{p-q}}(p-q)c}{q(p-2)},
\end{aligned}
\end{equation}
where $\varepsilon=\frac{q(p-2)}{p(q-2)}$.
We see that $E_{\gamma}$ is coercive on $\tilde{S}_{c}$, and so $\tilde{m}_{c,\gamma}>-\infty$. The proof is complete.
\end{proof}

\begin{lemma}\label{Lemma5.2}
For all $c'>c$,
\begin{equation*}
\begin{aligned}
\tilde{m}_{c',\gamma}\leqslant \frac{c'}{c} \tilde{m}_{c,\gamma}.
\end{aligned}
\end{equation*}
Furthermore, for all $\bar{c}\in(0,c)$, we know that
\begin{equation*}
\begin{aligned}
\tilde{m}_{c,\gamma}\leqslant \tilde{m}_{\bar{c},\gamma}+\tilde{m}_{c-\bar{c},\gamma}.
\end{aligned}
\end{equation*}
If $\tilde{m}_{\bar{c},\gamma}$ is achieved, then
\begin{equation*}
\begin{aligned}
\tilde{m}_{c,\gamma}< \tilde{m}_{\bar{c},\gamma}+\tilde{m}_{c-\bar{c},\gamma}.
\end{aligned}
\end{equation*}
\end{lemma}

\begin{proof}
Let $u\in \tilde{S}_{c}$ and $\bar{u}(x):=u((\frac{c'}{c})^{-\frac{1}{k+l(1+\gamma)}}x,(\frac{c'}{c})^{-\frac{1+\gamma}{k+l(1+\gamma)}}y)$.
Then we have that
\begin{equation*}
\begin{aligned}
\int_{\mathbb{R}^{N}}
|\bar{u}|^{2}
\mathrm{d}z
=c'.
\end{aligned}
\end{equation*}
From $\bar{u}\in S_{c'}$, we know that
\begin{equation}\label{5.3}
\begin{aligned}
\tilde{m}_{c',\gamma}
\leqslant& E_{\gamma}(\bar{u})
\\=&
\left(\frac{c'}{c}\right)
^{1-\frac{2}{k+l(1+\gamma)}}
\frac{1}{2}
\int_{\mathbb{R}^{N}}
|\nabla_{\gamma}u|^{2}
\mathrm{d}z
+\frac{c'}{c}
\frac{1}{p}
\int_{\mathbb{R}^{N}}
|u|^{p}
\mathrm{d}z
-\frac{c'}{c}
\frac{1}{q}
\int_{\mathbb{R}^{N}}
|u|^{q}
\mathrm{d}z
\\=&
\frac{c'}{c}
E_{\gamma}(u)
+
\frac{1}{2}
\int_{\mathbb{R}^{N}}
|\nabla_{\gamma}u|^{2}
\mathrm{d}z
\left[
\left(
\frac{c'}{c}
\right)
^{1-\frac{2}{k+l(1+\gamma)}}
-
\frac{c'}{c}
\right]\\
\leqslant&
\frac{c'}{c}
E_{\gamma}(u),
\end{aligned}
\end{equation}
which shows
\begin{equation*}
\begin{aligned}
\tilde{m}_{c',\gamma}\leqslant \frac{c'}{c} \tilde{m}_{c,\gamma}.
\end{aligned}
\end{equation*}
For $\bar{c}\in(0,c)$,
we have that
\begin{equation*}
\begin{aligned}
\tilde{m}_{c,\gamma}
=
\frac{c-\bar{c}}{c}\tilde{m}_{c,\gamma}
+\frac{\bar{c}}{c}\tilde{m}_{c,\gamma}
\leqslant
\tilde{m}_{c-\bar{c},\gamma}
+
\tilde{m}_{\bar{c},\gamma}.
\end{aligned}
\end{equation*}
If $m_{\bar{c},\gamma}$ is achieved,
then there exists $u\in S_{\bar{c}}$ such that $E_{\gamma}(u)=\tilde{m}_{\bar{c},\gamma}$.
Then similar to \eqref{5.3}, 
we have that
\begin{equation*}
\begin{aligned}
\tilde{m}_{c,\gamma}
<\frac{c}{\bar{c}}
E_{\gamma}(u)
=\frac{c}{\bar{c}}m_{\bar{c},\gamma},
\end{aligned}
\end{equation*}
which gives
\begin{equation*}
\begin{aligned}
\tilde{m}_{c,\gamma}
=
\frac{c-\bar{c}}{c}\tilde{m}_{c,\gamma}
+\frac{\bar{c}}{c}\tilde{m}_{c,\gamma}
<\tilde{m}_{\bar{c},\gamma}+\tilde{m}_{c-\bar{c},\gamma}.
\end{aligned}
\end{equation*}
The proof is complete.
\end{proof}

\begin{lemma}\label{Lemma5.3}
Let $c<c_{\gamma}$.
Then
$\tilde{m}_{c,\gamma}=0$ is not achieved.
\end{lemma}

\begin{proof}
Define
\begin{equation*}
\begin{aligned}
E_{c,\gamma}(u):=
\frac{c^{-\frac{2}{k+l(1+\gamma)}}}{2}
\int_{\mathbb{R}^{N}}
|\nabla_{\gamma}u|^{2}
\mathrm{d}z
+\frac{1}{p}
\int_{\mathbb{R}^{N}}
|u|^{p}
\mathrm{d}z
-
\frac{1}{q}
\int_{\mathbb{R}^{N}}
|u|^{q}
\mathrm{d}z.
\end{aligned}
\end{equation*}
Clearly, $E_{1,\gamma}(u)=E_{\gamma}(u)$, $E_{c_{\gamma},\gamma}(u)<E_{c,\gamma}(u)$, and
\begin{equation*}
\begin{aligned}
&E_{\gamma}(u(c^{-\frac{1}{k+l(1+\gamma)}}x,c^{-\frac{1+\gamma}{k+l(1+\gamma)}}y))\\
=&
\frac{c^{1-\frac{2}{k+l(1+\gamma)}}}{2}
\int_{\mathbb{R}^{N}}
|\nabla_{\gamma}u|^{2}
\mathrm{d}z
+
\frac{c}{p}
\int_{\mathbb{R}^{N}}
|u|^{p}
\mathrm{d}x
-
\frac{c}{q}
\int_{\mathbb{R}^{N}}
|u|^{q}
\mathrm{d}x
=cE_{c,\gamma}(u).
\end{aligned}
\end{equation*}

On the contrary, we suppose that $\tilde{m}_{c,\gamma}=0$ is achived, and $u_{c}$ is the minimizer. Then
\begin{equation*}
\begin{aligned}
0
=\tilde{m}_{c,\gamma}
=E_{\gamma}(u_{c}).
\end{aligned}
\end{equation*}

Set $\tilde{u}_{c}:=u_{c}(c^{\frac{1}{k+l(1+\gamma)}}x,c^{\frac{1+\gamma}{k+l(1+\gamma)}}y)$. Then $\tilde{u}_{c}\in \tilde{S}_{1}$ and
$u_{c}=\tilde{u}_{c}(c^{-\frac{1}{k+l(1+\gamma)}}x,c^{-\frac{1+\gamma}{k+l(1+\gamma)}}y)$
and
\begin{equation*}
\begin{aligned}
0=E_{\gamma}(u_{c})
=E_{\gamma}(\tilde{u}_{c}(c^{-\frac{1}{k+l(1+\gamma)}}x,c^{-\frac{1+\gamma}{k+l(1+\gamma)}}y))
=&cE_{c,\gamma}(\tilde{u}_{c}).
\end{aligned}
\end{equation*}
This shows that $E_{c,\gamma}(\tilde{u}_{c})=0$ and $E_{c_{\gamma},\gamma}(\tilde{u}_{c})<E_{c,\gamma}(\tilde{u}_{c})=0$.

Set $\tilde{\tilde{u}}_{c_{\gamma}}:=\tilde{u}_{c}(c_{\gamma}^{-\frac{1}{k+l(1+\gamma)}}x,c_{\gamma}^{-\frac{1+\gamma}{k+l(1+\gamma)}}y)$. 
Then $\tilde{\tilde{u}}_{c_{\gamma}}\in \tilde{S}_{c_{\gamma}}$ and
\begin{equation*}
\begin{aligned}
\tilde{m}_{c_{\gamma},\gamma}\leqslant E_{\gamma}(\tilde{\tilde{u}}_{c_{\gamma}})
= E_{\gamma}(\tilde{u}_{c}(c_{\gamma}^{-\frac{1}{k+l(1+\gamma)}}x,c_{\gamma}^{-\frac{1+\gamma}{k+l(1+\gamma)}}y))
=&c_{\gamma}E_{c_{\gamma},\gamma}(\tilde{u}_{c})<0.
\end{aligned}
\end{equation*}
This is a contradiction with $\tilde{m}_{c_{\gamma},\gamma}=0$. Hence, the infimum is not achieved. The proof is complete.
\end{proof}

\begin{lemma}\label{Lemma5.4}
Let $c>c_{\gamma}$.
Then
$\tilde{m}_{c,\gamma}<0$ is achieved.
\end{lemma}
\begin{proof}
Let $\{u_{n}\}$ be the minimizing sequence of $\tilde{m}_{c,\gamma}$.
It is easy to see that $\{u_{n}\}$ is bounded in $H^{1,2}_{\gamma,rad_{x}}(\mathbb{R}^{N})$.

If $\lim\limits_{n\to\infty}\int_{\mathbb{R}^{N}}|u_{n}|^{q}\mathrm{d}z=0$, then
\begin{equation*}
\begin{aligned}
0>\tilde{m}_{c,\gamma}+o_{n}(1)=E_{\gamma}(u_{n})=
\frac{1}{2}
\int_{\mathbb{R}^{N}}
|\nabla_{\gamma}u_{n}|^{2}
\mathrm{d}z
+\frac{1}{p}
\int_{\mathbb{R}^{N}}|u_{n}|^{p}\mathrm{d}z+o(1)\geqslant 0.
\end{aligned}
\end{equation*}
This is a contradiction.
We get that $\lim\limits_{n\to\infty}\int_{\mathbb{R}^{N}}|u_{n}|^{q}\mathrm{d}z>0$.
By using Lions' vanishing Lemma \cite[Lemma 2.1]{Alves-de-Holanda2021} and Lemma \ref{Lemma2.1}, there exists $\{y_{n}\}\subset \mathbb{R}^{l}$ such that  $\bar{u}_{n}:=u_{n}(x,y+y_{n})\rightharpoonup u\not \equiv0$ in $L^{q}_{loc}(\mathbb{R}^{N})$.

We claim that $\int_{\mathbb{R}^{N}}|u|^{2}\mathrm{d}x=c$.
If $\int_{\mathbb{R}^{N}}
|u|^{2}
\mathrm{d}x
=\bar{c}
<c$,
we deduce from the Br\'{e}zis-Lieb Lemma \cite{Brezis-Lieb1983PAMS} and Lemma \ref{Lemma5.2} that
\begin{equation*}
\begin{aligned}
\tilde{m}_{c,\gamma}
=\lim_{n\to\infty}E_{\gamma}(\bar{u}_{n})
=E_{\gamma}(u)
+\lim_{n\to\infty}E_{\gamma}(\bar{u}_{n}-u)
\geqslant \tilde{m}_{\bar{c},\gamma}+\tilde{m}_{c-\bar{c},\gamma}
\geqslant \tilde{m}_{c,\gamma}.
\end{aligned}
\end{equation*}
Hence, $E_{\gamma}(u)=\tilde{m}_{\bar{c},\gamma}$.
Using the Br\'{e}zis-Lieb Lemma \cite{Brezis-Lieb1983PAMS} and Lemma \ref{Lemma5.2} again,
we know that
\begin{equation*}
\begin{aligned}
\tilde{m}_{c,\gamma}
=\lim_{n\to\infty}E_{\gamma}(\bar{u}_{n})
=E_{\gamma}(u)
+\lim_{n\to\infty}E_{\gamma}(\bar{u}_{n}-u)
\geqslant \tilde{m}_{\bar{c},\gamma}+\tilde{m}_{c-\bar{c},\gamma}
> \tilde{m}_{c,\gamma}.
\end{aligned}
\end{equation*}
This is a contradiction.
Therefore,
we get that $\int_{\mathbb{R}^{N}}|u|^{2}\mathrm{d}x=c$ and
$E_{\gamma}(u)=\tilde{m}_{c,\gamma}$. The proof is complete.
\end{proof}

\begin{lemma}\label{Lemma5.5}
$\tilde{m}_{c_{\gamma},\gamma}=0$ is achieved.
\end{lemma}
\begin{proof}
Let $c_{n}:=c_{\gamma}+\frac{1}{n}$, for all $n\in \mathbb{N}^{+}$.
From $c_{n}>c_{\gamma}$ and $\tilde{m}_{c_{n}}<0$,
we know by Lemma \ref{Lemma5.4} that $\tilde{m}_{c_{n},\gamma}$ admits a minimizer $u_{c_{n}}\in\tilde{S}_{c_{n}}$ for all $n\in \mathbb{N}^{+}$.
Since $\tilde{m}_{c_{n},\gamma}\to \tilde{m}_{c_{\gamma},\gamma}$,
it can be deduced that
\begin{equation*}
\begin{aligned}
\lim_{n\to\infty}
E_{\gamma}(u_{c_{n}})
=
\lim_{n\to\infty}
\tilde{m}_{c_{n},\gamma}
=
\tilde{m}_{c_{\gamma},\gamma}
=0.
\end{aligned}
\end{equation*}
{\bf Step 1.}
In this step,
we show that $\{u_{c_{n}}\}$ is bounded in $H^{1,2}_{\gamma,rad_{x}}(\mathbb{R}^{N})$.

For $2<p<q<2+\frac{4}{N_{\gamma}}$,
it follows from \eqref{5.1} that
\begin{equation*}
\begin{aligned}
E_{\gamma}(u_{c_{n}})
\geqslant
\frac{1}{2}
\|u_{c_{n}}\|_{D_{\gamma}^{1,2}(\mathbb{R}^{N})}^{2}
-
\frac{C_{N_{\gamma},q}}{q}
\|u_{c_{n}}\|_{D_{\gamma}^{1,2}(\mathbb{R}^{N})}
^{\frac{N_{\gamma}(q-2)}{2}}
c_{n}^{\frac{q}{2}-\frac{N_{\gamma}(q-2)}{4}}.
\end{aligned}
\end{equation*}
For $2+\frac{4}{N_{\gamma}}<q<p<2_{\gamma}^{*}$,
it follows from \eqref{5.2} that
\begin{equation*}
\begin{aligned}
E_{\gamma}(u_{c_{n}})
\geqslant
\frac{1}{2}
\|u_{c_{n}}\|_{D_{\gamma}^{1,2}(\mathbb{R}^{N})}^{2}
-\frac{\varepsilon^{-\frac{q-2}{p-q}}(p-q)c_{n}}{q(p-2)}.
\end{aligned}
\end{equation*}
These imply that $\int_{\mathbb{R}^{N}}
|\nabla_{\gamma}u_{c_{n}}|^{2}
\mathrm{d}z$ is bounded.
Then we conclude that $\{u_{c_{n}}\}$ is bounded in $H^{1,2}_{\gamma,rad_{x}}(\mathbb{R}^{N})$.

{\bf Step 2.}
In this step,
we prove that
\begin{equation*}
\begin{aligned}
\lim_{n\to\infty}
\int_{\mathbb{R}^{N}}
|u_{c_{n}}|^{q}
\mathrm{d}z
>0.
\end{aligned}
\end{equation*}
Indeed,
if $\lim\limits_{n\to\infty}
\|u_{c_{n}}\|_{L^{q}(\mathbb{R}^N)}^{q}=0$,
then by $\lim\limits_{n\to\infty}E_{\gamma}(u_{c_{n}})=0$,
we obtain that
\begin{equation}\label{5.4}
\begin{aligned}
\lim_{n\to\infty}
\int_{\mathbb{R}^{N}}
|\nabla_{\gamma}u_{c_{n}}|^{2}
\mathrm{d}z
=0
\ \
\mathrm{and}
\ \
\lim_{n\to\infty}
\int_{\mathbb{R}^{N}}
|u_{c_{n}}|^{p}
\mathrm{d}z
=0.
\end{aligned}
\end{equation}
We now 
divide the proof into three cases:
\begin{enumerate}
\item $2<p<q<2+\frac{4}{N_{\gamma}}$, $k=l=1$ and $\gamma=0$;

\item $2<p<q<2+\frac{4}{N_{\gamma}}$ and others;

\item $2+\frac{4}{N_{\gamma}}<q<p<2_{\gamma}^{*}$.
\end{enumerate}

{\bf Step 2.1.}
For case (1), 
by using the the inequality in Remark 2.1, the Young inequality and \eqref{5.4},
we have that
\begin{equation*}
\begin{aligned}
E_{\gamma}(u_{c_{n}})
\geqslant&
\frac{1}{2}
\|u_{c_{n}}\|_{D_{\gamma}^{1,2}(\mathbb{R}^N)}^{2}
+
\frac{1}{p}
\|u_{c_{n}}\|_{L^{p}(\mathbb{R}^N)}^{p}
-\frac{C}{q}
\|u_{c_{n}}\|_{L^{p}(\mathbb{R}^N)}^{p}
\left(
\|u_{c_{n}}\|_{D_{\gamma}^{1,2}(\mathbb{R}^N)}^{2}
\right)^{\frac{q-p}{2}}\\
=&
\frac{1}{2}
\|u_{c_{n}}\|_{D_{\gamma}^{1,2}(\mathbb{R}^N)}^{2}
+
\|u_{c_{n}}\|_{L^{p}(\mathbb{R}^N)}^{p}
\left[
\frac{1}{p}
-\frac{C}{q}
\left(
\|u_{c_{n}}\|_{D_{\gamma}^{1,2}(\mathbb{R}^N)}^{2}
\right)^{\frac{q-p}{2}}
\right]
\geqslant 0.
\end{aligned}
\end{equation*}
This is a contradiction with $E_{\gamma}(u_{c_{n}})
=\tilde{m}_{c_{n},\gamma}
<0$.
Hence,
$\lim\limits_{n\to\infty}
\|u_{c_{n}}\|_{L^{q}(\mathbb{R}^N)}^{q}>0$.

{\bf Step 2.2.}
For case (2), by using the H\"{o}lder, Young inequalities, Lemma \ref{Lemma2.2} and \eqref{5.4},
we have that
\begin{equation*}
\begin{aligned}
E_{\gamma}(u_{c_{n}})
\geqslant&
\frac{1}{2}
\|u_{c_{n}}\|_{D_{\gamma}^{1,2}(\mathbb{R}^N)}^{2}
+
\frac{1}{p}
\|u_{c_{n}}\|_{L^{p}(\mathbb{R}^N)}^{p}\\
&-\frac{1}{q}
\left(
\|u_{c_{n}}\|_{L^{p}(\mathbb{R}^N)}^{p}
\right)^{\frac{2_{\gamma}^{*}-q}{2_{\gamma}^{*}-p}}
\left(
C_{N_\gamma,2_\gamma^*}
\|u_{c_{n}}\|_{D_{\gamma}^{1,2}(\mathbb{R}^N)}^{2_{\gamma}^{*}}
\right)^{\frac{q-p}{2_{\gamma}^{*}-p}}\\
\geqslant&
\frac{1}{2}
\|u_{c_{n}}\|_{D_{\gamma}^{1,2}(\mathbb{R}^N)}^{2}
+
\frac{1}{p}
\|u_{c_{n}}\|_{L^{p}(\mathbb{R}^N)}^{p}\\
&-\frac{1}{q}
\left[
\frac{\varepsilon(2_{\gamma}^{*}-q)}{2_{\gamma}^{*}-p}
\|u_{c_{n}}\|_{L^{p}(\mathbb{R}^N)}^{p}
+
\frac{C_{N_\gamma,2_\gamma^*}\varepsilon^{-\frac{2_{\gamma}^{*}-q}{q-p}}(q-p)}{2_{\gamma}^{*}-p}
\|u_{c_{n}}\|_{D_{\gamma}^{1,2}(\mathbb{R}^N)}^{2_{\gamma}^{*}}
\right]
\\\geqslant&
\frac{1}{2}
\|u_{c_{n}}\|_{D_{\gamma}^{1,2}(\mathbb{R}^N)}^{2}
-
\frac{C_{N_\gamma,2_\gamma^*}\varepsilon^{-\frac{2_{\gamma}^{*}-q}{q-p}}(q-p)}{q(2_{\gamma}^{*}-p)}
\|u_{c_{n}}\|_{D_{\gamma}^{1,2}(\mathbb{R}^N)}^{2_{\gamma}^{*}}
\geqslant 0,
\end{aligned}
\end{equation*}
where $\varepsilon>0$ small.
This is a contradiction with $E_{\gamma}(u_{c_{n}})
=\tilde{m}_{c_{n},\gamma}
<0$.
Hence,
$\lim\limits_{n\to\infty}
\|u_{c_{n}}\|_{L^{q}(\mathbb{R}^N)}^{q}>0$.

{\bf Step 2.3.}
For case (3),
by using Lemma \ref{Lemma2.2} and \eqref{5.4},
we have that
\begin{equation*}
\begin{aligned}
E_{\gamma}(u_{c_{n}})
\geqslant&
\frac{1}{2}
\|u_{c_{n}}\|_{D_{\gamma}^{1,2}(\mathbb{R}^N)}^{2}
+
\frac{1}{p}
\|u_{c_{n}}\|_{L^{p}(\mathbb{R}^N)}^{p}
-\frac{C_{N_\gamma,q}}{q}
\|u_{c_{n}}\|_{D_{\gamma}^{1,2}(\mathbb{R}^N)}^{\frac{N_{\gamma}(q-2)}{2}}
c_{n}^{\frac{q}{2}-\frac{N_{\gamma}(q-2)}{4}}
\geqslant 0.
\end{aligned}
\end{equation*}
This is a contradiction with $E_{\gamma}(u_{c_{n}})
=\tilde{m}_{c_{n},\gamma}
<0$.
Hence,
$\lim\limits_{n\to\infty}
\|u_{c_{n}}\|_{L^{q}(\mathbb{R}^N)}^{q}>0$.

By virtue of Lions' vanishing Lemma \cite[Lemma 2.1]{Alves-de-Holanda2021} and Lemma \ref{Lemma2.1}, we know that there exists $\{y_{n}\}\subset \mathbb{R}^{l}$ such that  $\bar{u}_{c_{n}}:=u_{c_{n}}(x,y+y_{n})$ converges to $\bar{u}_{c_{\gamma}}\not \equiv0$ in $L_{loc}^{q}(\mathbb{R}^{N})$.

{\bf Step 3.}
In this step,
we prove that $\bar{u}_{c_{\gamma}}$ is a minimizer of $\tilde{m}_{c_{\gamma},\gamma}$.
By using the Br\'{e}zis-Lieb Lemma \cite{Brezis-Lieb1983PAMS},
we get that
\begin{equation}\label{5.5}
\begin{aligned}
0=
\tilde{m}_{c_{\gamma},\gamma}
=
\lim_{n\to\infty}
\tilde{m}_{c_{n}}
=
\lim_{n\to\infty}
E_{\gamma}(\bar{u}_{c_{n}})
=E_{\gamma}(\bar{u}_{c_{\gamma}})
+\lim_{n\to\infty}E_{\gamma}(\bar{u}_{c_{n}}-\bar{u}_{c_{\gamma}})
\end{aligned}
\end{equation}
and
\begin{equation*}
\begin{aligned}
c_{\gamma}
=
\lim_{n\to\infty}
c_{n}
=
\lim_{n\to\infty}
\|\bar{u}_{c_{n}}\|_{L^{2}(\mathbb{R}^{N})}^{2}
=
\|\bar{u}_{c_{\gamma}}\|_{L^{2}(\mathbb{R}^{N})}^{2}
+
\lim_{n\to\infty}
\|\bar{u}_{c_{n}}-\bar{u}_{c_{\gamma}}\|_{L^{2}(\mathbb{R}^{N})}^{2}.
\end{aligned}
\end{equation*}
From $0<\|\bar{u}_{c_{0}}\|_{L^{2}(\mathbb{R}^{N})}^{2}\leqslant c_{\gamma}$ and Theorem \ref{Theorem1.2},
we obtain that $E_{\gamma}(\bar{u}_{c_{\gamma}})\geqslant 0$
and
\begin{equation*}
\begin{aligned}
\lim_{n\to\infty}
\|\bar{u}_{c_{n}}-\bar{u}_{c_{\gamma}}\|_{L^{2}(\mathbb{R}^{N})}^{2}
=
c_{\gamma}
-
\|\bar{u}_{c_{\gamma}}\|_{L^{2}(\mathbb{R}^{N})}^{2}<c_{\gamma}.
\end{aligned}
\end{equation*}
This shows that
$\lim_{n\to\infty}E_{\gamma}(\bar{u}_{c_{n}}-\bar{u}_{c_{\gamma}})\geqslant 0$.
From \eqref{5.5},
we know that $E_{\gamma}(\bar{u}_{c_{\gamma}})=0$.

If $\|\bar{u}_{c_{\gamma}}\|_{L^{2}(\mathbb{R}^{N})}^{2}<c_{\gamma}$,
then $\bar{u}_{c_{\gamma}}$ is a minimizer of $\tilde{m}_{c,\gamma}$ for some $c<c_{\gamma}$,
and we reach a contradiction with Lemma \ref{Lemma5.3}.
Hence, we know that $\|\bar{u}_{c_{\gamma}}\|_{L^{2}(\mathbb{R}^{N})}^{2}=c_{\gamma}$ and $E_{\gamma}(\bar{u}_{c_{\gamma}})=\tilde{m}_{c_{\gamma},\gamma}=0$. The proof is complete.
\end{proof}

We are now in a position to prove Theorem \ref{Theorem1.3}.
\begin{proof}[Proof of Theorem \ref{Theorem1.3}]
From Lemmas \ref{Lemma5.3}-\ref{Lemma5.5},
we get the existence and non-existence results of minimizer for $\tilde{m}_{c,\gamma}$.

For $c\geqslant c_{\gamma}$, we know that $u\in \tilde{S}_{c}$ is a minimizer of $\tilde{m}_{c,\gamma}$, then $u$ is a critical point of $E_{\gamma}$ on the sphere $\tilde{S}_{c}$. There exists a Lagrange multiplier $\omega$ such that $(u,\omega)\in H^{1,2}_{\gamma,rad_{x}}(\mathbb{R}^{N})\times \mathbb{R}$ is a couple of weak solution to equation \eqref{G}.
By applying Lemma \ref{Lemma2.5}, one can deduce
\begin{equation}\label{5.6}
\begin{aligned}
\|u\|_{D_{\gamma}^{1,2}(\mathbb{R}^N)}^{2}
=
\frac{N_{\gamma}(q-2)}{2q}
\|u\|_{L^{q}(\mathbb{R}^{N})}^{q}
-\frac{N_{\gamma}(p-2)}{2p}
\|u\|_{L^{p}(\mathbb{R}^{N})}^{p}.
\end{aligned}
\end{equation}
It follows from \eqref{5.6} that
\begin{equation*}
\begin{aligned}
\omega c=\omega \|u\|_{L^{2}(\mathbb{R}^{N})}^{2}
=&
\|u\|_{L^{q}(\mathbb{R}^{N})}^{q}
-
\|u\|_{D_{\gamma}^{1,2}(\mathbb{R}^N)}^{2}
-\|u\|_{L^{p}(\mathbb{R}^{N})}^{p}\\
=&
\frac{2}{N_{\gamma}}
\|u\|_{D_{\gamma}^{1,2}(\mathbb{R}^N)}^{2}
-2\tilde{m}_{c,\gamma}
\\\geqslant&
\frac{2}{N_{\gamma}}
\|u\|_{D_{\gamma}^{1,2}(\mathbb{R}^N)}^{2}>0.
\end{aligned}
\end{equation*}
Hence, $\omega>0$. 

It follows from $2+\frac{4}{N_{\gamma}}<q<p<2_{\gamma}^{*}$, the H\"{o}lder and Young inequalities that
\begin{equation*}
\begin{aligned}
\frac{1}{q}
\|u\|_{L^{q}(\mathbb{R}^{N})}^{q}
\leqslant&
\frac{1}{q}
\left(
\|u\|_{L^{2}(\mathbb{R}^{N})}^{2}
\right)^{\frac{p-q}{p-2}}
\left(
\|u\|_{L^{p}(\mathbb{R}^{N})}^{p}
\right)^{\frac{q-2}{p-2}}
\\\leqslant&
\frac{\varepsilon^{-\frac{q-2}{p-q}}(p-q)}{q(p-2)}
\|u\|_{L^{2}(\mathbb{R}^{N})}^{2}
+
\frac{\varepsilon(q-2)}{q(p-2)}
\|u\|_{L^{p}(\mathbb{R}^{N})}^{p}
\\=&
\frac{p-q}{q(p-2)}
\left[
\frac{p(q-2)}{q(p-2)}
\right]
^{\frac{q-2}{p-q}}
\|u\|_{L^{2}(\mathbb{R}^{N})}^{2}
+
\frac{1}{p}
\|u\|_{L^{p}(\mathbb{R}^{N})}^{p},
\end{aligned}
\end{equation*}
where we take
$\varepsilon^{-1}
=
\frac{p(q-2)}{q(p-2)}$.
If $\omega\geqslant
\frac{2(p-q)}{q(p-2)}
\left[
\frac{p(q-2)}{q(p-2)}
\right]
^{\frac{q-2}{p-q}}$,
then we have
\begin{equation*}
\begin{aligned}
0=&
\frac{N_{\gamma}-2}{2}
\|u\|_{D_{\gamma}^{1,2}(\mathbb{R}^N)}^{2}
+\omega
\frac{N_{\gamma}}{2} \|u\|_{L^{2}(\mathbb{R}^{N})}^{2}
+\frac{N_{\gamma}}{p}
\|u\|_{L^{p}(\mathbb{R}^{N})}^{p}
-
\frac{N_{\gamma}}{q}
\|u\|_{L^{q}(\mathbb{R}^{N})}^{q}\\
\geqslant&
\frac{N_{\gamma}-2}{2}
\|u\|_{D_{\gamma}^{1,2}(\mathbb{R}^N)}^{2}.
\end{aligned}
\end{equation*}
Hence,
we know that $\omega<\omega_{\sharp}:=
\frac{2(p-q)}{q(p-2)}
\left[
\frac{p(q-2)}{q(p-2)}
\right]
^{\frac{q-2}{p-q}}$. This ends the proof.
\end{proof}

\section{Proof of Theorem \ref{Theorem1.4}}
\label{section6}
In this section,
we prove Theorem \ref{Theorem1.4}.

\begin{lemma}\label{Lemma6.1}
Let $N=k+l$,
$k,l\geqslant 1$,
$\gamma\geqslant0$,
$\alpha=1$
and $p,q$ satisfy \eqref{H}.
Let $Q\in H_{\gamma,rad_{x}}^{1,2}(\mathbb{R}^{N})$ be the optimizer of $C_{N_{\gamma},p,q}^{-1}$, which is obtained in Theorem \ref{Theorem1.1}.
Then $Q$ is a weak solution of
\begin{equation*}
\begin{aligned}
&
-\Delta_{\gamma}Q
+
\frac{2(q-p)}{N_{\gamma}(q-p)}
\frac{Q}
{\int_{\mathbb{R}^{N}}
|Q|^{2}
\mathrm{d}z}
+
\frac{p[4-N_{\gamma}(q-2)]}{2N_{\gamma}(q-p)}
|Q|^{p-2}Q
\\&=
\frac{q[4-N_{\gamma}(p-2)]}{2N_{\gamma}(q-p)}
\frac{
|Q|^{q-2}Q
}
{\int_{\mathbb{R}^{N}}
|Q|^{q}
\mathrm{d}z},
\ \ x\in \mathbb{R}^{N}.
\end{aligned}
\end{equation*}
\end{lemma}
\begin{proof}
Note that $Q$ is a minimizer of $C_{N_{\gamma},p,q}^{-1}$. Then it satisfies
\begin{equation*}
\begin{aligned}
\left.\frac{\mathrm{d}}{\mathrm{d}\varepsilon}\right|_{\varepsilon=0}W_{N_{\gamma},p,q}(Q+\varepsilon \phi)=0,
\end{aligned}
\end{equation*}
for all $\phi\in
H_{\gamma,rad_{x}}^{1,2}(\mathbb{R}^{N})$. Thus, $Q$ is a weak solution to the  equation.
\end{proof}

\begin{lemma}\label{Lemma6.2}
There exists $u_{c}\in H_{\gamma,rad_{x}}^{1,2}(\mathbb{R}^{N})$ such that
\begin{equation}\label{6.1}
\begin{aligned}
-\Delta_{\gamma} u_{c}
+\omega_{*} u_{c}
+|u_{c}|^{p-2}u_{c}
=|u_{c}|^{q-2}u_{c},
\ \ x\in \mathbb{R}^{N},
\end{aligned}
\end{equation}
where
\begin{equation*}
\begin{aligned}
\omega_{*}=
\frac{4(q-p)}{p[4-N_{\gamma}(q-2)]}
\left[
\frac{p}{q}
\frac{4-N_{\gamma}(q-2)}{4-N_{\gamma}(p-2)}
\int_{\mathbb{R}^{N}}
|Q|^{q}
\mathrm{d}z
\right]
^{\frac{p-2}{p-q}}
\frac{1}
{\int_{\mathbb{R}^{N}}
|Q|^{2}
\mathrm{d}z}.
\end{aligned}
\end{equation*}
Furthermore,
$W_{N_{\gamma},p,q}(u_{c})=C_{N_{\gamma},p,q}^{-1}$,
i.e. $u_{c}$ is an optimizer with $C_{N_{\gamma},p,q}^{-1}$.
\end{lemma}
\begin{proof}
Note that $Q$ is a minimizer of $C_{N_{\gamma},p,q}^{-1}$. 
We define
\begin{equation*}
\begin{aligned}
u_{c}(x,y)=\lambda Q(\mu x, \mu^{1+\gamma} y),
\end{aligned}
\end{equation*}
where
\begin{equation}\label{6.2}
\begin{aligned}
\begin{cases}
\lambda=
\left[
\frac{p}{q}
\frac{4-N_{\gamma}(q-2)}{4-N_{\gamma}(p-2)}
\int_{\mathbb{R}^{N}}
|Q|^{q}
\mathrm{d}z
\right]^{\frac{1}{p-q}},\\
\mu
=
\left[
\frac{2}{p}
\frac{N_{\gamma}(q-p)}{4-N_{\gamma}(q-2)}
\right]^{\frac{1}{2}}
\left[
\frac{p}{q}
\frac{4-N_{\gamma}(q-2)}{4-N_{\gamma}(p-2)}
\int_{\mathbb{R}^{N}}
|Q|^{q}
\mathrm{d}z
\right]
^{\frac{p-2}{2(p-q)}},
\end{cases}
\end{aligned}
\end{equation}
Then $W_{N_{\gamma},p,q}(u_{c})=W_{N_{\gamma},p,q}(Q)=C_{N_{\gamma},p,q}^{-1}$ and
\begin{equation*}
\begin{aligned}
-\Delta_{\gamma} u_{c}
+\omega_{*} u_{c}
+|u_{c}|^{p-2}u_{c}
=|u_{c}|^{q-2}u_{c},
\ \ 
x\in\mathbb{R}^{N},
\end{aligned}
\end{equation*}
with
\begin{equation*}
\begin{aligned}
\omega_{*}=
\frac{4(q-p)}{p[4-N_{\gamma}(q-2)]}
\left[
\frac{p}{q}
\frac{4-N_{\gamma}(q-2)}{4-N_{\gamma}(p-2)}
\int_{\mathbb{R}^{N}}
|Q|^{q}
\mathrm{d}z
\right]
^{\frac{p-2}{p-q}}
\frac{1}
{\int_{\mathbb{R}^{N}}
|Q|^{2}
\mathrm{d}z}.
\end{aligned}
\end{equation*}
This completes the proof.
\end{proof}

\begin{lemma}\label{Lemma6.3}
It holds
\begin{equation*}
\begin{aligned}
\int_{\mathbb{R}^{N}}
|\nabla_{\gamma}u_{c}|^{2}
\mathrm{d}z
=
\frac{2N_{\gamma}(q-p)}{p[4-N_{\gamma}(q-2)]}
\int_{\mathbb{R}^{N}}
|u_{c}|^{p}
\mathrm{d}z,
\end{aligned}
\end{equation*}
\begin{equation*}
\begin{aligned}
\int_{\mathbb{R}^{N}}
|\nabla_{\gamma}u_{c}|^{2}
\mathrm{d}z
=
\frac{2N_{\gamma}(q-p)}{q[4-N_{\gamma}(p-2)]}
\int_{\mathbb{R}^{N}}
|u_{c}|^{q}
\mathrm{d}z
\end{aligned}
\end{equation*}
and
\begin{equation*}
\begin{aligned}
\int_{\mathbb{R}^{N}}|u_{c}|^{q}\mathrm{d}x
=
\frac{q[4-N_{\gamma}(p-2)]}{p[4-N_{\gamma}(q-2)]}
\int_{\mathbb{R}^{N}}
|u_{c}|^{p}
\mathrm{d}z,
\end{aligned}
\end{equation*}
where $u_{c}$ is defined in Lemma \ref{Lemma6.2}.
\end{lemma}

\begin{proof}
Using \eqref{6.2} and $\int_{\mathbb{R}^{N}}
|\nabla_{\gamma}Q|^{2}
\mathrm{d}z=\int_{\mathbb{R}^{N}}|Q|^{p}\mathrm{d}z=1$, a direct calculation yields
\begin{equation*}
\begin{aligned}
\frac{\int_{\mathbb{R}^{N}}
|\nabla_{\gamma}u_{c}|^{2}
\mathrm{d}z}
{\int_{\mathbb{R}^{N}}|u_{c}|^{p}\mathrm{d}z}
=&
\frac{
\lambda^{2}
\mu^{2-k-l(1+\gamma)}
\int_{\mathbb{R}^{N}}
|\nabla_{\gamma}Q|^{2}
\mathrm{d}z}
{\lambda^{p}
\mu^{-k-l(1+\gamma)}
\int_{\mathbb{R}^{N}}|Q|^{p}\mathrm{d}z}
=
\lambda^{2-p}
\mu^{2}
=
\frac{2N_{\gamma}(q-p)}{p[4-N_{\gamma}(q-2)]}
\end{aligned}
\end{equation*}
and
\begin{equation*}
\begin{aligned}
\frac{\int_{\mathbb{R}^{N}}
|\nabla_{\gamma}u_{c}|^{2}
\mathrm{d}z}
{\int_{\mathbb{R}^{N}}|u_{c}|^{q}\mathrm{d}z}
=&
\frac{
\lambda^{2}
\mu^{2-k-l(1+\gamma)}
\int_{\mathbb{R}^{N}}
|\nabla_{\gamma}Q|^{2}
\mathrm{d}z}
{\lambda^{q}
\mu^{-k-l(1+\gamma)}
\int_{\mathbb{R}^{N}}|Q|^{q}\mathrm{d}z}
=
\frac{
\lambda^{2-q}\mu^{2}}
{\int_{\mathbb{R}^{N}}|Q|^{q}\mathrm{d}z}
=
\frac{2N_{\gamma}(q-p)}{q[4-N_{\gamma}(p-2)]}.
\end{aligned}
\end{equation*}
Then
\begin{equation*}
\begin{aligned}
\int_{\mathbb{R}^{N}}
|u_{c}|^{q}
\mathrm{d}z
=
\frac{q[4-N_{\gamma}(p-2)]}{p[4-N_{\gamma}(q-2)]}
\int_{\mathbb{R}^{N}}
|u_{c}|^{p}
\mathrm{d}z,
\end{aligned}
\end{equation*}
which shows our desired result.
\end{proof}

\begin{lemma}\label{Lemma6.4}
Let $u_{c}$ be defined in Lemma \ref{Lemma6.2}. Then
\begin{equation*}
\begin{aligned}
\frac{1}{2}
\int_{\mathbb{R}^{N}}
|\nabla_{\gamma}u_{c}|^{2}
\mathrm{d}z
+
\frac{1}{p}
\int_{\mathbb{R}^{N}}
|u_{c}|^{p}
\mathrm{d}z
-
\frac{1}{q}
\int_{\mathbb{R}^{N}}
|u_{c}|^{q}
\mathrm{d}z
=0,
\end{aligned}
\end{equation*}
and
\begin{equation}\label{6.3}
\begin{aligned}
\int_{\mathbb{R}^{N}}|u_{c}|^{2}\mathrm{d}x
=&
C_{N_{\gamma},p,q}
^{-\frac{4-N_{\gamma}(p-2)}{2(q-p)}}
\left[
\frac{p}{q}
\frac{4-N_{\gamma}(q-2)}{4-N_{\gamma}(p-2)}
\right]
^{\frac{4-N_{\gamma}(p-2)}{2(p-q)}}
\left[
\frac{2}{p}
\frac{N_{\gamma}(q-p)}{4-N_{\gamma}(q-2)}
\right]
^{\frac{-N_{\gamma}}{2}}
\\=&c_{\gamma}.
\end{aligned}
\end{equation}
\end{lemma}

\begin{proof}
It follows from Lemma \ref{Lemma6.3} that
\begin{equation*}
\begin{aligned}
&\frac{1}{2}
\int_{\mathbb{R}^{N}}
|\nabla_{\gamma}u_{c}|^{2}
\mathrm{d}z
+
\frac{1}{p}
\int_{\mathbb{R}^{N}}
|u_{c}|^{p}
\mathrm{d}z
-
\frac{1}{q}
\int_{\mathbb{R}^{N}}
|u_{c}|^{q}
\mathrm{d}z
\\=&
\frac{N_{\gamma}(q-p)}
{p[4-N_{\gamma}(q-2)]}
\int_{\mathbb{R}^{N}}
|u_{c}|^{p}
\mathrm{d}z
+
\frac{1}{p}
\int_{\mathbb{R}^{N}}
|u_{c}|^{p}
\mathrm{d}x
-
\frac{4-N_{\gamma}(p-2)}
{p[4-N_{\gamma}(q-2)]}
\int_{\mathbb{R}^{N}}
|u_{c}|^{p}
\mathrm{d}z
=0.
\end{aligned}
\end{equation*}
From \eqref{6.2}  and $\int_{\mathbb{R}^{N}}
|\nabla_{\gamma}Q|^{2}
\mathrm{d}z=\int_{\mathbb{R}^{N}}|Q|^{p}\mathrm{d}x=1$,
we know that
\begin{equation*}
\begin{aligned}
\int_{\mathbb{R}^{N}}|u_{c}|^{2}\mathrm{d}z
=&
\lambda^{2}
\mu^{-k-l(1+\gamma)}
\int_{\mathbb{R}^{N}}
|Q|^{2}
\mathrm{d}z\\
=&
C_{N_{\gamma},p,q}
^{-\frac{4-N_{\gamma}(p-2)}{2(q-p)}}
\left[
\frac{p}{q}
\frac{4-N_{\gamma}(q-2)}{4-N_{\gamma}(p-2)}
\right]
^{\frac{4-N_{\gamma}(p-2)}{2(p-q)}}
\left[
\frac{2}{p}
\frac{N_{\gamma}(q-p)}{4-N_{\gamma}(q-2)}
\right]
^{\frac{-N_{\gamma}}{2}}.
\end{aligned}
\end{equation*}
The proof is complete.
\end{proof}

\begin{lemma}\label{Lemma6.5}
Let $u_{c}$ be a solution in Lemma \ref{Lemma6.2}.
Then $u_{c}$ is a $x$-radial action GSS of \eqref{G} for $\omega=\omega_{*}$.
\end{lemma}

\begin{proof}
Since $u_{c}$ satisfies \eqref{6.1}, it follows that
\begin{equation}\label{6.4}
\begin{aligned}
\tilde{e}_{\omega_{*}}\leqslant
I_{\omega_{*}}(u_{c}).
\end{aligned}
\end{equation}
Our goal is to prove
$I_{\omega_{*}}(u_{c})= \tilde{e}_{\omega_{*}}$.

Now let $v_{c}\in H_{\gamma,rad_{x}}^{1,2}(\mathbb{R}^N)\backslash\{0\}$ be any $x$-radial action GSS of equation \eqref{G} for $\omega=\omega_{*}$. Then by Lemma \ref{Lemma2.6}, one has, for any $t\in(0,\infty)$,
\begin{equation}\label{6.5}
\begin{aligned}
I_{\omega_{*}}(v_{c})-I_{\omega_{*}}(v_{c}^{t})\geqslant 0,
\end{aligned}
\end{equation} 
where $v_{c}^{t}(z)=t^{2}v_{c}(t^{p-2}x,t^{(p-2)(1+\gamma)}y)$. Choosing 
$t_{*}:=\left(
\frac{c_{\gamma}}{\|v_{c}\|^{2}_{L^{2}(\mathbb{R}^N)}}
\right)^{\frac{1}{4-N_{\gamma}(p-2)}}$, one finds that
\begin{equation*}
\begin{aligned}
\|v_{c}^{t_*}\|^{2}_{L^{2}(\mathbb{R}^N)}=c_{\gamma}.
\end{aligned}
\end{equation*}
Thus, we have
\begin{equation*}
\begin{aligned}
E_{\gamma}(u_c)\leqslant E_{\gamma}(v^{t_{*}}_{c}).
\end{aligned}
\end{equation*}
It follows from \eqref{6.4} and \eqref{6.5} that
\begin{equation*}
\begin{aligned}
\tilde{e}_{\omega_{*}}\leqslant
I_{\omega_{*}}(u_{c})
=&
E_{\gamma}(u_{c})+\frac{\omega_{*}c_{\gamma}}{2}
\\\leqslant&
E_{\gamma}(v^{t_{*}}_{c})+\frac{\omega_{*}}{2}\|v_{c}^{t_*}\|^{2}_{L^{2}(\mathbb{R}^N)}
\\=&
I_{\omega_{*}}(v_{c}^{t_*})
\leqslant
I_{\omega_{*}}(v_{c})
=
\tilde{e}_{\omega_{*}}.
\end{aligned}
\end{equation*}
Obviously, one has
$I_{\omega_{*}}(u_{c})=\tilde{e}_{\omega_{*}}$.
The proof is complete.
\end{proof}

\begin{proof}[Proof of Theorem \ref{Theorem1.4}]
Theorem \ref{Theorem1.4} follows by Lemmas \ref{Lemma6.1}-\ref{Lemma6.5}.
\end{proof}

\section{Proof of Theorem \ref{Theorem1.5}}
In this section,
we give the proof of Theorem \ref{Theorem1.5}.

\begin{lemma}
\label{Lemma7.1}
All $x$-radial energy GSS with critical mass $c_{\gamma}$ are optimizers of \eqref{1.2}.
\end{lemma}

\begin{proof}
Let $(u_{c_{\gamma}},\omega_{c_{\gamma}})\in H^{1,2}_{\gamma,rad_{x}}(\mathbb{R}^{N})\times \mathbb{R}$ be a couple of solution for equation \eqref{G} with critical mass $ c_{\gamma}$.
From Theorem \ref{Theorem1.3}, we know
\begin{equation*}
\begin{aligned}
0=E_{\gamma}(u_{c_{\gamma}})
=
\tilde{m}_{c_{\gamma}}
=
\inf_{u\in S_{c_{\gamma}}} E_{\gamma}(u).
\end{aligned}
\end{equation*}
Set \begin{equation*}
\begin{aligned}
u_{c_{\gamma},t}(x)=t^{\frac{N_{\gamma}}{2}}u_{c_{\gamma}}(tx,t^{1+\gamma}y).
\end{aligned}
\end{equation*}

For $2<p<q<2+\frac{4}{N_{\gamma}}$,
it follows from Lemma \ref{Lemma4.1} that 
\begin{equation*}
\begin{aligned}
\min\limits_{t>0}
E_{\gamma}(u_{c_{\gamma},t})
=\min\limits_{t>0}
h_{u_{c_{\gamma}}}(t)=0
=h_{u_{c_{\gamma}}}(t_{u_{c_{\gamma}}}),
\end{aligned}
\end{equation*}
where $t_{u_{c_{\gamma}}}$ is the unique critical point of $h_{u_{c_{\gamma}}}(\cdot)$.
From $E_{\gamma}(u_{c_{\gamma}})=0$, we further know
\begin{equation*}
\begin{aligned}
1=t_{u_{c_{\gamma}}}
=
\left\{
\frac{q[4-N_{\gamma}(p-2)]}{p[4-N_{\gamma}(q-2)]}
\frac{
\int_{\mathbb{R}^{N}}
|u_{c_{\gamma}}|^{p}
\mathrm{d}z
}
{
\int_{\mathbb{R}^{N}}
|u_{c_{\gamma}}|^{q}
\mathrm{d}z
}
\right\}
^{\frac{2}{(q-p)N_{\gamma}}},
\end{aligned}
\end{equation*}
which gives
\begin{equation}\label{7.1}
\begin{aligned}
\int_{\mathbb{R}^{N}}
|u_{c_{\gamma}}|^{p}
\mathrm{d}z
=
\frac{p[4-N_{\gamma}(q-2)]}{q[4-N_{\gamma}(p-2)]}
\int_{\mathbb{R}^{N}}
|u_{c_{\gamma}}|^{q}
\mathrm{d}z.
\end{aligned}
\end{equation}
From $E_{\gamma}(u_{c_{\gamma}})=0$ and \eqref{7.1}, 
we have
\begin{equation}\label{7.2}
\begin{aligned}
\int_{\mathbb{R}^{N}}
|\nabla_{\gamma}u_{c_{\gamma}}|^{2}
\mathrm{d}z
=&
\frac{2}{q}
\int_{\mathbb{R}^{N}}
|u_{c_{\gamma}}|^{q}
\mathrm{d}z
-
\frac{2}{p}
\int_{\mathbb{R}^{N}}
|u_{c_{\gamma}}|^{p}
\mathrm{d}z\\
=&
\frac{2}{q}
\frac{N_{\gamma}(q-p)}{4-N_{\gamma}(p-2)}
\int_{\mathbb{R}^{N}}
|u_{c_{\gamma}}|^{q}
\mathrm{d}z.
\end{aligned}
\end{equation}
Combining \eqref{7.1} and \eqref{7.2}, one has
\begin{equation*}
\begin{aligned}
&\left(
\int_{\mathbb{R}^{N}}
|u_{c_{\gamma}}|^{p}
\mathrm{d}z
\right)
^{\frac{2N_{\gamma}+4-qN_{\gamma}}{2N_{\gamma}+4-pN_{\gamma}}}
\left(
\int_{\mathbb{R}^{N}}
|\nabla_{\gamma}u_{c_{\gamma}}|^{2}
\mathrm{d}z
\right)
^{\frac{N_{\gamma}(q-p)}{2N_{\gamma}+4-pN_{\gamma}}}\\
=&
\left[
\frac{p}{q}
\frac{N_{\gamma}(q-2)-4}{N_{\gamma}(p-2)-4}
\right]
^{\frac{2N_{\gamma}+4-qN_{\gamma}}{2N_{\gamma}+4-pN_{\gamma}}}
\left[
\frac{2}{q}
\frac{(p-q)N_{\gamma}}{N_{\gamma}(p-2)-4}
\right]
^{\frac{N_{\gamma}(q-p)}{2N_{\gamma}+4-pN_{\gamma}}}
\int_{\mathbb{R}^{N}}
|u_{c_{\gamma}}|^{q}
\mathrm{d}z\\
=&
c_{\gamma}^{-\frac{2(q-p)}{2N_{\gamma}+4-pN_{\gamma}}}
C_{N_{\gamma},p,q}^{-1}
\int_{\mathbb{R}^{N}}
|u_{c_{\gamma}}|^{q}
\mathrm{d}z\\
=&
C_{N_{\gamma},p,q}^{-1}
\left[
\int_{\mathbb{R}^{N}}
|u_{c_{\gamma}}|^{2}
\mathrm{d}z
\right]^{-\frac{2(q-p)}{2N_{\gamma}+4-pN_{\gamma}}}
\int_{\mathbb{R}^{N}}
|u_{c_{\gamma}}|^{q}
\mathrm{d}z,
\end{aligned}
\end{equation*}
since 
\begin{equation*}
\begin{aligned}
c_{\gamma}^{-\frac{2(q-p)}{2N_{\gamma}+4-pN_{\gamma}}}
C_{N_{\gamma},p,q}^{-1}
=
\left[
\frac{p}{q}
\frac{N_{\gamma}(q-2)-4}{N_{\gamma}(p-2)-4}
\right]
^{\frac{2N_{\gamma}+4-qN_{\gamma}}{2N_{\gamma}+4-pN_{\gamma}}}
\left[
\frac{2}{q}
\frac{(p-q)N_{\gamma}}{N_{\gamma}(p-2)-4}
\right]
^{\frac{N_{\gamma}(q-p)}{2N_{\gamma}+4-pN_{\gamma}}}
\end{aligned}
\end{equation*} and $\|u_{c_{\gamma}}\|_{L^{2}(\mathbb{R}^{N})}^{2}=c_{\gamma}$.

The case of $2+\frac{4}{N_{\gamma}}<q<p<2_{\gamma}^{*}$ is similar,
we omit the proof.
\end{proof}

\begin{proof}
[Proof of Theorem \ref{Theorem1.5}]
(i) The claim follows by Lemma \ref{Lemma7.1}.

(ii)
The proof is same as Lemma \ref{Lemma6.5} and thus we omit it here.
\end{proof}

\section{Proof of Theorem \ref{Theorem1.6}}
\label{section8}
In this section, we prove Theorem \ref{Theorem1.6}.

\begin{lemma}
\label{Lemma8.1}
Let $N=k+l$,
$k,l\geqslant 1$,
$\gamma\geqslant0$,
$\alpha=1$
and 
$2<p<q<2+\frac{4}{N_{\gamma}}$.
Then equation \eqref{G} has a $x$-radial action GSS.  
\end{lemma}

\begin{proof}
The method is similar to \cite[Theorems 4.2 and 4.3]{Willem1996book}, so we skip the details.
\end{proof}

\begin{lemma}
\label{Lemma8.2}
Let $V\in H_{\gamma,rad_{x}}^{1,2}(\mathbb{R}^N)\backslash\{0\}$ be a $x$-radial action GSS of \eqref{G} for $\omega=\omega_{*}$. Then
\begin{enumerate}
\item [(i)]
$V$ is a $x$-radial energy GSS of \eqref{G} for the mass $c_{\gamma}$ with the Lagrange multiplier $\omega_{*}$;

\item [(ii)]
$V$ satisfies the Pohozaev normalization condition
\begin{equation}\label{8.1}
\begin{aligned}
\int_{\mathbb{R}^{N}}
|\nabla_{\gamma}V|^{2}
\mathrm{d}z
=&
\frac{2N_{\gamma}(q-p)}
{q[4-N_{\gamma}(p-2)]}
\int_{\mathbb{R}^{N}}
|V|^{q}
\mathrm{d}z,
\\
\int_{\mathbb{R}^{N}}
|V|^{p}
\mathrm{d}z
=&
\frac{p[4-N_{\gamma}(q-2)]}
{q[4-N_{\gamma}(p-2)]}
\int_{\mathbb{R}^{N}}
|V|^{q}
\mathrm{d}z;
\end{aligned}
\end{equation}

\item [(iii)]
It holds that
\begin{equation*}
\begin{aligned}
W_{N_{\gamma,p,q}}(V)=
C_{N_{\gamma},p,q}^{-1}.
\end{aligned}
\end{equation*}
\end{enumerate}
\end{lemma}

\begin{proof}
(i) Since $V$ satisfies equation \eqref{G} with $\omega=\omega_{*}$, Lemma \ref{Lemma2.6} implies that
\begin{equation}\label{8.2}
\begin{aligned}
&I_{\omega_{*}}(V)-I_{\omega_{*}}(V^{t})\\
\geqslant&
\frac{f_{1}(t)\omega_{*}}{2[2p-N_{\gamma}(p-2)]}
\int_{\mathbb{R}^{N}}
|V|^{2}
\mathrm{d}z
+
\frac{f_{2}(t)}{q[2p-N_{\gamma}(p-2)]}
\int_{\mathbb{R}^{N}}
|V|^{q}
\mathrm{d}z.
\end{aligned}
\end{equation}
Now let $w_{c}$ be any $x$-radial energy GSS of equation \eqref{G} for the mass $c_{\gamma}$, whose Lagrange multiplier coincides with 
$\omega_{*}$. Then one has
\begin{equation*}
\begin{aligned}
\tilde{e}_{\omega_{*}}\leqslant
I_{\omega_{*}}(w_{c})
\ \
\mathrm{and}
\ \
E_{\gamma}(w_{c})\leqslant E_{\gamma}(u),
\ \
\mathrm{for~all}~u\in \tilde{S}_{c_{\gamma}}.
\end{aligned}
\end{equation*}
Especially taking
$\tilde{t}_{*}:=\left(
\frac{c_{\gamma}}{\|V\|^{2}_{L^{2}(\mathbb{R}^N)}}
\right)^{\frac{1}{4-N_{\gamma}(p-2)}}$, we find that
\begin{equation}\label{8.3}
\begin{aligned}
\|V^{\tilde{t}_{*}}\|^{2}_{L^{2}(\mathbb{R}^N)}=
c_{\gamma}
\Rightarrow
E_{\gamma}(w_{c})\leqslant
E_{\gamma}(V^{\tilde{t}_{*}}).
\end{aligned}
\end{equation}
Then from \eqref{8.2} and \eqref{8.3}, it follows that
\begin{equation}\label{8.4}
\begin{aligned}
\tilde{e}_{\omega_{*}}
\leqslant
I_{\omega_{*}}(w_{c})
=&
E_{\gamma}(w_{c})+\frac{\omega_{*}c_{\gamma}}{2}\\
\leqslant&
E_{\gamma}(V^{\tilde{t}_{*}})+
\frac{\omega_{*}}{2}
\|V^{\tilde{t}_{*}}\|_{L^{2}(\mathbb{R}^N)}^{2}\\
=&
I_{\omega_{*}}(V^{\tilde{t}_{*}})
\leqslant
I_{\omega_{*}}(V)
=
\tilde{e}_{\omega_{*}}.
\end{aligned}
\end{equation}
Hence, we get 
$I_{\omega_{*}}(V^{\tilde{t}_{*}})=
I_{\omega_{*}}(V)$. Going back to \eqref{8.2}, we deduce that
\begin{equation*}
\begin{aligned}
0\geqslant
\frac{f_{1}(\tilde{t}_{*})\omega_{*}}{2[2p-N_{\gamma}(p-2)]}
\int_{\mathbb{R}^{N}}
|V|^{2}
\mathrm{d}z
+
\frac{f_{2}(\tilde{t}_{*})}{q[2p-N_{\gamma}(p-2)]}
\int_{\mathbb{R}^{N}}
|V|^{q}
\mathrm{d}z
\geqslant 0,
\end{aligned}
\end{equation*}
from which we conclude that
$\tilde{t}_{*}=1$. Then \eqref{8.4} implies that
\begin{equation}\label{8.5}
\begin{aligned}
\|V\|_{L^{2}(\mathbb{R}^N)}^{2}=c_{\gamma}
\ \
\mathrm{and}
\ \
E_{\gamma}(V)=E_{\gamma}(w_{c})=\tilde{m}_{c_{\gamma}}=0.
\end{aligned}
\end{equation}
This implies that $V$ is a $x$-radial energy GSS for the mass $c_{\gamma}$ with the Lagrange multiplier $\omega_{*}$.

\noindent
(ii) We observe from \eqref{8.5} that
\begin{equation*}
\begin{aligned}
I_{\omega_{*}}(V)=
E_{\gamma}(V)
+
\frac{\omega_{*}}{2}
\|V\|_{L^{2}(\mathbb{R}^N)}^{2}=
\frac{\omega_{*}}{2}c_{\gamma}.
\end{aligned}
\end{equation*}
Moreover, by Lemma \ref{Lemma2.5}, we also have
\begin{equation*}
\begin{aligned}
0=E_{\gamma}(V)=
\frac{1}{2}
\int_{\mathbb{R}^{N}}
|\nabla_{\gamma}V|^{2}
\mathrm{d}z
+
\frac{1}{p}
\int_{\mathbb{R}^{N}}
|V|^{p}
\mathrm{d}z
-
\frac{1}{q}
\int_{\mathbb{R}^{N}}
|V|^{q}
\mathrm{d}z
\end{aligned}
\end{equation*}
and
\begin{equation*}
\begin{aligned}
0=&\frac{N_{\gamma}}{2}N(V)-P(V)
\\=&
\int_{\mathbb{R}^{N}}
|\nabla_{\gamma}V|^{2}
\mathrm{d}z
+
\frac{N_{\gamma}(p-2)}{2p}
\int_{\mathbb{R}^{N}}
|V|^{p}
\mathrm{d}z
-
\frac{N_{\gamma}(q-2)}{2q}
\int_{\mathbb{R}^{N}}
|V|^{q}
\mathrm{d}z,
\end{aligned}
\end{equation*}
from which we obtain \eqref{8.1}.

\noindent
(iii) From \eqref{6.3} and \eqref{6.5}, we find that
\begin{equation*}
\begin{aligned}
C_{N_{\gamma},p,q}^{-1}
\leqslant&
W_{N_{\gamma,p,q}}(V)\\
=&
c_{\gamma}
^{\frac{2(q-p)}{4-N_{\gamma}(p-2)}}
\frac{\big\{
p[4-N_{\gamma}(q-2)]
\big\}
^{\frac{4-N_{\gamma}(q-2)}{4-N_{\gamma}(p-2)}}
\big[
2N_{\gamma}(q-p)
\big]
^{\frac{N_{\gamma}(q-p)}{4-N_{\gamma}(p-2)}}}
{q[4-N_{\gamma}(p-2)]}
\\=&
C_{N_{\gamma},p,q}^{-1},
\end{aligned}
\end{equation*}
and hence $W_{N_{\gamma,p,q}}(V)=
C_{N_{\gamma},p,q}^{-1}$.
This completes the proof.
\end{proof}

\begin{proof}[Proof of Theorem \ref{Theorem1.6}]
From Lemmas \ref{Lemma8.1} and \ref{Lemma8.2}, we get our desired result.
\end{proof}

\section*{Acknowledgements}
This work is supported by National Natural Science Foundation of China (12101006, 12401200), and Natural Science Foundation of Hunan Province (2025JJ60066). A.J. is member of the INDAM Research Group GNAMPA.

\section*{Conflict of interest}
The author declares that he has no known competing financial interests or personal relationships that could have appeared to influence the work reported in
this article.

\section*{Data availability statement}
No data were used for the research described in the article.

\end{document}